\documentclass[11pt]{preprint}
\usepackage[english]{babel}
\usepackage{amssymb}
\usepackage{amsmath}
\usepackage{hyperref}
\usepackage{breakurl}
\usepackage{fonttable}
\usepackage{mathtools}
\usepackage{pifont}
\usepackage{mhenvs}
\usepackage{mhequ}
\usepackage{graphicx}

\usepackage{mhsymb}
\usepackage{booktabs}
\usepackage{stmaryrd}
\usepackage{float}
\usepackage{tikz}
\usepackage{mathrsfs}
\usepackage{array}
\usepackage{times}
\usepackage{microtype}
\usepackage{mathtools}
\usepackage{comment}
\usepackage{slashed}
\usepackage{upgreek}
\usepackage{cases}
\usepackage{mathtools}
\usepackage{centernot}
\usepackage{leftidx}
\usepackage{accents}
\usepackage{arydshln}
\usepackage{footnote}
\makesavenoteenv{tabular}
\usepackage{enumitem}

\usepackage{stackrel}
\usepackage{halloweenmath}
\usepackage{longtable}
\usepackage{array}
\usepackage{cprotect}
\usepackage{xstring}
\usepackage{ulem}
\usepackage{stmaryrd}
\usepackage{empheq}
\usepackage{color}
\usepackage{framed}

\definecolor{shadecolor}{gray}{0.875}

\DeclareSymbolFont{timesoperators}{T1}{ptm}{m}{n}
\SetSymbolFont{timesoperators}{bold}{T1}{ptm}{b}{n}

\makeatletter
\renewcommand{\operator@font}{\mathgroup\symtimesoperators}
\makeatother

\colorlet{symbolsgrey}{blue!30!black!50}
\colorlet{testcolor}{green!60!black}
\definecolor{purple}{rgb}{0.55,0.05,0.8}
\definecolor{symbols}{rgb}{0.55,0.05,0.8}

\usetikzlibrary{shapes.misc}
\usetikzlibrary{shapes.symbols}
\usetikzlibrary{shapes.geometric}
\usetikzlibrary{snakes}
\usetikzlibrary{decorations}
\usetikzlibrary{decorations.markings}

\usetikzlibrary{calc}
\usetikzlibrary{external}

\let\oldskull\skull
\def\skull{\mathord{\oldskull}}

\newcommand{\cbm}{{\mcb{m}}}
\newcommand{\cbn}{{\mcb{n}}}
\newcommand{\cbB}{{\mcb{B}}}
\newcommand{\cbC}{{\mcb{C}}}

\DeclareMathAlphabet{\mathbbm}{U}{bbm}{m}{n}

\DeclareFontFamily{U}{BOONDOX-calo}{\skewchar\font=45 }
\DeclareFontShape{U}{BOONDOX-calo}{m}{n}{
  <-> s*[1.05] BOONDOX-r-calo}{}
\DeclareFontShape{U}{BOONDOX-calo}{b}{n}{
  <-> s*[1.05] BOONDOX-b-calo}{}
\DeclareMathAlphabet{\mcb}{U}{BOONDOX-calo}{m}{n}
\SetMathAlphabet{\mcb}{bold}{U}{BOONDOX-calo}{b}{n}

\setlist{noitemsep,topsep=4pt}
\newcommand{\nnorm}[1]{{\vert\kern-0.25ex\vert\kern-0.25ex\vert #1 
    \vert\kern-0.25ex\vert\kern-0.25ex\vert}}
\newcommand{\nnnorm}[1]{{\talloblong #1 \talloblong}}
\newcommand{\norm}[1]{{\| #1 
    \|}}

\makeatletter 
\newcommand*{\bigcdot}{}
\DeclareRobustCommand*{\bigcdot}{%
  \mathbin{\mathpalette\bigcdot@{}}%
}
\newcommand*{\bigcdot@scalefactor}{.5}
\newcommand*{\bigcdot@widthfactor}{1.15}
\newcommand*{\bigcdot@}[2]{%
  \sbox0{$#1\vcenter{}$}
  \sbox2{$#1\cdot\m@th$}%
  \hbox to \bigcdot@widthfactor\wd2{%
    \hfil
    \raise\ht0\hbox{%
      \scalebox{\bigcdot@scalefactor}{%
        \lower\ht0\hbox{$#1\bullet\m@th$}%
      }%
    }%
    \hfil
  }%
}
\makeatother

\def\dash{\leavevmode\unskip\kern0.18em--\penalty\exhyphenpenalty\kern0.18em}
\def\slash{\leavevmode\unskip\kern0.15em/\penalty\exhyphenpenalty\kern0.15em}

\newcommand{\Norm}[1]{{\talloblong #1 \talloblong}}

\colorlet{darkblue}{blue!90!black}
\colorlet{darkgreen}{green!82!black}
\colorlet{darkyellow}{yellow!65!red}
\colorlet{darkred}{red!80!black}
\let\comm\comment

\def\auth#1{}

\newcommand{\stau}{{\boldsymbol{\tau}}}
\newcommand{\ssigma}{{\boldsymbol{\sigma}}}
\newcommand{\slambda}{{\boldsymbol{\lambda}}}

\newcommand{\mfT}{\mathfrak{T}}

\newcommand{\mfo}{{\mathfrak{o}}}

\newcommand{\mfL}{\mathfrak{L}}

\newcommand{\mfc}{\mathfrak{c}}

\newcommand{\mfl}{\mathfrak{l}}
\newcommand{\mfz}{\mathfrak{z}}

\newcommand{\mcE}{\mathcal{E}}
\newcommand{\mcH}{\mathcal{H}}

\newcommand{\mcR}{\mathcal{R}}
\newcommand{\mcC}{\mathcal{C}}

\newcommand{\mcS}{\mathcal{S}}

\newcommand{\mcL}{\mathcal{L}}
\newcommand{\mcT}{\mathcal{T}}

\newcommand{\mcN}{\mathcal{N}}

\newcommand{\mcD}{\mathcal{D}}
\newcommand{\mcW}{\mathcal{W}}
\newcommand{\mcP}{\mathcal{P}}
\newcommand{\mcG}{\mathcal{G}}

\newcommand{\mcQ}{\mathcal{Q}}

\newcommand{\bff}{\mathbf{f}}

\newcommand{\bfL}{\mathbf{L}}

\newcommand{\bell}{\boldsymbol{\ell}}

\newcommand{\Aut}{\mathrm{Aut}}
\newcommand{\Iso}{\mathrm{Iso}}

\newcommand{\tiI}{\Tilde{I}}

\newcommand{\mfs}{\mathfrak{s}}

\newcommand{\D}{\mathrm{D}}
\newcommand{\dd}{\mathrm{d}}
\newcommand{\Id}{\mathrm{Id}}

\newcommand{\mVec}{\mathrm{Vec}}

\newcommand{\epmu}{{\eps,\mu}}
\newcommand{\epnu}{{\eps,\nu}}

\newcommand{\A}{\mathrm{A}}
\newcommand{\bfX}{\mathbf{X}}

\newcommand{\Pol}{\mathrm{Pol}}
\newcommand{\Cov}{\mathrm{Cov}}

\newcommand{\btau}{\underline{\tau}}
\newcommand{\bsigma}{\underline{\sigma}}
\newcommand{\bT}{\boldsymbol{\mathfrak{T}}}
\newcommand{\bUpsilon}{\underline{\Upsilon}}

\newcommand{\Ind}{\text{$\mathrm{Ind}$}}

\DeclareFontFamily{U} {MnSymbolA}{}
\DeclareFontShape{U}{MnSymbolA}{m}{n}{
  <-6> MnSymbolA5
  <6-7> MnSymbolA6
  <7-8> MnSymbolA7
  <8-9> MnSymbolA8
  <9-10> MnSymbolA9
  <10-12> MnSymbolA10
  <12-> MnSymbolA12}{}
\DeclareFontShape{U}{MnSymbolA}{b}{n}{
  <-6> MnSymbolA-Bold5
  <6-7> MnSymbolA-Bold6
  <7-8> MnSymbolA-Bold7
  <8-9> MnSymbolA-Bold8
  <9-10> MnSymbolA-Bold9
  <10-12> MnSymbolA-Bold10
  <12-> MnSymbolA-Bold12}{}

\DeclareSymbolFont{MnSyA} {U} {MnSymbolA}{m}{n}

\DeclareMathSymbol{\graft}{\mathrel}{MnSyA}{184}

\usetikzlibrary{calc}
\usetikzlibrary{shapes.misc}
\usetikzlibrary{shapes.symbols}
\usetikzlibrary{shapes.geometric}
\usetikzlibrary{snakes}
\usetikzlibrary{decorations}
\usetikzlibrary{decorations.markings}

\tikzset{
	root/.style={circle,fill=testcolor,inner sep=0pt, minimum size=2mm},
	dot/.style={circle,fill=symbols,draw=symbols,inner sep=0pt,minimum size=0.5mm},
	bdot/.style={circle,fill=symbols,draw=symbols,inner sep=0pt,minimum size=1mm},
	bdotsml/.style={circle,fill=symbols,draw=symbols,inner sep=0pt,minimum size=0.75mm},
	square/.style={regular polygon,regular polygon sides=4,fill=black,draw=black,inner sep=0pt,minimum size=1.2mm},
	wsquare/.style={regular polygon,regular polygon sides=4,fill=white,draw=black, inner sep=0pt,minimum size=1.2mm},
	squaresml/.style={regular polygon,regular polygon sides=4,fill=black,draw=black,inner sep=0pt,minimum size=0.9mm},
	wsquaresml/.style={regular polygon,regular polygon sides=4,fill=white,draw=black, inner sep=0pt,minimum size=0.9mm},
	eps/.style={circle,fill=white,draw=symbols,inner sep=0pt,minimum size=1mm},
	int/.style={circle,fill=black,draw=black,inner sep=0pt,minimum size=0.7mm},
	var/.style={circle,fill=black!10,draw=black,inner sep=0pt, minimum size=2mm},
	dotred/.style={circle,fill=black!50,inner sep=0pt, minimum size=2mm},
	generic/.style={semithick,shorten >=1pt,shorten <=1pt},
	dist/.style={ultra thick,draw=testcolor,shorten >=1pt,shorten <=1pt},
	testfcn/.style={ultra thick,testcolor,shorten >=1pt,shorten <=1pt,<-},
	testfcnx/.style={ultra thick,testcolor,shorten >=1pt,shorten <=1pt,<-,
		postaction={decorate,decoration={markings,mark=at position 0.6 with {\drawx}}}},
	keps/.style={semithick,shorten >=1pt,shorten <=1pt,densely dashed,->},
	kprimex/.style={semithick,shorten >=1pt,shorten <=1pt,densely dashed,->,
		postaction={decorate,decoration={markings,mark=at position 0.4 with {\drawx}}}},
	kernel/.style={semithick,shorten >=1pt,shorten <=1pt,->},
	multx/.style={shorten >=1pt,shorten <=1pt,
		postaction={decorate,decoration={markings,mark=at position 0.5 with {\drawx}}}},
	kernelx/.style={semithick,shorten >=1pt,shorten <=1pt,->,
		postaction={decorate,decoration={markings,mark=at position 0.4 with {\drawx}}}},
	kernel1/.style={->,semithick,shorten >=1pt,shorten <=1pt,postaction={decorate,decoration={markings,mark=at position 0.45 with {\draw[-] (0,-0.1) -- (0,0.1);}}}},
	kernel2/.style={->,semithick,shorten >=1pt,shorten <=1pt,postaction={decorate,decoration={markings,mark=at position 0.45 with {\draw[-] (0.05,-0.1) -- (0.05,0.1);\draw[-] (-0.05,-0.1) -- (-0.05,0.1);}}}},
	kernelBig/.style={semithick,shorten >=1pt,shorten <=1pt,decorate, decoration={zigzag,amplitude=1.5pt,segment length = 3pt,pre length=2pt,post length=2pt}},
	rho/.style={dotted,semithick,shorten >=1pt,shorten <=1pt},
	renorm/.style={shape=circle,fill=white,inner sep=1pt},
	labl/.style={shape=rectangle,fill=white,inner sep=1pt},
	xi/.style={circle,fill=symbols!10,draw=symbols,inner sep=0pt,minimum size=1.2mm},
	xix/.style={crosscircle,fill=symbols!10,draw=symbols,inner sep=0pt,minimum size=1.2mm},
	xib/.style={circle,fill=symbols!10,draw=symbols,inner sep=0pt,minimum size=1.6mm},
	xibx/.style={crosscircle,fill=symbols!10,draw=symbols,inner sep=0pt,minimum size=1.6mm},
	not/.style={circle,fill=symbols,draw=symbols,inner sep=0pt,minimum size=0.5mm},
cumu2n/.style={inner sep=3pt},
cumu2/.style={draw=red!80,fill=red!40},
cumu2b/.style={draw=blue!80,fill=blue!40},
cumu2nv/.style={inner sep=3pt},
cumu2v/.style={draw=red!80,fill=white,very thick},
cumu3/.style={regular polygon, regular polygon sides=3,draw=red!80,rounded corners=3pt,fill=red!40,minimum size=5mm},
cumu4/.style={regular polygon, regular polygon sides=4,draw=red!80,rounded corners=3pt,fill=red!40,minimum size=7mm},
cumu5/.style={regular polygon, regular polygon sides=5,draw=red!80,rounded corners=3pt,fill=red!40,minimum size=7mm},
	>=stealth,
	}

\newcommand{\rmd}{\mathrm{d}}
\newcommand{\rmI}{\mathrm{I}}

\newcommand{\tF}{\tilde{F}}
\newcommand{\tiR}{\tilde{R}}
\newcommand{\ttu}{{\tt u}}

\newcommand{\Kappa}{\text{{\Large $\kappa$}}}

\newcommand{\X}{\mathbf{X}}

\DeclareFontFamily{U}{mathx}{}
\DeclareFontShape{U}{mathx}{m}{n}{<-> mathx10}{}
\DeclareSymbolFont{mathx}{U}{mathx}{m}{n}
\DeclareMathAccent{\widecheck}{0}{mathx}{"71}

\makeatletter

\DeclareRobustCommand{\TitleEquation}[2]{\texorpdfstring{\StrLeft{\f@series}{1}[\@firstchar]$\if%
b\@firstchar\boldsymbol{#1}\else#1\fi$}{#2}}

\makeatother

\newcommand{\noise}{\raisebox{0ex}{\includegraphics[scale=1.2]{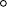}}}

\newcommand{\bignoise}{\raisebox{+0.05ex}{\includegraphics[scale=1.5]{figures/Xi.pdf}}}
\newcommand{\tOne}{\raisebox{-0.5ex}{\includegraphics[scale=1.2]{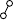}}}
\newcommand{\tOnesmall}{\raisebox{-0.5ex}{\includegraphics[scale=0.8]{figures/XiX.pdf}}}

\newcommand{\tTwo}{\raisebox{-1ex}{\includegraphics[scale=1.2]{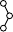}}}
\newcommand{\tTwosmall}{\raisebox{-0.5ex}{\includegraphics[scale=0.8]{figures/Tree1.pdf}}}
\newcommand{\tThree}{\raisebox{-.5ex}{\includegraphics[scale=1.2]{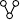}}}
\newcommand{\tThreesmall}{\raisebox{-0.5ex}{\includegraphics[scale=0.8]{figures/Tree2.pdf}}}

\usepackage[margin=1.15in]{geometry}

\begin{document}

\title{Rough differential equations in the flow approach}

\author{Ajay Chandra$^{1}$, L\'{e}onard Ferdinand$^{2}$}

\institute{Imperial College London, UK 
\and Max–Planck Institute for Mathematics in the Sciences, Leipzig, Germany\\[1em]
\email{a.chandra@ic.ac.uk, lferdinand@mis.mpg.de}}

\maketitle

\begin{abstract}
We show how the flow approach of Duch \cite{Duch21}, with elementary differentials as coordinates as in \cite{gKPZFlow}, can be used to prove well-posedness for rough stochastic differential equations driven by fractional Brownian motion with Hurst index $H > \frac{1}{4}$. 
A novelty appearing here is that we use coordinates for the flow that are indexed by trees rather than multi-indices. 
\end{abstract}

\setcounter{tocdepth}{2}
\tableofcontents

\section{Introduction}
We revisit local and global well-posedness for the finite dimensional rough differential equation 
\begin{equs}\label{eq:eq0}   
\mathrm{d}u=V\big(u(t)\big)\mathrm{d}W(t)\,,\;u(0)={\tt u}\in\R^{\cbn}
\end{equs}
where $u$ takes values in $\R^{\cbn}$, $W$ is an $\R^\cbm$-valued fractional Brownian motion with Hurst index $H \in \big(\frac{1}{4},\frac{1}{2}\big]$, and $V$ is a sufficiently regular map $V: \R^{\cbn} \rightarrow \mcL(\R^{\cbm},\R^{\cbn})$ where $\mcL(\R^\cbm,\R^{\cbn})$ is the space of linear maps from $\R^{\cbm}$ to $\R^{\cbn}$. 

In particular, we take $W:[0,\infty) \rightarrow \R^{\cbm}$ to be a Gaussian process with covariance 
\begin{equs}\label{eq:fracBrown}
\E \big[ \scal{W(t),e}_{\R^{\cbm}} \scal{W(s),e'}_{\R^{\cbm}}  \big]
=
\scal{e,e'}_{\R^{\cbm}} \Cov^+_W(t,s)\,,
\end{equs}
for any $s,t \in \R_{\geqslant 0} \eqdef [0,\infty)$ and $e,e' \in \R^{\cbm}$, where
\begin{equs}    \Cov^+_W(t,s)\eqdef C_H\big(t^{2H}+s^{2H}-|t-s|^{2H}\big)\,.
\end{equs}
Here $C_{1/2}=\frac12$ and for $H\neq 1/2$ we choose $C_H \eqdef\frac{1}{2H(2H-1)} $. 

We write $W(t)=\int_0^t\xi(s)\rmd s$ where $\xi$ is a stationary noise on $\R$ with covariance
\begin{equs}\label{eq:fracBrown2}
  \E \big[ \scal{\xi,e\otimes\varphi} \scal{\xi,e'\otimes\psi} \big] =\langle e,e'\rangle_{\R^{\cbm}}\int_{\R^2} |s-r|^{2H}\frac{\rmd \varphi}{\rmd t}(s)\frac{\rmd\psi}{\rmd t}(r)\rmd s\rmd r
\end{equs} 
for any $\varphi,\psi\in\mcD(\R)$. 
Above and in what follows, the notation $\scal{ \bigcdot, \bigcdot}$ without a subscript denotes the $L^{2}(\R; \mathcal{V})$ inner product, where $\mathcal{V} = \R^{\cbm}$ above but more generally $\mathcal{V}$ is some inner product space that can be inferred from context and we have 
\begin{equation}\label{def:innerproduct}
\scal{f,g} \equiv \int_{\R} \scal{f(t),g(t)}_{\mathcal{V}}\ \rmd t\;.
\end{equation}
In the sequel, we denote by $\Cov_W^{\tt S}(t,s)\eqdef -C_H|t-s|^{2H}$ the stationary part of the covariance of $W$, and by
\begin{equs}\label{eq:defCov}
   \Cov^{\tt S}(t,s)\equiv \Cov^{\tt S}(t-s)\eqdef\d_t\d_s\Cov^{\tt S}_W(t,s)
\end{equs}
the Schwarz kernel of the covariance of the stationary noise. We also denote by $\Cov^+(t,s)\eqdef\d_t\d_s\Cov^+_W(t,s)$ the Schwarz kernel of the covariance of the non-stationary noise. When there is no ambiguity, we always refer to the stationary covariance and drop the superscript $\tt S$, just writing $\Cov$ instead of $\Cov^{\tt S}$. For $H = \frac{1}{2}$, we have $\Cov = \delta_{\R}$ where $\delta_{\R}$ is the Dirac delta distribution on $\R$. Additionally, for $H < \frac{1}{2}$, the covariances are not locally integrable. However, the key property we will use is that they are homogeneous of degree $-2+2H$, in the sense that for any $\varphi,\psi\in\mcD(\R)$, setting, for $\mu\in(0,1]$, $\mcS_\mu\rho(t)\eqdef\mu^{-1}\rho(t/\mu)$, we have for every $\mu\in(0,1]$
\begin{equs}\label{eq:homogeneity}    \Big( \Cov^{\tt S/+},\mcS_\mu\varphi\otimes\mcS_\mu\psi \Big) =
\mu^{-2+2H} \Big( \Cov^{\tt S/+},\varphi\otimes\psi \Big)\,,
\end{equs}
where the brackets above denote the pairing of $\mcD(\R)^{\otimes2}$ with $\mcD'(\R)^{\otimes2}$, and the superscript $\tt S/+$ means that the property holds for both $\Cov^{\tt S}$ and $\Cov^+$.

The equation \eqref{eq:eq0} is formal: for $H \leqslant 1$, $\xi$ is a distribution rather than function over time. 
At a deterministic level, the equation \eqref{eq:eq0} is  ill-posed once $H \leqslant \frac{1}{2}$ -- we make this precise.  
Denote by $\mcC^{\alpha} (\R)$ the local H\"{o}lder-Besov space\footnote{See Section~\ref{subsec:notation}.} on $\R$ with exponent $\alpha\in\R$. 
We have $W \in \mcC^{H-\kappa}(\R)$ for every $\kappa > 0$, which gives $\xi 
\in \mcC^{-1+H-\kappa}(\R)$ due to the derivative. 
Since an integral improves regularity by $1$, one can hope for at best $H-\kappa$ regularity for $u$ and also $V(u)$, but since $H\leqslant\frac12\Rightarrow H + (-1 + H)  \leqslant 0$ we cannot canonically define $V(u)\xi$.

\subsection{Main results}
Our main result is the probabilistic well-posedness of \eqref{eq:eq0}. It is obtained by showing convergence of the solution $u_\eps$ to \eqref{eq:eq1}, a regularised version of \eqref{eq:eq0} which is introduced as follows.
\eqref{eq:eq0} can be formally rewritten as the integral equation  
\begin{equs}    u(t)=G\big(\1_{\geqslant0}Z[u]\big)(t)+{\tt u}\,,\quad  Z[u](t)\eqdef V\big(u(t)\big)\xi(t)\,. 
\end{equs} 
Here, $\1_{\geqslant0}$ denotes the indicator function of $\R_{\geqslant0} \subset \R$, and $G$ is the operator given by convolution with $\1_{\geqslant0}$. $G$ is therefore the ``Green's function'' for $\frac{\mathrm{d}}{\mathrm{d}t}$. 


We then introduce a regularisation scale $\eps\in(0,1]$ and study the regularised equation
\begin{equs}\label{eq:eq1}    
u_\eps(t)
= G\big(\1_{\geqslant0}
Z_\eps[u_\eps]\big)(t)+{\tt u}\,,\quad Z_\eps[u_\eps](t)\eqdef V\big(u_\eps(t)\big)\xi_{\eps}(t) 
\,,
\end{equs} 
where $\xi_{\eps}\eqdef\rho_\eps*\xi$, $\rho_{\eps}(t) \eqdef \mcS_\eps\rho(t)$ and $\rho : \R \rightarrow \R_{ \geqslant 0}$ is a smooth and compactly supported function that integrates to $1$.

\begin{theorem}\label{thm:main} 
Fix any $\kappa>0$, $H\in(1/4,1/2]$, and $V\in C^9\big( \R^{\cbn} , \mcL(\R^{\cbm},\R^{\cbn})\big)$. 
Then there exists a random variable $0 < \tilde T \leqslant 1$ such that, for any deterministic $T \in (0,1]$, the following holds on the event $\{T \leqslant \tilde T \}$: the equation \eqref{eq:eq1} is well posed on $\CC^{H -\kappa }([0,{T}])$  with solution $u_{\eps}$ which converges in probability to a limit $u$ in $\CC^{H-\kappa}( [0,{T}] )$ as $\eps \downarrow 0$.
\end{theorem}

\begin{remark}\label{rem:no_counterterm} 
Our main theorem gives convergence of local solutions to regularised equations \eqref{eq:eq1} without any renormalisation counterterm.
In our flow approach analysis, the solution is built from a collection of ``force coefficients" indexed by a class of trees (see Definition~\ref{def:trees}) and solving a truncated Polchinski equation (see Definition~\ref{def:proj}). We construct two such solutions, the stationary and non-stationary force coefficients, that have different initial conditions specified at the beginning of Section~\ref{sec:2_4}. 

We encounter a tree $\tOne$ with associated stationary force coefficient $\xi^{\tOnesmall}_{\eps,1}$ whose expectation might diverge like $\eps^{-1+2H}$, and thus require a renormalisation counterterm. 
However, we show in Lemma~\ref{lem:CT} that the potential counterterm is actually finite. 

Moreover (even though it is not needed for our analysis), in Lemma~\ref{lem:CT}, we also compute explicitly the counterterm for the corresponding non-stationary force coefficient $\zeta^{\tOnesmall}_{\eps,1}$, which is the one that appears on the RHS of the equation we use to construct the solution.
It turns out that this contribution is indeed finite in the regularisation parameter $\eps>0$, but not in time. Indeed the expectation of $\mfL\zeta^{\tOnesmall}_{\eps,1}(t)$ verifies 
\begin{equs}\label{eq:estimateCTintro}
       |\E[\mfL\zeta^{\tOnesmall}_{\eps,1}(t)]|\lesssim (t\vee \eps)^{-1+2H}\,.
\end{equs}
Here, $\mfL$ denotes a time-localisation operator, since in the flow approach divergences are first encountered in a time non-local form. 

The estimate \eqref{eq:estimateCTintro} shows that the time dependent counterterm blows up for small times. Of course, as expected, the divergence is integrable, so that this counterterm would simply boil down to a finite Itô-Stratonovich correction of the form $\rmd t^{2H}$ (which we do not need to add to the RHS of the equation).
 
In summary, no divergent renormalisation is needed in order to define the force coefficients associated with $\tOne$, in agreement with the rough path approach to fBm. 
\end{remark}

Additionally, we show that when $V$ and its first nine derivatives are globally bounded, then Theorem~\ref{thm:main} can be extended to pathwise global well-posedness.
\begin{theorem}\label{thm:thm2}
    Assume that
\begin{equs}\label{eq:assumpV}
\norm{V}_{C^9(\R^{\cbn} , \mcL(\R^{\cbm},\R^{\cbn}))}<\infty\,.
\end{equs}
Then, there exists $a,b,c, C \geqslant 0$ such that the following holds.
For every $\mcb N\geqslant1$, there exists an event\footnote{See Lemma~\ref{lem:Thm21precise} for a precise definition of $A_{\mcb N}$ in terms of the enhanced data.} $A_{\mcb N}$ with $\P(A_{\mcb N})\geqslant 1- e^{-a\mcb N^{b}}$ such that, for any initial condition $\ttu\in\R^\cbn$, \eqref{eq:eq0} is globally well-posed on $A_{\mcb N}$ and the solution $u$ to \eqref{eq:eq0} with initial condition $\ttu$ satisfies the estimate
\begin{equs}
    \sup_{t\in[0,1]}|u(t)|
     &\leqslant|\ttu|
     +C\mcb N^c\,.
\end{equs}
\end{theorem}
\begin{remark}\label{rem:1.3}
Our assumption on the regularity of $V$ is not optimal. 
The ansatz we use for the ``force'' which describes the right hand side of \eqref{eq:eq0}  (see \eqref{eq:defFmu}, \eqref{eq:trees}, and Definition~\ref{def:elem_diff1}), requires $V$ be $C^3$. 
However, when performing the probabilistic construction of the force coefficients $\zeta^\tau_\epmu$, and in particular the localisation procedure, the Kolmogorov  argument, and the construction of the non-stationary force, we need several Sobolev embedding type estimates~\eqref{eq:KmuLp}. 
This makes it necessary to control the force coefficients using the kernel $Q_\mu^{*6}$ introduced in Definition~\ref{def:defQmu}. 
The adjoint of the inverse of this kernel ends up being applied to $V$, which is why we need $V$ to be of regularity $C^9$.
Small technical improvements could likely allow us to use a weaker assumption than $V \in C^{9}$, but it seems unclear if these improvements would get us all the way to $V \in C^{3}$. 
\end{remark}

\begin{remark}\label{rem:rem1_4}
As mentioned earlier, \eqref{eq:eq1} is singular for $H\leqslant 1/2$.
In the language of singular SPDE/QFT it is subcritical/super-renormalisable for $H>0$ -- as explained above, $V(u)\xi$ is associated a power counting $-1+2H-\kappa$ which is better than the power counting $-1+H-\kappa$ of the noise $\xi$ as long as $H>0$.
The requirement that $H>1/4$ comes from the fact that below this threshold, one has $-1+2H \leqslant -1/2$, and the covariance of $V(u)\xi$ could have a $2(-1+2H) \leqslant -1$ blow-up which is not locally integrable in one dimension of time. 
This problem is familiar in rough paths, and also has analogues for singular partial differential equations - see \cite{Hairer24} where this regime is investigated. 

Regarding the case where $H\in(0,1/4]$, while the problem of the covariance blow-up would lead to a failure of the probabilistic argument performed in Section~\ref{sec:proba}, the rest of the construction of the solution only requires subcriticality, and would therefore still work all the way down to $H=0$, see Remark~\ref{rem:2_31}.
\end{remark}


Of course, we make no claim to novelty regarding the results stated in Theorem~\ref{thm:main} and Theorem~\ref{thm:thm2}, rather we use  simpler setting of \eqref{eq:eq0} to present the flow method.

We build a path-wise solution theory by following the lines Duch's implementation of the Polchinski flow \cite{Pol84} for local solution theory for parabolic SPDE \cite{Duch21}, along with a remainder argument\footnote{This remainder argument is not crucial for our either our local or global well-posedness results, but we believe it gives a more steamlined proof.} developed in \cite{GR23} and a family of flow coordinates introduced in \cite{gKPZFlow}. 
In Section~\ref{subsec:flow_approach} we briefly review the flow approach in our context. 

Stochastic differential equations driven by fractional Brownian motion have been heavily studied using a variety of methods, we point to \cite{fBMBook} which gives a good overview. 
An approach to \eqref{eq:eq0} that is the close to the path-wise solution obtained via the flow method is Lyons' theory of rough paths \cite{Lyons}, where equations driven by rough drivers can be made sense of as long as one can lift the driver to richer ``enhanced driver''. 
The construction of this enhancement for fBM with $H > \frac{1}{4}$ was first carried out in \cite{CQ02}, see also \cite{MR2667703} for more results on rough paths over fBM and rough equations driven by them. 
We note that rough path methods only require the condition $V \in C^{3}$ for both local and global well-posedness, which seems to be a relative drawback of the flow approach - see Remark~\ref{rem:1.3}. 
On the other hand, an advantage of the flow approach is that it is a unified approach to two different steps, analytic and probabilistic, that are separated in the rough path approach.  

Previous results on global in time existence for rough differential can be found in \cite[Thm~10.57]{FV10} and \cite[Thm~8.3]{FrizHairer}. 

\subsection*{Acknowledgements}
{\small
AC gratefully acknowledges partial support by the EPSRC through EP/S023925/1 through the ``Mathematics of Random Systems'' CDT EP/S023925/1. AC and LC thank the referees for helpful comments, including one that streamlined our approach to proving global in time well-posedness.}

\subsection{Review of the Polchinski flow}\label{subsec:flow_approach}

We first sketch some details of the flow approach under the assumption of vanishing initial data, that is ${\tt u} = 0 $. 
In this case the integral equation \eqref{eq:eq1} can be rewritten as  
\begin{equs}\label{eq:eqDef1_0}    u_{\eps}(t)&=G\big(F_\eps[u_\eps]\big)(t)\,,
\end{equs}
where we have set $F_\eps[\bigcdot]= \1_{\geqslant0}Z_{\eps}[\bigcdot]$.

The next key step is introducing a second scale, called an ``effective scale'', denoted by $\mu \in [0,1]$, along with associated cut-off Green's functions $(G_{\mu})_{\mu \in [0,1]}$, where $G_{\mu}$ suppresses time scales $\lesssim$ $\mu $. 
\begin{definition}\label{def:G_mu}
Fix a smooth, increasing function $\chi:\R\rightarrow[0,1]$ 
such that $\chi|_{(-\infty,1]}=0$ and $\chi|_{[2,\infty)}=1$.
For any $\mu\in[0,1]$ and $t\geqslant0$, 
we define the regularized kernel $G_\mu$, which is given by
   \begin{equs}\label{eq:DefGmu}
       G_\mu(t)\eqdef \chi(t/\mu) \,, \end{equs} 
with the convention that $G_0=G$. 
Finally, we write $\dot G_\mu(t)\eqdef\d_\mu G_\mu(t)$.
\end{definition}
\begin{remark}
Key to our analysis are the following support properties of $G_\mu$ and $\dot G_\mu$:
    \begin{equs}\label{eq:suppGmu}        \text{supp}\,G_\mu\subset[\mu,\infty)\,,\;\text{and}\;\text{supp}\,\dot G_\mu \subset[\mu,2\mu]\,.
    \end{equs}
In particular, we have $G_1=0$ on $[0,1]$.
\end{remark}
We use the kernels introduced in Definition~\ref{def:G_mu} to define an additional regularization of $u_{\eps}$ at effective scale $\mu$, which we denote by $u_{\epmu}$ and is defined by 
\begin{equs}\label{eq:eqDef2_0}    u_{\epmu}(t)&=G_{\mu}\big(F_{\eps}[u_\eps]\big)(t)\;.
\end{equs}
We obtain control over the $\eps \downarrow 0$ limit of $u_{\eps}$ over small times by getting, for any sufficiently small $\mu$, \textit{uniform in} $\eps$ control over $u_{\epmu}$. 
The RG approach obtains this control by deriving a closed, ``effective'' equation for $u_{\epmu}$ and studying how this equation evolves as $\mu$ is increased. 
The ``effective equation'' in our context is given by
\begin{equs}\label{eq:eqDefreg2}
u_{\epmu} (t)&= G_{\mu}\big( F_{\eps,\mu}[u_{\epmu}]\big)(t)\;,
\end{equs}
where we enforce that 
\begin{equ}\label{eq:force_identity}
F_{\eps,\mu}[u_{\epmu}]
=
F_{\eps}[u_{\eps}]\;.
\end{equ}
The functionals $F_{\epmu}$ or $F_{\eps}$, which have both as input and output functions of time $u: \R \rightarrow \R^{\cbn}$, will be called ``forces''.
While the force $F_{\eps}$ is local in time (in that $F_{\eps}[u](t)$ is determined by $u(t)$, many of the forces we encounter will not be local in time. 

Taking the derivative in $\mu$ of \eqref{eq:force_identity} gives the following continuous dynamic in the space of (random) forces: 
\begin{equs}\label{eq:polch}
\partial_{\mu} F_{\eps,\mu}[\bigcdot]
&=
- \scal{\D F_{\eps,\mu}[\bigcdot], \dot{G}_{\mu} F_{\eps,\mu}[\bigcdot]}\,,\\
F_{\eps,0}[\bigcdot] &= F_{\eps}[\bigcdot]\;.
\end{equs}
Above, the notation $\D$ refers to the functional derivative of the force , for a force $F$ and $f,g: \R \rightarrow \R^{\cbn}$ we have $\scal{\D F[f],g}(t) = \frac{\rmd}{\rmd\lambda}F[f + \lambda g](t)\Big|_{\lambda = 0}$. 

The idea of using a continuous change in scale in RG analysis first appeared in \cite{Pol84} in Quantum Field Theory, and the dynamic \eqref{eq:polch} is called a ``Polchinski flow''. 
The article \cite{Kupiainen2016} introduced a discretized analog of this argument for singular SPDE, and then \cite{Duch21} developed a continuous Polchinski flow approach to singular SPDE.

In the next section we incorporate the initial data {\tt u}, and also describe how we exchange the infinite dimensional flow \eqref{eq:polch} for (i) a finite dimensional approximate flow which captures the rough parts of the evolving force along with (ii) an inexplicit, evolving remainder which solves a fixed point problem in a space of fairly regular functions of time.

\subsection{Initial data and remainder flow}\label{sec:flow}

We incorporate the initial data ${\tt u}$ into our integral equation by introducing $ v_{\eps}(t)\eqdef  u_{\eps}(t)-{\tt u}$, so that setting $S_\eps[\bigcdot]\eqdef Z_\eps[\bigcdot+\tt u]$ and $F_\eps[\bigcdot]\eqdef\1_{\geqslant0}S_\eps[\bigcdot]$ we rewrite \eqref{eq:eq1} as 
\begin{equ}
 v_{\eps}(t)=G \big(F_{\eps}[v_\eps]\big)(t)\,.
\end{equ}

For fixed $\eps \in (0,1]$ and a trajectory of ``effective forces'' $\big(F_{\eps,\mu}[\bigcdot] : \mu\in(0,1] \big)$ we define a trajectory of remainder stochastic processes $R_{\eps,\mu}$ by requiring that for every $\mu \in (0,1]$ and $t \in [0,1]$,  
\begin{equs}\label{eq:FmuRmu}
  F_\eps[v_\eps](t)= F_{\eps,\mu}[v_{\eps,\mu}](t)+R_{\eps,\mu}(t)\,,
\end{equs}
where $v_\epmu$ is defined as
\begin{equs}   v_{\eps,\mu}(t)&\eqdef G_\mu \big(F_\eps[v_\eps]\big)(t)\label{eq:psimu}\,.
\end{equs}
Note that \eqref{eq:FmuRmu} and \eqref{eq:psimu} allow us to write 
\begin{equs}\label{eq:phi_mu}
    v_\epmu(t)=G_\mu\big( F_\epmu[v_\epmu]+R_\epmu\big)(t)\,.
\end{equs}
As described in Section~\ref{sec:2_1_2}, we will in practice choose $F_{\eps,\mu}$ to be the solution to an approximate flow equation which is written below as \eqref{eq:Pol} with initial condition 
\begin{equs}\label{eq:F_0}
    F_{\eps,0}[\bigcdot]=F_\eps[\bigcdot]=\1_{\geqslant0}S_\eps[\bigcdot]=\1_{\geqslant0}Z_\eps[\bigcdot+{\tt u}]\,.
\end{equs}
$F_{\eps,\mu}$ will be an explicit polynomial enhancement of the noise and will take values in spaces of stochastic processes as rough as the noise, while $R_{\eps,\mu}$ will be an inexplicit random \textit{remainder} which we will aim to show belongs to $L^\infty([0,1])$ uniformly in $\eps \in (0,1]$. Before constructing $F_\epmu$, we will first construct a stationary effective force $S_\epmu$, which solves the same equation as $F_\epmu$ but with initial condition $S_\eps$. 

Note that since $F_{\eps,0}=F_\eps$, we necessarily have $R_{\eps,0}=0$. Moreover, since by definition $F_\eps[\bigcdot]$ is supported on positive times, we have the freedom to enforce that so are $F_\epmu[\bigcdot]$ and $R_\epmu$.

\begin{remark}
  Plugging the definition \eqref{eq:F_0} of $F_\eps$ into the definition \eqref{eq:psimu} of $v_\epmu$ and using the support property \eqref{eq:suppGmu} of $G_\mu$ shows that for every $\mu\geqslant0$, $v_\epmu$ is supported after time $\mu$. In particular, this implies that for any $T\in(0,1]$, one has for every $t\in[0,T]$
  \begin{equs}
      v_\epmu(t)=(G_\mu-G_{T})F_\eps[v_\eps](t)=-\int_\mu^{T}\dot G_\nu\big(F_\epnu[v_\epnu]+R_\epnu\big)(t)\rmd\nu\,.
  \end{equs}
\end{remark}
With all of this in place, we can formulate a dynamic in $\mu$ for $R_{\eps,\mu}$ that preserves the relationship \eqref{eq:FmuRmu} and allows us to solve for it scale by scale in $\mu$.  
Observing that
\begin{equ}
   0= \frac{\rmd}{\rmd\mu} F_\eps[v_\eps]=  \frac{\rmd}{\rmd\mu}\Big(F_{\eps,\mu}[v_{\eps,\mu}]+R_{\eps,\mu}\Big)=\d_\mu F_{\eps,\mu}[v_{\eps,\mu}]+\D F_{\eps,\mu}[v_{\eps,\mu},\d_\mu v_{\eps,\mu}]+\d_\mu R_{\eps,\mu}\,,
\end{equ}
and combining the latter with $\d_\mu v_{\eps,\mu}=\dot G_\mu\big(F_{\epmu}[v_\epmu]+R_{\eps,\mu}\big)$, which is just the  application of $\partial_{\mu}$ to  \eqref{eq:phi_mu}, we obtain on $[0,T]$ the system of equations
\begin{subequations}\label{eq:sys1}
  \begin{empheq}[left=\empheqlbrace]{alignat=2}  
\d_\mu R_{\eps,\mu}&=-\D F_{\eps,\mu}[v_{\eps,\mu},\dot G_\mu R_{\epmu}]-\big(\d_\mu F_{\eps,\mu}[v_{\eps,\mu}]+\D F_{\eps,\mu}\big[v_{\eps,\mu},\dot G_\mu  F_{\eps,\mu}[v_{\eps,\mu}]\big]\big)\label{eq:flow1} \ \,,\\
        v_{\epmu}&= -\int_\mu^{T}\dot G_\nu\big(F_{\epnu}[v_\epnu]+R_{\epnu}\big)\rmd\nu\label{eq:phimu}\,,
  \end{empheq}
\end{subequations}
for any $T\in(0,1]$. One can view  the second term of the r.h.s. of \eqref{eq:flow1}, which is quadratic in  $F_{\eps,\mu}$, as a rough forcing term in the dynamic for $R_{\eps,\mu}$. However it will not be too rough -- since we will choose $F_{\eps,\mu}$ to satisfy an approximate flow equation.  
We will then solve ~\eqref{eq:sys1} for $R$ using path-wise / deterministic analysis, but we will only be able to due so for $\mu \in (0,{T}] \subset (0,1]$ for some random scale / time ${T} > 0$ --  which is enough to solve \eqref{eq:eq1} \textit{locally in time}.

\subsection{Notation and basic estimates}\label{subsec:notation}

We often consider functions $v$ taking values in a finite dimensional vector space $\mcE$ (typically a Hilbert space) with a distinguished basis $(e_i)_{i\in[\mathrm{dim}(\mcE)]}$ and we write $v=\sum_{i\in[\mathrm{dim}(\mcE)]}v^i e_i$. For $p\in[1,\infty]$ and $T\geqslant0$, we endow $C^\infty(\R,\mcE)$ with the norms
\begin{equs}    \norm{\psi}_{L^p_{T}}\eqdef\max_{i\in[\mathrm{dim}(\mcE)]}\bigg(\int_{-\infty}^T|\psi^i(t)|^p \rmd t\bigg)^{1/p}\,,
\end{equs}
with the usual convention when $p = \infty$. We drop the subscript $T$ when $T=1$. 
\auth{comment 2}
We denote by $L^p_T$ the completion, under the norm $\norm{\bullet}_{L^p_{T}}$ above, of all smooth functions on  $(-\infty,T]$ under the norm $L^p_T$.
We write $L^p_{0;T} \subset L^{p}_{T}$ for the subset of $L^p_T$ consisting of functions that vanish on $(-\infty,0)$.


Throughout the paper, we systematically identify distributions $K\in\mcD'(\R)$ with operators
    \begin{equs}
     K:C^\infty_c(\R)\ni F\mapsto   K(F)\equiv KF\eqdef K*F=\int_{\R}K(\bigcdot-z)F(z)\rmd z\,.
    \end{equs}
We endow such operators with the norm
\begin{equs}    \norm{K}_{\mcL^{p,\infty}}\equiv\norm{K}_{\mcL(L^p,L^\infty)}\,.
\end{equs}
When the kernel of $K$ is sufficiently integrable and supported on positive times, then the $\mcL^{p,\infty}$ norm corresponds to the $L^{p/(p-1)}(\R)$ norm of the kernel.

Next, following \cite[Section~4]{Duch21}, we define kernels $K_{N,\mu}$.
\begin{definition}\label{def:defQmu}
Fix $\mu\in(0,1]$. For $t\in\R$, we define
    \begin{equs}\label{eq:defQum}
        Q_\mu(t)\eqdef \mu^{-1}e^{-t/\mu}  \1_{\geqslant0}(t)
        \,,
    \end{equs}
the fundamental solution to the differential operator $\mcP_\mu\eqdef \big(1+\mu\partial_{t}\big) $ on $\R$, with null initial condition at $t=-\infty$. 
We sometimes write $Q_\mu$ for the operator given by convolution with the kernel $Q_\mu(t)$ on $C^\infty_c(\R)$, with the convention that $Q_0=\Id$ the operator with kernel $\delta_\R(t)$.\\
We now introduce the kernels we use to test scale-dependent quantities. For $N\in\N$, we let
\begin{equs}\label{eq:defKmu}
    K_{N,\mu}\eqdef Q_\mu^{*N}
\end{equs}
stand for the fundamental solution for $\mcP_{N,\mu}\eqdef\mcP^N_\mu$. We will also use the short-hand notations 
\begin{equ}\label{eq:fix_kernel}
K_\mu\eqdef K_{6,\mu}\,,\;
\text{  and  }\;\mcR_\mu\eqdef\mcP_{6,\mu}\,.
\end{equ}
\end{definition}

Given $\beta<0$, we define the parabolic Hölder-Besov space $\mcC^{\beta}([a,b])$ as the completion of smooth functions supported on $[a,b]$ under the norm\auth{comment 2}
\begin{equs}   \norm{\bigcdot}_{\mcC^\beta([a,b])}\eqdef\sup_{\mu\in(0,1]}\mu^{-\beta}\norm{K_{\lceil-\beta \rceil,\mu}\bigcdot}_{L^\infty_{a;b}}\,.
\end{equs}
For $\beta\in(0,1)$ and $b>0$, we define $\mcC^{\beta}([0,b])$ as the completion of smooth functions supported after $0$ under the norm
\begin{equs}    \norm{\bigcdot}_{\mcC^\beta([0,b]\times\T^n)}\eqdef\sup_{\mu\in(0,1]}\mu^{-\beta}\norm{(Q_\mu-\Id)\bigcdot}_{L^\infty_{0;b}}\,.
\end{equs}
Finally, given $k \in \N$, we write $C^k([a,b])$ for the usual Banach space of $k$-times continuously differentiable functions from $[a,b]$ to $\R$, and for $t \geqslant 0$, we write $B_0(t)\eqdef[-t,t] \subset \R$. 

We now give an elementary bound on the scale derivative of the effective Green's function.
\begin{lemma}\label{lem:dotG}
For every $N\in\N$, we have
\begin{equs}        \label{eq:heat1}
    \norm{\mcP_{N,\mu}\dot G_\mu}_{\mcL^{p,\infty}}&\lesssim \mu^{-1/p}\;\mathrm{for}\;\mathrm{all}\;p\in[1,\infty]
    \end{equs}
uniformly in $\mu\in(0,1]$.
\end{lemma}
\begin{proof} 
Observe that $\dot G_\mu(t)=\mu^{-1}\tilde\chi(t/\mu)$ where $\tilde\chi=-\Id\chi$. 
Therefore, we have $\dot G_\mu(t)=$ $\mcS_\mu\tilde\chi$ where $\tilde\chi$ is integrable on $\R$, so that the thesis follows using the bound
\begin{equs} \label{eq:boundSmu}   \norm{\mcS_\mu\rho}_{L^{p}(\R)}&\lesssim\mu^{-(p-1)/p}\norm{\rho}_{L^{p}(\R)}\,,\;\text{$\mathrm{for}$}\;\text{$\mathrm{all}$}\;p\in[1,\infty]\,.
\end{equs}
\end{proof}

\section{Deterministic analysis}\label{sec:2}
We now introduce the coordinates we will use for the flow approach. 
Since our non-linearity $V$ is not necessarily polynomial, we will use elementary differentials as coordinates as in \cite{gKPZFlow}.
However, in \cite{gKPZFlow} these elementary differentials were indexed by multi-indices - this seems more natural for scalar valued solutions driven by a scalar noise, while for $\cbn \vee \cbm > 1$ it seems more natural to follow the standard \cite{Cay1857,Trees} approach of indexing elementary differentials by rooted trees.  

\subsection{Coordinates for the flow equation}\label{subsec:tree_coord}
We first describe an indexing of elementary differentials (and associated force coefficients) using \textit{labelled} rooted trees, which allows us to give concrete and simple definitions of the key grafting operations \eqref{eq:labeled_graftingrelation} and \eqref{eq:defB}, before formulating them with unlabelled trees in the subsequent subsection.

\subsubsection{Elementary differentials indexed by labelled trees}\label{sec:elemdiff_labelled}
\begin{definition}
 A \emph{labelled} rooted tree $\btau = (N,E,r)$ is a \textit{connected acyclic graph} with finite node set $N \subset \N$, edge set $E \subset N^{(2)}$, and a distinguished root node $r \in N$. 
 We often write $N(\btau) = N$, $E(\btau) = E$ and $r(\btau) = r$. 


We write $\bT$ for the set of all labelled rooted trees, and often just call elements of $\bT$ labelled trees. 
We also write, for any $r\in\N$, $\circ_r$ for the labeled tree with a single node $N(\circ_{r}) = \{r\}$.
\end{definition}

\begin{definition}
Given $\btau, \bsigma \in \bT$, we call a bijection $f: N(\btau) \rightarrow N(\bsigma)$ an \textit{isomorphism} from $\btau$ to $\bsigma$ if $f\big(r(\btau)\big) = r(\bsigma)$, and if the induced mapping $\vec{f}: E(\btau) \rightarrow \big(N(\bsigma)\big)^{(2)}$ given by $e = \{v,\bar{v}\} \mapsto \vec{f}(e) = \{f(v),f(\bar{v})\}$ is in fact a bijection between $E(\btau)$ and $E(\bsigma)$. 

We write $\Iso(\btau,\bsigma)$ for the collection of isomorphisms from $\btau$ to $\btau'$ and use the shorthand $\Aut(\btau) \eqdef \Iso(\btau,\btau)$ for the \textit{automorphism group} of $\tau$. 
We say $\btau$ and $\bsigma$ are isomorphic if there exists an isomorphism between them, and denote by $\mfT$ the set of isomorphism classes of elements of $\bT$. 
An element of $\mfT$ is called an \textit{unlabelled rooted tree}, or just a tree. 

We write $\circ = \{ \circ_{r}: r \in \N\}$ as the isomorphism class of labelled trees with a single node. 
\end{definition}

Recall that we will use these trees to index elementary differentials. 
In the next definition, we make use of the fact that given a labeled tree $\btau$, we can always write, for some $k \geqslant 0$ and $r \in \N$, 
\begin{equ}\label{eq:product_of_trees}
\btau = [\btau_{1}, \dots,  \btau_{k}]_{r}\;.
\end{equ}
where $\btau_1, \dots, \btau_k$ are labelled trees whose node sets are disjoint and do not contain $r \in \N$, and $\btau$ consists of of joining all the roots of the $\btau_1,\dots,\btau_k$ to the new root $r$. 

We now describe how these labelled trees correspond to elementary differentials. 
\begin{definition}\label{def:elem_diff1}  
First, we associate to each labelled tree $\btau \in \bT$ the Hilbert space $\mathcal{H}^{\btau} \eqdef \bigotimes_{n \in N(\btau)} \R^{\cbm}$. 
We introduce this notation since it is natural to write $j$-multi-linear functionals of the noise using duality pairings in $(\R^{\cbm})^{\otimes j}$, and since in our trees each node will represent an occurrence of the noise.
We also write 
\[
\text{$\cbB$}[\btau] \eqdef C^{\infty}\Big( (\R^{\cbn})^{N(\btau)}, \R^{\cbn} \otimes \mcH^{\btau}\Big)\;.\]

The elementary differential corresponding to a labelled tree $\btau$ will be a map $\bUpsilon^{\btau} \in \cbB[\btau]$ inductively defined by setting, for $\btau$ as in \eqref{eq:product_of_trees}, and $x_{N(\btau)} = (x_{n})_{n \in N(\btau)} \in (\R^{\cbn})^{N(\btau)}$, 
\begin{equ}\label{eq:elem_diff1}
\bUpsilon^{\btau}(x_{N(\btau)}) 
=
\big( \dd^{k}V(x_{r(\btau)}) \big) \big[ \bUpsilon^{\btau_{1}}(x_{N(\btau_1)}),\cdots, \bUpsilon^{\btau_{k}}(x_{N(\btau_k)}) \big]\,,
\end{equ}
where $\dd$ here represents taking a differential on $\R^{\cbn}$.
The the base case for this induction given by the above formula when $k=0$, that is when $\btau = \circ_{r}$, so
\begin{equ}\label{eq:elem_diff1_base}
\bUpsilon^{\circ_{r}}(x_{r})
=
V(x_{r})\;.
\end{equ}
and we use the canonical isomorphism $\mathcal{L}(\R^{\cbn},\R^{\cbm}) \simeq \R^{\cbn} \otimes \R^{\cbm}$. 
\end{definition}
Note that in \eqref{eq:elem_diff1} we are overloading this derivative notation in the natural way: if $a_{j} \otimes b_{j} \in \R^{\cbn} \otimes \mathcal{H}^{\btau_j}$ for every $1 \leqslant j \leqslant k$, then we set
\[
\rmd^{k}V(\bigcdot) [ a_{1} \otimes b_{1} ,\dots,  a_{k} \otimes b_{k} ]
\eqdef
\rmd^{k}V(\bigcdot)[ a_{1} ,\dots,  a_{k}] 
\otimes  b_{1} \otimes \cdots \otimes b_{k} \in \R^{\cbn} \otimes \mathcal{H}^{\btau}\;. 
\]
Note also that our explicit definition of $\mathcal{H}^{\btau}$ and $\bUpsilon^{\btau}$ refers to our labelled tree $\btau$. 

One twist versus the elementary differentials seen in Butcher series is non-locality in its dependence on the solution -- i.e. the variables $x_{N(\btau)}$ associated to each node of the labelled tree are allowed to be distinct. 
Given $\btau\in\bT$, $\Gamma^{\btau} \in \text{$\cbB$}[\btau]$, and $v:\R\rightarrow\R^{\cbn}$, we define the composition $\Gamma^{\btau}[v]$ acting on $s_{N(\btau)} = (s_{n})_{n \in N(\btau)} \in \R^{N(\btau)}$ by setting 
\begin{equ}\label{eq:path_eval}
\Gamma^{\btau}[v](s_{N(\btau)}) \eqdef \Gamma^{\btau}\Big( \big( v(s_{n}) : n \in N(\btau) \big) \Big)\;.
\end{equ}

Ignoring truncations that we will introduce in Section~\ref{subsec:trunc}, and allowing for infinite sums over trees, we start with an expansion of an effective force $A_{\epmu}$ in terms of labelled trees would be given by 
\begin{equs}\label{eq:defFmu_0}
    \text{$A$}_{\epmu}[v](t)=\sum_{\btau\in \bT} \int_{\R^{N(\btau)}}\langle \mfz_\epmu^{\btau}(t,s_{N(\btau)}),\bUpsilon^{\btau}[v+{\tt u}](s_{N(\btau)})\rangle_{\mcH^{\btau}}\rmd s_{N(\btau)}\,,
\end{equs}
for some stochastic force coefficients 
 $\text{$\mfz$}_\epmu^{\btau} \in \text{$\cbC$}[\btau]$ where, for $\btau \in \bT$ we set
\begin{equ}\label{def:coordinate}
\text{$\cbC$}[\btau] \eqdef C^{\infty}\Big( [0,1]\times \R^{N(\btau)},  \mcH^{\btau}\Big)\,.
\end{equ}

\begin{remark}\label{rem:not_smooth}
For Sections~\ref{sec:elemdiff_labelled}, \ref{sec:2_1_2}, and \ref{subsec:trunc} we assume that $V$ in \eqref{eq:eq0} is smooth and the force coefficients are also smooth in the $s$ variables but allow a delta function dependence in $t$ (see the initialisation given in the next remark).  
The discussion here is primarily to describe algebraic/combinatorial operations \eqref{eq:labeled_graftingrelation} \slash  
and \eqref{eq:defB} \slash \eqref{eq:grafting_forcecoeff} that describe the hierarchy of elementary differentials and force coefficients. 

The truncation we perform in \eqref{subsec:trunc} makes it clear that $V \in C^3$ will in the end suffice for defining these operations since this is the maximum branching we will see in any tree. 
More major is the extension to rough force coefficients and subsequent sections will introduce the necessary norms and probabilistic and deterministic arguments to allow this. 
\end{remark}

\begin{remark}\label{remark:two_forces}
In the expansion \eqref{eq:defFmu_0}, we will always encode the functional dependence of the effective force on $v$ through elementary differentials translated by ${\tt u}$, but we will have two different families of important force coefficients whose main difference is their initialisation data for $\mu = 0$ for some fixed\footnote{This choice of $r$ will disappear when we move to unlabelled trees in the next section} $r$:
\begin{itemize} 
    \item A collection $\xi_\epmu^{\btau}$ of stationary (in time) force coefficients, that have the benefit of being simpler to construct, and are initialised using 
    \begin{equ}
    \xi_{\eps,0}^{\circ_r}(t,s_r) = \delta(t-s_r)\xi_{\eps}(s_r)\quad
    \text{ and }\quad \xi_{\eps,0}^{\circ_r} = 0 \text{ for } \btau \not = \circ_{r}\;.
    \end{equ}
\item  A collection $\zeta_\epmu^{\btau}$ of non-stationary force coefficients, that are initialised using 
  \begin{equ}
    \zeta_{\eps,0}^{\circ_r}(t,s_r)  = 
    \delta(t-s_r)
    \mathbf{1}_{\geqslant0}(s_r) \xi_{\eps}(s_r)
    \quad \text{ and }\quad \zeta_{\eps,0}^{\circ_r} = 0 \text{ for } \btau \not = \circ_{r}\;, 
\end{equ}
which is the correct initialisation for solving the initial value problem for \eqref{eq:eq1}. 
\end{itemize}  
\end{remark}

We now show how grafting allows us to write the Polchinski flow. 
Indeed, the initial data at scale $\mu = 0$ is given by the contribution from just $\circ_{r}$, and the expansion of \eqref{eq:defFmu_0} for $\mu > 0$ is generated by iterative graftings of trees due to the non-linear term
\begin{equ}\label{eq:Pol_nonlinearity}
   \text{$\D$}A_{\epmu} \big[v, \dot G_\mu A_{\epmu}[v] \big]\;.
\end{equ}
When grafting labelled trees, it is convenient to make sure node sets of trees we graft are disjoint. 
\begin{definition}\label{def:grafting0}
Given $\ell \in\N$ and a labelled tree $\bsigma\in\bT$, we define a new labelled tree ${\bsigma}+\ell \in\bT$ to be the labeled tree with node set ${N}(\bsigma+\ell) \eqdef \{ j + \ell : j \in N(\bsigma)\}$ and such that the \textit{shift by $\ell$} defined as $ N(\bsigma)\ni j \mapsto j+ \ell \in{N}(\bsigma+\ell)$ belongs to $\Iso(\bsigma,\bsigma+\ell)$ -- in particular, we have $ r(\bsigma+\ell)= r(\bsigma)+\ell$. 
Note that $\bsigma + 0 = \bsigma$. 

Finally, given two labelled trees $\btau,\bsigma\in\bT$, we set 
\[
\ell_{\btau,\bsigma}
\eqdef 
 \begin{cases}  
    \max N(\btau) + 1& \text{if} \;N(\btau) \cap N(\bsigma) \not = \emptyset\,, \\
    0 & \text{otherwise} \,. \\
\end{cases}
\]
We overload notation, identifying $\ell_{\btau,\bsigma}$ with the shift-isomorphism and write $\ell_{\btau,\bsigma} \in \Iso(\bsigma, \bsigma + \ell_{\btau,\bsigma})$. 
\end{definition}
We now describe grafting on labelled trees.
\begin{definition}\label{def:grafting}
Given $\btau_1, \btau_2\in \bT$, we define, for any $n \in N(\btau_1)$ a new labelled tree $\btau_2 \graft_{n} \btau_1 \in \bT$ which will be constructed by ``grafting the tree $\btau_2$ onto $\btau_1$ at the node $n$''. More precisely, we set 
\begin{equs}  
r(\btau_2 \graft_{n} \btau_1 )\eqdef r(\btau_1)\,,\; 
\enskip 
N(\btau_2 \graft_{n} \btau_1 ) \eqdef N(\btau_1) \cup {N}\big(\btau_2+ \ell_{\btau_1,\btau_2} \big)\,,
\end{equs}
and
\begin{equs}   
{E}(\btau_2 \graft_{n} \btau_1) \eqdef E(\btau_1) \cup E\big(\btau_2+ \ell_{\btau_1,\btau_2} \big) \cup \big\{\{{r}\big(\btau_2+ \ell_{\btau_1,\btau_2} 
 \big),n\}\big\}\,.
\end{equs}
\end{definition}

\begin{remark}\label{rem:summed_grafting}
Note that for $n \not = r(\btau_1)$, the operation $\btau_2\graft_{n} \btau_1$  depends on $\btau_1$ being labelled -- in particular, it does not directly lift to an operation on isomorphism classes of trees. 

However, if we sum over all the nodes $n \in N(\btau_1)$ we are grafting over, this does lift to trees without labelling. 
Namely, one can define 
$\graft : \mVec(\bT) \otimes \mVec(\bT) \rightarrow \mVec(\bT)$ given by
\begin{equ}\label{summed_grafting}
\btau_2 \graft \btau_1 = \sum_{n \in N(\btau_1)} \btau_2 \graft_{n} \btau_1\;,
\end{equ}
which lifts to a well-defined map $\graft: \mVec(\mfT) \otimes \mVec(\mfT) \rightarrow  \mVec(\mfT)$ by setting 
\begin{equ}\label{summed_grafting_iso}
\pi(\btau_2) \graft \pi(\btau_1) = \sum_{n \in N(\btau_1)} \pi(\btau_2 \graft_{n} \btau_1)\;,
\end{equ}
where $\pi: \bT \rightarrow \mfT$ maps labeled trees to (unlabelled) trees. 
Note that each labelled tree appearing on the right hand side of \eqref{summed_grafting} appears at most once, while (unlabelled) trees can appear with multiplicity on the right hand side of \eqref{summed_grafting_iso}. 
\end{remark}

Next we define some useful actions of isomorphisms between labelled trees. 
\begin{definition}\label{def:iso_action}  
Given $\btau, \bsigma \in \bT$ and $f\in\Iso(\btau,\bsigma)$, we define an induced linear isomorphism $\hat f: \mathcal{H}^{\btau} \rightarrow \mathcal{H}^{\bsigma}$ and, for any $k \in \{1,\cbn\}$, a bijection $\tilde f^{[k]}: (\R^k)^{N(\btau)} \rightarrow (\R^k)^{N(\bsigma)}$ by setting
\begin{equs}\label{induced_maps}
\hat f:\mathcal{H}^{\btau} 
\ni
\bigotimes_{n \in N(\btau)}
e_{n}
&\mapsto 
\bigotimes_{m \in N(\bsigma)} e_{f^{-1}(m)} \in \mathcal{H}^{\bsigma}\;,\\
\tilde f^{[k]}:(\R^{k})^{N(\btau)} 
\ni
\big(x_n: n \in N(\btau)\big) 
&\mapsto 
 \big(x_{f^{-1}(m)} : m \in N(\bsigma)\big) \in (\R^{k})^{N(\bsigma)}\;.
\end{equs}
We also lift $f$ to a map $\mathbf{f}:\cbB[\btau]\rightarrow \cbB[\bsigma]$ by setting $$\text{$\mathbf{f}$}F(\bigcdot) \eqdef \big(\Id_{\R^\cbn}\otimes \hat f\big) \circ F(\tilde f^{[\cbn]} \,\bigcdot )\,.$$ 
Overloading notation, we also lift $f$ to a map $\text{$\mathbf{f}$}:\text{$\cbC$}[\btau]\rightarrow \text{$\cbC$}[\bsigma]$ by setting $$\text{$\mathbf{f}$}F(\bigcdot) \eqdef \hat f \circ F\big( 
 (\Id_{[0,1]}\otimes \tilde{f}^{[1]}) \bigcdot \big)\,.$$ 
\end{definition}
We now formulate, for $\btau_1, \btau_2 \in \bT$ and $n \in N(\btau_1)$, a map  $$\graft_{n}: \cbB[\btau_1] \otimes \cbB[\btau_2] \rightarrow \cbB[\btau_2 \graft_n \btau_1]=\cbB[\btau_1]\otimes\cbB[\btau_2+\ell_{\btau_1,\btau_2}]$$ by setting, for $\Gamma^{\btau_1} \in \cbB[\btau_1] $, and $\Gamma^{\btau_2} \in \cbB[\btau_1]$, 
\begin{equ}\label{eq:labelled_grafting_function}
\Gamma^{\btau_2} \graft_{n} \Gamma^{\btau_1}\equiv \graft_n(\Gamma^{\btau_1},\Gamma^{\btau_2}) 
\eqdef
\dd_{n}\Gamma^{\btau_1}\big[
\bell_{\btau_1,\btau_2}
\Gamma^{\btau_2}\big]\,,
\end{equ}
where we are applying Definition~\ref{def:iso_action}  to $f = \ell_{\btau_1,\btau_2}$ which gives us a map $\bell_{\btau_1,\btau_2}:\cbB[\btau_2]\rightarrow \cbB[\btau_2 + \ell_{\btau_1,\btau_2}]$ and on the right hand side we are using the canonical isomorphism between $\mathcal{H}^{\btau_1} \otimes \mathcal{H}^{\btau_2 + \ell_{\btau_1,\btau_2}}$ and $\mathcal{H}^{\btau_1 \graft_n \btau_2}$. 
Here $\dd_{n}$ denotes a total $\R^{\cbn}$-derivative in the $n \in N(\btau_1)$ component of $(\R^{\cbn})^{N(\btau_1)}$.

It then follows that, for any $\btau_1, \btau_2 \in \mfT$,  $n \in N(\btau)$, and $v: \R \rightarrow \R^{\cbn}$, we have 
\begin{equ}\label{eq:labeled_graftingrelation}
\bUpsilon^{\btau_2}
\graft_{n}
\bUpsilon^{\btau_1}
=
\bUpsilon^{\btau_2 \graft_n \btau_1}
\quad
\text{and}
\quad
\mathrm{D}\bUpsilon^{\btau_1} 
\big[ v,
\bell_{\btau_1,\btau_2}\bUpsilon^{\btau_2}[v] \big]
=
\sum_{n \in N(\btau_1)}
\bUpsilon^{\btau_2 \graft_n \btau_1} [v]\;.
\end{equ}
Using this notation to expand the non-linearity in \eqref{eq:Pol_nonlinearity} evaluated on \eqref{eq:defFmu_0} gives:
\begin{equs}\label{eq:expressionDF}
{}&\text{$\D$}A_{\epmu}\big[v,\dot G_\mu A_{\epmu}[v] \big](t)\\
&=
\sum_{\btau \in \bT}
\int_{\R^{N(\btau)}}
\Big\langle \sum_{\btau_1, \btau_2 \in \bT} 
\sum_{n \in N(\btau_1)}
\text{$\1$}\{ \btau = \btau_2 \graft_n \btau_1 \} 
\big( \text{$\mfz$}^{\btau_2}_\epmu \graft_{n}^{\mu} \text{$\mfz$}^{\btau_1}_\epmu \big)  (t,s_{N(\btau)}),\bUpsilon^{\btau}[v+{\tt u}](s_{N(\btau)}) \Big\rangle_{\mcH^{\btau}} \rmd s_{N(\btau)}\,,
\end{equs}
where the notation $\graft_{n}^{\mu}$ is defined below. 

\begin{definition}\label{def:scaledlabeledgraft}
Fix $\mu \in[0,1]$, $\btau_1,\btau_2\in\bT$, and $n\in N(\btau_1)$. 
We define $$\graft_{n}^{\mu}: \text{$\cbC$}[\btau_1]  \otimes \text{$\cbC$}[\btau_2] \rightarrow \text{$\cbC$}[\btau_2 \graft_n \btau_1]=\cbC[\btau_1]\otimes\cbC[\btau_2+\ell_{\btau_1,\btau_2}]$$ as follows. For $\text{$\mfz$}^{\btau_1} \in\cbC[\btau_1]$, $\text{$\mfz$}^{\btau_2} \in\cbC[\btau_2]$ and using the shorthand ${\btau}_{1,2,n} = \btau_2 \graft_n \btau_1$, we let 
\begin{equs}[eq:defB]
{}&
\big(\text{$\mfz$}^{\btau_2} \graft_{n}^{\mu} \text{$\mfz$}^{\btau_1}
\big)(t,s_{N({\btau_{1,2,n}})})\equiv\graft_n^\mu\big(\text{$\mfz$}^{\btau_1},\text{$\mfz$}^{\btau_2}\big)(t,s_{N({\btau_{1,2,n}})})\\
{}& \enskip \eqdef  \text{$\mfz$}^{\btau_1}(t,s_{N(\btau_1)})\int_{0}^1 \dot G_\mu(s_n-w )(\bell_{\btau_1,\btau_2}\text{$\mfz$}^{\btau_2})(w,s_{N({\btau_2}+\ell_{\btau_1,\btau_2})}) \rmd w\,
    \in \text{$\cbC$}[\btau] \,.\textcolor{white}{blabla}
\end{equs}
\end{definition}

The formal equation \eqref{eq:expressionDF} shows that the Polchinski flow preserves our system of coordinates. However, while labellings were necessary for writing down these explicit formulas for $\bUpsilon$, they are already inconvenient -- we need to relabel with $\bell_{\bigcdot,\bigcdot}$ when grafting trees).
Infinitely labelled trees correspond to the same tree, but the labelling itself is an extraneous and messy bit of data to carry along. 

It is cleaner to index the expansion \eqref{eq:defFmu_0} with (unlabelled) trees, in the following section we use some ideas from \cite[Sec. 5]{CCHS2d} to do this. 
Those readers who are comfortable with taking the final representation given in Lemma~\ref{lem:unlabelled_Pol} on trust can skip to that point in the article.

\subsubsection{Elementary differentials indexed by isomorphism classes of trees}\label{sec:2_1_2}
We first introduce analogs of the spaces $\text{$\cbB$}[\btau]$ and $\text{$\cbC$}[\btau]$ that are independent of the labelling of the tree.
These are defined as a class of ``sections'' in a ``bundle'' over a base space of different labellings of $\tau$ with that the constraint that the sections should be invariant under isomorphisms of labelled trees that move between fibers. 

\begin{definition}\label{def:sections}
Given $\tau \in \mfT$, we define $\text{$\cbB$}[\tau] \subset \prod_{\btau \in \tau} \text{$\cbB$}[\btau]$ 
by setting 
\begin{equs}
\text{$\cbB$}[\tau]
\eqdef
\big\{ 
\big(F_{\btau}\in \text{$\cbB$}[\btau] : \btau \in \tau\big): \enskip 
\forall \btau, \btau' \in \tau\;\text{and}\; f \in \Iso(\btau,\btau')\,,\; \mathbf{f} F_{\btau} =  F_{\btau'}
\big\}\,.
\end{equs}
We define $\text{$\cbC$}[\tau]\subset \prod_{\btau \in \tau} \text{$\cbC$}[\btau]$ analogously. 

Moreover, for $\btau \in \tau \in \mfT$, let $\iota_{\tau,\btau}: \text{$\cbB$}[\tau] \rightarrow \text{$\cbB$}[\btau]$ be the ``fiber'' map $\big(F_{\btau'}\in\mcH^{\btau'} : \btau' \in \tau\big) \mapsto F_{\btau}$. 
Observe that $\iota_{\tau,\btau}$ has a left inverse given by the ``section'' map $\pi_{\tau,\btau}: \text{$\cbB$}[\btau] \rightarrow \text{$\cbB$}[\tau]$ defined by setting
\[
\text{$\cbB$}[\btau]\ni F_{\btau}\mapsto\pi_{\tau,\btau}F_{\btau}\eqdef\Big(\frac{1}{|\Iso(\btau,\btau')|}
\sum_{f \in \Iso(\btau,\btau')} \text{$\mathbf{f}$} F_{\btau} : \btau'\in\tau\Big)\in \text{$\cbB$}[\tau]
\,.
\] 
We use the same notation for the analogous maps between $\text{$\cbC$}[\btau]$ and $\text{$\cbC$}[\tau]$. 

Note that $|\Iso(\btau,\btau')|$ actually does not depend on the choice of $\btau,\btau'\in\tau$. In the sequel, we therefore just write $|\Aut(\tau)|\eqdef |\Aut(\btau)|$ for any $\btau\in\tau$.

Finally, for $\btau \in \tau$ we sometimes write $\iota_{\btau} \eqdef \iota_{\tau,\btau}$ and $\pi_{\btau} \eqdef \pi_{\tau,\btau}$ to lighten subscripts. 
\end{definition}

The elementary differential and force coefficient corresponding to a given $\tau \in \mfT$ will be elements of $\text{$\cbB$}[\tau]$ and $\text{$\cbC$}[\tau]$, respectively. 

The main idea behind passing between operations and objects defined for labelled trees to those defined for unlabelled trees follows two steps. 
The first step is to verify that the operation/function on labeled objects behaves  ``covariantly'' with respect to relabellings/isomorphisms on labelled trees. 
The second step is to transfer the operation/function from labelled trees to unlabelled trees by using section and fiber maps -- this will be well-defined because of the first step. 
We first describe this in the case of elementary differentials themselves in the following definition.

\begin{definition} We define the family of elementary differentials $(\Upsilon^{\tau}: \tau\in\mfT)$, with $\Upsilon^{\tau} \in \text{$\cbB$}[\tau]$, by using an induction in $\tau$. 
For the base case, we define  $\Upsilon^{\circ} = \big( \Upsilon^{\circ_{r}}\big)_{r \in \N} \in \prod_{r \in \N} \text{$\cbB$}[\circ_{r}]$ by setting, in analogy with \eqref{eq:elem_diff1_base}, 
\[
\Upsilon^{\circ_{r}}(x_{r}) = V(x_{r})\;.
\]
It is clear that one indeed has $\Upsilon^{\circ_{r}} \in \text{$\cbB$}[\circ]$. 
In other words, we could have just fixed a single $r \in \N$, applied the above formula, and then defined $\Upsilon^{\circ} = \pi_{\circ,\circ_{r}} \Upsilon^{\circ_{r}}$ and our construction would be well-defined in that it didn't depend on our initial choice of $r$. 

We now proceed to the inductive step. 
Suppose we have $\tau=[\tau_1,\dots,\tau_k]$, where $k > 0$, the $\tau_1,\dots,\tau_k$ are unlabelled, and $[\cdot]$ denotes joining the roots of unlabelled trees at a new unlabelled root. 
Our inductive assumption is that we have been given $\Upsilon^{\circ} \in \text{$\cbB$}[\bignoise]$, and $\Upsilon^{\tau_{j}} \in \text{$\cbB$}[\btau_j]$ for $1 \leqslant j \leqslant k$.

Next, fix any labelled $\btau \in \tau$, and choose any labelled $\btau_1,\dots,\btau_k$ with $\btau_j \in \tau_j$, such that $\btau = [\btau_1,\dots,\btau_k]_{r}$ (here we use the notation of \eqref{eq:product_of_trees}, and $r = r(\btau)$).
For every $\bsigma\in \{\circ_{r}, \btau_1,\cdots,\btau_k\}$, we first set
\begin{equs}\label{eq:deftildeupsilon}
   \Upsilon^{\bsigma} \eqdef \iota_{\sigma,\bsigma} \Upsilon^{\sigma}\in\text{$\cbB$}[\bsigma]\,.
\end{equs}
We then define $\Upsilon^{\btau} \in \text{$\cbB$}[\btau]$ in analogy with \eqref{eq:elem_diff1} by setting
\begin{equ}\label{eq:elem_diff2}
\Upsilon^{\btau}(x_{N(\btau)})
=
\big( \dd^{k} \Upsilon^{\circ_r}(x_{r(\btau)}) \big) \big[ \Upsilon^{\btau_{1}}(x_{N(\btau_1)}),\cdots, \Upsilon^{\btau_{k}}(x_{N(\btau_k)}) \big]\,,
\end{equ}
and finally we can set
\begin{equ}\label{eq:already_sym}
\Upsilon^{\tau} \eqdef
\pi_{\tau,\btau} \Upsilon^{\btau}\;.
\end{equ}
In the next lemma we verify that $\Upsilon^{\tau}$ as defined above is well-defined, in that it does not depend on how we realised $\tau$ as the labelled tree $[\btau_1,\dots,\btau_k]_{r}$. 
\end{definition}
\begin{lemma}\label{lem:already_sym}
The elementary differential $\Upsilon^\tau$ is well-defined. 
Moreover, one has, for any $\btau \in \tau \in \mfT$,  $\Upsilon^{\btau} = \bUpsilon^{\btau}$ where the latter is defined as in Definition~\ref{def:elem_diff1}. 
\end{lemma}
\begin{proof}    
The case of $\tau = \circ$ is immediate as described above. 
 
For the inductive step, one should see the right hand side of $\eqref{eq:elem_diff2}$ in terms of a map 
\[
\text{$\cbB$}[\circ_r] \otimes \text{$\cbB$}[\btau_1] \otimes \cdots \otimes \text{$\cbB$}[\btau_k] \rightarrow \text{$\cbB$}[\btau]\,.
\]
To see \eqref{eq:elem_diff2} is indeed well-defined as such a map, we first note that there is an obvious bijection from $N(\circ_{r}) \sqcup \big(N(\btau_1) \sqcup \cdots \sqcup N(\btau_k)\big)$ to $N(\btau)$ that transports relabellings on $\circ_{r}$ and $\btau_1,\dots, \btau_k$ to relabellings of $\btau$ -- our construction here does not depend on the specific choice of $r$ and $\btau_j \in \tau_j$. 
Moreover, the symmetry of the derivative means our construction also does not depend on the ordering of the $\tau_1,\dots,\tau_k$ themselves. 

For the second claim we again argue inductively. 
For the base case it is immediate. 
For the inductive step, it suffices to verify that for any $\btau \in \tau$, $\iota_{\btau} \Upsilon^{\tau} = \Upsilon^{\btau}$ - this means that there is no actual symmetrization occurring in \eqref{eq:already_sym} and so the inductive step for $\Upsilon^{\btau}$ matches that of $\bUpsilon^{\btau}$.
What must be checked is that $\Upsilon^{\btau}$ automatically satisfies $\mathbf{f} \Upsilon^{\btau} = \Upsilon^{\btau}$ for every $f \in \Aut(\btau)$ if $\Upsilon^{\btau_1}, \dots, \Upsilon^{\btau_k}$ are assumed to satisfy to be invariant under $\Aut(\btau_1),\dots,\Aut(\btau_k)$. 
However, this is immediate since $\Aut(\btau)$ is generated by $\Aut(\btau_1),\dots,\Aut(\btau_k)$ and swaps of trees -- and for the latter we can argue due to the symmetry of $\dd^{k}$ in its $k$-arguments. 
\end{proof}

We now prepare to write the \eqref{eq:Pol_nonlinearity} in terms of grafting of unlabelled trees.  
We write 
\[
\text{$\cbB$} \eqdef \prod_{\tau \in \text{\scriptsize{$\mfT$}}} \text{$\cbB$}[\tau]
\quad
\text{and}
\quad
\text{$\cbC$} \eqdef \prod_{\tau \in\text{\scriptsize{$\mfT$}}} \text{$\cbC$}[\tau]\;.
\]
\begin{definition}\label{def:grafting_full}  
We define, for any $\mu \in [0,1]$, a ``grafting with scale'' $\graft^{\mu}: \text{$\cbC$} \otimes \text{$\cbC$} \rightarrow \text{$\cbC$}$. 

Let $\mu \in [0,1]$. 
Given $\tau_1,\tau_2 \in \mfT$ and $\text{$\mfz$}^{\tau_1} \in \text{$\cbC$}[\tau_1]$, $\text{$\mfz$}^{\tau_2} \in \text{$\cbC$}[\tau_2]$, we define $\text{$\mfz$}^{\tau_2} \graft^{\mu} \text{$\mfz$}^{\tau_1} \in \text{$\cbC$}$ as follows. 
We first choose $\btau_1 \in \tau_1$, $\btau_2 \in \tau_2$, and write  $\text{$\mfz$}^{\btau_1} = \iota_{\tau_1,\btau_1} \text{$\mfz$}^{\tau_1}$, $\text{$\mfz$}^{\btau_2} = \iota_{\tau_2,\btau_2} \text{$\mfz$}^{\tau_2}$. 
We then set
\begin{equ}\label{eq:grafting_forcecoeff}
\text{$\mfz$}^{\tau_2} \graft^{\mu} \text{$\mfz$}^{\tau_1} 
\eqdef \sum_{n \in N(\btau_1)} 
\pi_{\btau_{1,2,n}} \big(
\text{$\mfz$}^{\btau_1}\graft^\mu_n\text{$\mfz$}^{\btau_2}\big)\,,
\end{equ}
where $\graft^\mu_n$ is given in Definition~\ref{def:scaledlabeledgraft}. 
where we are writing $\btau_{1,2,n} = \btau_{2} \graft_n \btau_{1}$, and using, for $\bsigma \in \sigma \in \mfT$, the shorthand $\pi_{\bsigma} =\pi_{\sigma,\bsigma}$. 
We then extend $\graft^\mu$ to $\cbC\otimes\cbC$ by linearity.
\end{definition}

One must argue that $\text{$\mfz$}^{\tau_2} \graft^{\mu} \text{$\mfz$}^{\tau_1}$ is indeed well defined, this is done in the following lemma.

\begin{lemma}
For any $\mu \in [0,1]$, $\graft^{\mu}: \text{$\cbC$} \otimes \text{$\cbC$} \rightarrow \text{$\cbC$}$ is well-defined. 
\end{lemma}
\begin{proof}   
It suffices to check that the construction of $\text{$\mfz$}^{\tau_2} \graft^{\mu} \text{$\mfz$}^{\tau_1} $ given in Definition~\ref{def:grafting_full} does not depend on the labellings of $\tau_1$ and $\btau_2$ we chose. 

This follows from the observation for each $n$, there is an explicit bijection between $N(\btau_1) \sqcup N(\btau_2)$ and $N(\btau_2 \graft_n \btau_1)$, and relabelling of $\btau_1$ and $\btau_2$ just generate a rearranging of the sum over $n$ and relabellings of trees that appear in the summands. 
\end{proof}

\begin{remark}
One can argue similarly to show that there is a well-defined map $\graft: \text{$\cbB$} \otimes \text{$\cbB$} \rightarrow \text{$\cbB$}$. 
Given $\tau_1,\tau_2 \in \mfT$, $\Gamma^{\tau_1} \in \text{$\cbB$}[\tau_1]$ and $\Gamma^{\tau_2} \in \text{$\cbB$}[\tau_1]$ one defines $\Gamma^{\tau_2} \graft \Gamma^{\tau_1} \in \text{$\cbB$}$ as follows. 

Again choose $\btau_1 \in \tau_1$, $\btau_2 \in \tau_2$, and write $\Gamma^{\btau_1}  =  $, $\Gamma^{\btau_2}  = \iota_{\tau_2,\btau_2} \Gamma^{\tau_2} $. 
We can then set
\begin{equ}\label{eq:grafting_elemdiff}
\Gamma^{\tau_2}  \graft \Gamma^{\tau_1} \eqdef \sum_{n \in N(\btau_1)} \pi_{\btau_{1,2,n}} 
\big( (\iota_{\btau_2} \Gamma^{\tau_2} ) \graft_{n} ( \iota_{\btau_2} \Gamma^{\btau_2}  
)\big)\, \in \text{$\cbB$}\,,
\end{equ}
where $\graft_{n}$ is as \eqref{eq:labelled_grafting_function} and $\btau_{1,2,n} = \btau_{2} \graft_{n} \btau_{1}$.

In particular, one can find symmetry factors $S(\tau) \in \N_{>0}$ such that, if we define 
\[
\Upsilon \eqdef\prod_{\tau \in \mathfrak{T}} \frac{1}{S(\tau)}\Upsilon^{\tau} \in \text{$\cbB$}\,,
\]
then 
\[
\Upsilon \graft \Upsilon = \Upsilon\;.
\]
An analogue of the above relation is important in regularity structures \cite{BCCH21}, but does not immediately seem to play an important role in the flow approach. 
\end{remark}

\begin{lemma} 
Let $\mu \in[0,1]$, $\tau_i \in \mfT$, and $\text{$\mfz$}^{\tau_i} \in \cbC[\tau_i]\,$ for $i = 1,2$. 
We can then write 
\[
\text{$\mfz$}^{\tau_2}
\graft^\mu\text{$\mfz$}^{\tau_1}=\prod_{\tau\in\mfT} \big(\text{$\mfz$}^{\tau_2}\graft^\mu\text{$\mfz$}^{\tau_1}\big)^\tau\,,
\]
where one has $\big(\text{$\mfz$}^{\tau_2}\graft^\mu\text{$\mfz$}^{\tau_1}\big)^\tau \in \cbC[\tau]$. Moreover, for any $\btau \in \tau$, $\btau_1 \in \tau_1$ and $\btau_2 \in \tau$, again writing for $n \in N(\btau_1)$ $\btau_{1,2,n} = \btau_{2} \graft_{n} \btau_{1}$, we have
\begin{equ}  
\big(\text{$\mfz$}^{\tau_2}\graft^\mu\text{$\mfz$}^{\tau_1}\big)^{\btau}
\eqdef
\iota_{\btau}
\big(\text{$\mfz$}^{\tau_2}\graft^\mu\text{$\mfz$}^{\tau_1}\big)^{\tau}
=
\sum_{\substack{n\in N(\btau_1)\\
\btau_{1,2,n}\in\tau}}\
\frac{1}{|\Aut(\tau)|} 
\sum_{ f \in \Iso(\btau_{1,2,n},\btau)}
\mathbf{f} \big( \text{$\mfz$}^{\btau_2} \graft^{\mu}_{n} \text{$\mfz$}^{\btau_1} \big)\;.
\end{equ}
\end{lemma}
\begin{proof}
From the construction of $\graft^{\mu}$, we know that
\begin{equ}\label{eq:grafting_treepiece}
\big(\text{$\mfz$}^{\tau_2}\graft^\mu\text{$\mfz$}^{\tau_1}\big)^{\tau}
=
\sum_{\substack{n\in N(\btau_1)\\
\btau_{1,2,n}\in\tau}}\pi_{\btau_{1,2,n}} 
\big( \text{$\mfz$}^{\btau_2} \graft^{\mu}_{n} \text{$\mfz$}^{\btau_1} \big)\,.
\end{equ}
Note, that, unlike what we saw in construction of $\Upsilon$, the action of $\pi_{\tau,\btau_{1,2,n}}$ may impose some symmetrization, in particular we may have $\iota_{\btau_{1,2,n}} \circ \pi_{\btau_{1,2,n}} \big( \text{$\mfz$}^{\btau_2} \graft^{\mu}_{n} \text{$\mfz$}^{\btau_1} \big) \not = \text{$\mfz$}^{\btau_2} \graft^{\mu}_{n} \text{$\mfz$}^{\btau_1}$. 

Explicitly writing this symmetrization gives the desired result. 
\end{proof}

We now suppose that we have an expansion for the effective force
\begin{equs}\label{eq:defFmu_1}
    A_{\epmu}[v](t)=\sum_{\tau\in\text{\scriptsize{$\mfT$}} } \int_{\R^{\text{\tiny{$\mfs$}}(\tau)}}\langle \text{$\mfz$}_\epmu^{\tau}(t,s^{\tau}),\Upsilon^{\tau}[v+{\tt u}](s^{\tau})\rangle_{\mcH^{\tau}}\rmd s^{\tau}\,,
\end{equs}
where $\Upsilon$ is as defined above, and we are given $\text{$\mfz$}_{\epmu} = \prod_{\tau \in \mfT} \text{$\mfz$}^\tau_{\epmu} \in \cbC$. 
Here, the sum is over unlabelled trees $\tau$. 
For the notations $\int_{\R^{\mfs(\tau)}}$, $s^{\tau}$, and $\mcH^{\tau}$, the convention here is that 
\[
\int_{\R^{\text{\scriptsize{$\mfs$}}(\tau)}}\langle \text{$\mfz$}_\epmu^{\tau}(t,s^{\tau}),\Upsilon^{\tau}[v+\text{${\tt u}$}](s^{\tau})\rangle_{\mcH^{\tau}}\rmd s^{\tau}
\equiv
\int_{\R^{N(\btau)}}\langle \text{$\mfz$}_\epmu^{\btau}(t,s_{N(\btau)}),\bUpsilon^{\tau}[v+\text{${\tt u}$}](s_{N(\btau)})\rangle_{\mcH^{\btau}}\rmd s_{N(\btau)}\;,
\]
where the choice of $\btau \in \tau$ does not matter above. \eqref{eq:defFmu_1} is therefore a rewriting \eqref{eq:defFmu_0}, but with the additional restriction that the sum over $\bT$ includes just one ``representative'' labelled tree for each unlabelled tree. 

\begin{lemma}\label{lem:unlabelled_Pol}
Expanding the non-linearity in \eqref{eq:Pol_nonlinearity} evaluated on \eqref{eq:defFmu_1} gives

\begin{equs}\label{eq:expressionDF_1}
{}&\text{$\D$}A_{\epmu}\big[v,\dot G_\mu A_{\epmu}[v] \big](t)\\
&=
\sum_{\tau \in \text{\scriptsize{$\mfT$}}}
\int_{\R^{\text{\tiny{$\mfs$}}(\tau)}}
\Big\langle \sum_{\tau_1, \tau_2 \in \text{\scriptsize{$\mfT$}}} 
\big(\text{$\mfz$}^{\tau_2}_\epmu\graft^\mu \text{$\mfz$}^{\tau_1}_\epmu \big)^\tau  (t,s^{\tau}),\Upsilon^{\tau}[v+{\tt u}](s^{\tau}) \Big\rangle_{\mcH^{\tau}} \rmd s^{\tau}\,.
\end{equs} 

\end{lemma}
\begin{proof}
For every $\tau \in \mfT$ we choose a fixed representative $\btau \in \tau$, and let $\hat{\bT}$ be the collection of such representatives. 
We can then rewrite \eqref{eq:defFmu_1} in the form \eqref{eq:defFmu_0} with the convention that for any $ \btau' \in \tau \in \mfT$, we have\footnote{Note that here $\text{$\mfz$}_{\epmu}^{\btau'}$ does not denote the $\btau'$-component of the section $\text{$\mfz$}^{\tau} \in \cbC[\tau]$.} 
\[
\text{$\mfz$}_{\epmu}^{\btau'} =
\begin{cases}
\iota_{\btau} \text{$\mfz$}_{\epmu}^{\tau} & \text{ if }\btau' = \btau\;,\\
0 & \text{otherwise}\;. 
\end{cases}
\] 
Using the notation $\btau_{1,2,n} = \btau_{2} \graft_{n} \btau_{1}$ used earlier, we have:
\begin{equs} 
{}&\text{$\D$}A_{\epmu}\big[v,\dot G_\mu A_{\epmu}[v] \big](t)\\
&=
\sum_{\substack{\tau \in  \text{\scriptsize{$\mfT$}}\\\btau_1, \btau_2 \in  \text{\scriptsize{$\hat\bT$}}}} 
\sum_{\substack{n\in N(\btau_1)\\
\btau_{1,2,n}\in\tau}} 
\int_{\R^{N(\btau_{1,2,n})}}
\Big\langle 
\big( \text{$\mfz$}^{\btau_2}_\epmu \graft_{n}^{\mu} \text{$\mfz$}^{\btau_1}_\epmu \big)  (t,s_{N(\btau_{1,2,n})}),\Upsilon^{\btau_{1,2,n}}[v+{\tt u}](s_{N(\btau_{1,2,n})}) \Big\rangle_{\mcH^{\btau_{1,2,n}}} \rmd s_{N(\btau_{1,2,n})}\\
&=
\sum_{
\substack{ \tau \in  \text{\scriptsize{$\mfT$}} \\
\btau_1, \btau_2 \in  \text{\scriptsize{$\hat\bT$}}}} 
\sum_{\substack{n\in N(\btau_1)\\
\btau_{1,2,n}\in\tau}} 
\int_{\R^{N(\btau_{1,2,n})}}
\Big\langle 
\iota_{\btau_{1,2,n}}
\circ 
\pi_{\btau_{1,2,n}}
\big( \text{$\mfz$}^{\btau_2}_\epmu \graft_{n}^{\mu} \text{$\mfz$}^{\btau_1}_\epmu \big)  (t,s_{N(\btau_{1,2,n})}),\Upsilon^{\btau_{1,2,n}}[v+{\tt u}](s_{N(\btau_{1,2,n})}) \Big\rangle_{\mcH^{\btau_{1,2,n}}} \rmd s_{N(\btau_{1,2,n})}\\
&=
\sum_{
\substack{ \tau \in  \text{\scriptsize{$\mfT$}} \\
\btau_1, \btau_2 \in  \text{\scriptsize{$\hat\bT$}}}} 
\int_{\R^{\text{\scriptsize{$\mfs$}}(\tau)}}
\Big\langle 
\sum_{\substack{n\in N(\btau_1)\\
\btau_{1,2,n}\in\tau}} 
\pi_{\btau_{1,2,n}}
\big( \text{$\mfz$}^{\btau_2}_\epmu \graft_{n}^{\mu} \text{$\mfz$}^{\btau_1}_\epmu \big)  (t,s^{\tau}),\Upsilon^{\tau}[v+{\tt u}](s^{\tau}) \Big\rangle_{\mcH^{\tau}} \rmd s^{\tau}\\
&=
\sum_{
\substack{ \tau \in  \text{\scriptsize{$\mfT$}} \\
\btau_1, \btau_2 \in  \text{\scriptsize{$\hat\bT$}}}} 
\int_{\R^{\text{\scriptsize{$\mfs$}}(\tau)}}
\Big\langle 
\big( \text{$\mfz$}^{\tau_2}_\epmu \graft^{\mu} \text{$\mfz$}^{\tau_1}_\epmu \big)^{\tau} (t,s^{\tau}),\Upsilon^{\tau}[v+{\tt u}](s^{\tau}) \Big\rangle_{\mcH^{\tau}} \rmd s^{\tau}\;.
\end{equs}
The first equality is from the labelled case, where we can interchange between $\bUpsilon$ and $\Upsilon$ thanks to Lemma~\ref{lem:already_sym}. 
In the second equality, we can insert a symmetrization thanks to the symmetry of $\Upsilon^{\btau_{1,2,n}}$. 
For the third equality, we note that we are integrating labelling invariant objects, so the integral clearly does not depend on the labeling of $\btau_{1,2,n} \in \tau$. 
In the fourth equality, we used \eqref{eq:grafting_treepiece}. 
Replacing the sum over $\btau_1, \btau_2 \in \hat{\bT}$ with a sum over $\tau_1, \tau_2 \in \mfT$ in the last line finally gives the desired result. 
\end{proof}

\subsection{Truncation}\label{subsec:trunc}


\begin{definition}\label{def:trees}
We associate to any $\btau  \in \bT$, an \textit{order} $\mfo(\btau) \eqdef |E(\btau)|$, a \textit{size} $\mfs(\btau) \eqdef |N(\btau)| = |E(\btau)|+1$, and a \textit{scaling} $|\btau|$ defined by
    \begin{equs}
        |\btau|\eqdef -1+\mfs(\btau)H\equiv -1+(\mfo(\btau)+1)H\,.
    \end{equs}
All these quantities are well-defined for elements of $\mfT$ as well and we overload notation accordingly. 
    
For $k\in\N$, we denote by $\mfT_{\leqslant k}$ the set of all trees $\tau\in\mfT$ with $\mfo(\tau)\leqslant k$, and let
\begin{equs}\label{eq:trees}
    \mcT\eqdef\mfT_{\leqslant 2}\equiv\Big\{ \noise,\tOne,\tTwo,\tThree\Big\}\,,
\end{equs}    
where in the graphical representation for trees, we always draw the root as the bottom node. 

We also define 
\begin{equs}
    \mcT_*\eqdef \{\bignoise\}\cup
    \Bigg\{\tau\in\mfT_{\leqslant5} :
\begin{array}{c}
\exists\, \btau_1\in\tau_1\in\mcT\,,\;\btau_2\in\tau_2\in\mcT,\text{ and}\;n\in N(\btau_1)\\
\text{such that}\;\btau_2\graft_n\btau_1\in\tau
\end{array}
    \Bigg\}\,.
\end{equs}
Observe that $\mcT\subset\mcT_*$ and
\begin{equs}\label{eq:def_delta}
    \tau\in\mcT_*\setminus\mcT\Rightarrow\mfo(\tau)\geqslant3\Rightarrow|\tau|\geqslant\delta\eqdef-1+4H>0\,.
\end{equs}
\end{definition}

\begin{remark}
The key input for our method will come from estimates on the force coefficients $(\zeta_\epmu^\tau)_{\epmu\in(0,1]}^{\tau\in\mcT}$,  which we will obtain by post-processing\footnote{Note, we could have also directly obtained stochastic estimates on $(\zeta_\epmu^\tau)_{\epmu\in(0,1]}^{\tau\in\mcT}$, see Remark~\ref{rem:2_24}. } of stochastic estimates on the stationary force coefficients $(\xi_\epmu^\tau)_{\epmu\in(0,1]}^{\tau\in\mcT}$.   

Note that the set $\mcT$ also indexes the components of the branched rough path \cite{Trees} over fBM with Hurst index $H > \frac{1}{4}$, with the components representing iterated integrals. 
Writing $W(t) = W_{t}$, we would have:
\begin{equs}\label{eq:increments}
{}&W^{\noise}_{s,t} = \int_{s}^{t} \rmd W_{u} = W_{t} - W_{s}\,,
\quad
W^{\tOnesmall}_{s,t} = \int_{s}^{t}  W^{\noise}_{s,u} \otimes \rmd W_{u}\,,\\
{}& 
W^{\tTwosmall}_{s,t}
=
\int_{s}^{t}  W^{\tOnesmall}_{s,u} \otimes \rmd W_{u}\,,
\quad
W^{\tThreesmall}_{s,t}
=
\int_{s}^{t}  W^{\noise}_{s,u} \otimes W^{\noise}_{s,u} \otimes \rmd W_{u}\,.
\end{equs}
In the setting of branched rough paths, the objects in \eqref{eq:increments} serve as generalised increments that are used in local expansions of the solution to a rough differential equation.   
Increments play an important role because these capture the power-counting of objects of H\"{o}lder regularity $ \alpha \in (0,1)$. 
Though the force coefficients $(\zeta_{0,\mu}^\tau)_{\mu\in(0,1]}^{\tau\in\mcT}$ are not directly increments, they do have similarities to the objects in \eqref{eq:increments}

Take for instance $\tau=\tOne$, we have, for $t\in[0,\mu]$, and using the support properties of $G_\mu$, 
\begin{equs}
  \int_{\R^2}  \zeta^{\tOnesmall}_{0,\mu}(t,s^{\tOnesmall}_1,s^{\tOnesmall}_2)\rmd s^{\tOnesmall}_1\rmd s^{\tOnesmall}_2= \xi(t)(G-G_\mu)\big(\1_{\geqslant0}\xi\big)(t)=\xi(t)\int_0^t\rmd W_s=W^{\noise}_{0,t}\otimes \rmd W_t\,,
\end{equs}
which looks like (the derivative of) $W^{\tOnesmall}_{0,t}$.

However, $\zeta^{\tOnesmall}_{0,\mu}$ is not an increment process since we have fixed $s=0$, but in the flow approach the effective scale $\mu$ gives the desired power-counting without taking increments \footnote{In particular, for $t\geqslant2\mu$, $(G-G_\mu)\big(\1_{\geqslant0}\xi\big)=0$.}.

A major difference between the flow approach versus branched rough paths is that the framework of branched rough paths requires important algebraic relations between the components of \eqref{eq:increments} which are related to them being increments. 
These do not appear to have any analogue in the flow approach.

\end{remark}
\begin{remark}
The only trees $\tau$ in $\mcT_*$ whose elements have non-trivial automorphism group are $\tThree$ and the trees appearing in $\tThree\graft\sigma$ for $\sigma\in\mcT$. 
All of these mentioned trees $\tau$ have $\Aut(\btau)=\Z_2$. 
\end{remark}

In analogy with \eqref{eq:defFmu_1}, our ansatz for the truncated effective force $A_\epmu$ is given by 
\begin{equs}\label{eq:defFmu}
     A_{\epmu}[v](t)=\sum_{\tau\in \mcT} \int_{\R^{\mfs(\tau)}}\langle \text{$\mfz$}_\epmu^{\tau}(t,s^{\tau}),\Upsilon^{\tau}[v+{\tt u}](s^{\tau})\rangle_{\mcH^{\tau}}\rmd s^{\tau}\,.
\end{equs}
We work out how this truncation affects the system \eqref{eq:sys1}.
\begin{definition}\label{def:proj}
Given $\tt u\in\R$ and some force coefficients $(\mfz^\tau)^{\tau\in\mcT_*}$, we define the projector $\Pi$ acting on functionals $L:C^\infty([0,1],\R^\cbn)\longrightarrow C^\infty([0,1],\R^\cbn)$ of the form 
\begin{equs}\label{eq:formL}
    L[v](t)=\sum_{\tau\in\mcT_*}\int_{\R^{\mfs(\tau)}}\langle \mfz^{\tau}(t,s^{\tau}),\Upsilon^{\tau}[v+{\tt u}](s^{\tau})\rangle_{\mcH^{\tau}}\rmd s^{\tau}\,,
\end{equs}
by projection onto $\mcT$. More precisely, we have
\begin{equs}
    \Pi L[v](t)\eqdef\sum_{\tau\in \mcT} \int_{\R^{\mfs(\tau)}}\langle \mfz^{\tau}(t,s^{\tau}),\Upsilon^{\tau}[v+{\tt u}](s^{\tau})\rangle_{\mcH^{\tau}}\rmd s^{\tau}\,.
\end{equs}
Finally, we also define an operator $\Pol_\mu$ acting on families $(L_\mu)_{\mu\in(0,1]}$ of functionals of the form \eqref{eq:formL} -- where the $(\mfz^\tau)^{\tau\in\mcT_*}$ is replaced with $(\mfz^\tau_\mu)^{\tau\in\mcT_*}_{\mu \in (0,1]}$ -- by
\begin{equs}\label{eq:defPol}
    \Pol_{\mu}(L_\mu)[\bigcdot]\eqdef\Pi\big(\d_\mu L_\mu[\bigcdot]+ \D L_\mu\big[\bigcdot,\dot G_\mu L_{\mu}[\bigcdot]\big]\big)\,.
\end{equs}
Note that $\Pol_{\mu}$ is a truncated version of the Polchinski flow equation. 
\end{definition}
Our aim is to construct the stationary effective force $S_\epmu$ such that it solves the equation 
\begin{equs}\
   \Pol_{\mu}(S_\epmu)=0 \label{eq:Pol} 
\end{equs}
 with initial condition $S_{\eps,0}[\bigcdot]=S_\eps[\bigcdot]= Z_\eps[\bigcdot+{\tt u}]$. Once this is done, we aim to use $S_\epmu$ is order to construct an effective force $F_\epmu$ solving $\Pol_\mu(F_\epmu)=0$ with initial condition $F_{\eps,0}[\bigcdot]=F_{\eps}[\bigcdot]=\1_{\geqslant0}S_\eps[\bigcdot]=\1_{\geqslant0}Z_\eps[\bigcdot+\tt u]$. Provided this condition is satisfied, the system~\eqref{eq:sys1}
 rewrites as  \begin{subequations}\label{eq:sys2}
  \begin{empheq}[left=\empheqlbrace]{alignat=2}
  \label{eq:eqRwelldefined}  
\d_\mu R_{\eps,\mu}&=-\D F_{\eps,\mu}[v_{\eps,\mu},\dot G_\mu R_{\epmu}]-I_\epmu[v_\epmu]\,,\\
        v_\epmu&= -\int_\mu^{T}\dot G_\nu\big(F_\epnu[v_\epnu]+R_\epnu\big)\rmd\nu\,,\label{eq:vphimu}
  \end{empheq}
\end{subequations}\auth{comment 2}
where the pair $(R_{\eps,\bigcdot},v_{\eps,\bigcdot})$ is supported on $[0,T]$ and we introduced
\begin{equs}\label{eq:expL}
   I_\epmu[\bigcdot]\eqdef(1-\Pi)\big( \D  F_\epmu\big[\bigcdot,\dot G_\mu F_{\epmu}[\bigcdot]\big]\big)\,,
\end{equs}
and used the fact that $F_\epmu=\Pi F_\epmu$.

Recall that \eqref{eq:sys2} is a reformulation of \eqref{eq:eq1} on $[0,T]$ - in Section~\ref{sec:deter} we show that for sufficiently small $T$, it can be controlled uniformly in $\eps$. 

We introduce one more piece of notation and then give our truncated flow equation.  
\begin{definition}\label{def:defBmu} 
    For every $\tau\in\mcT_*$, we define 
    \begin{equs}
        \Ind(\tau)\eqdef\Big\{(\tau_1,\tau_2)\in\mcT^2 :\exists\,\btau_i\in\tau_i\in\mcT\,,\;\exists\,n\in N(\btau_1)\,,\;\btau_2\graft_n\btau_1\in\tau\Big\}
    \end{equs}
\end{definition}

  


\begin{lemma}
For $A_{\epmu}$ given by the ansatz \eqref{eq:defFmu}, we have
\begin{equs}\label{eq:rhsPol}
  {}& \text{$\D$} A_{\epmu}\big[v,\dot G_\mu A_{\epmu}[v]\big](t)=  
  \sum_{\substack{ \tau\in\text{\scriptsize{$\mcT_{*}$}} \\ (\tau_1,\tau_2)\in\text{\scriptsize{$\mathrm{Ind}(\tau)$}}}} \int_{\R^{\text{\tiny{$\mfs$}}(\btau)}}
   \langle \big(\text{$\mfz$}_\epmu^{\tau_2} \,\graft^\mu\text{$\mfz$}_\epmu^{\tau_1}\big)^{\tau}(t,s^{\tau}),\Upsilon^{\tau}[v+{\tt u}](s^{\tau})\rangle_{\mcH^{\tau}}\rmd s^{\tau}\,.\textcolor{white}{blablablabl}
\end{equs}
In particular, the definition of $\Pi$ combined with \eqref{eq:rhsPol} implies that \eqref{eq:Pol} is equivalent to the following system of equations for the stationary force coefficients for $\tau \in \mcT$
\begin{equs}\label{eq:flow2}
   \d_\mu\text{$\mfz$}_\epmu^{\tau}(t,s^\tau)&=-\sum_{\substack{(\tau_1,\tau_2)\in\text{\scriptsize{$\mathrm{Ind}(\tau)$}}}}    \big(\text{$\mfz$}_\epmu^{\tau_2}   \,\graft^\mu\text{$\mfz$}_\epmu^{\tau_1}\big)^{\tau}(t,s^\tau)\,.
     \end{equs}
\end{lemma}
We have the following estimate on the right hand side of \eqref{eq:flow2}. 
\begin{lemma}
Fix $\mu \in[0,1]$, $\tau_i \in \mfT$, and $\text{$\mfz$}^{\tau_i} \in \cbC[\tau_i]\,$ for $i = 1,2$. The most important property of the grafting $\big(\text{$\mfz$}^{\tau_2}\graft^\mu\text{$\mfz$}^{\tau_1}\big)^\tau$, upon which we rely extensively, is that for every $N\in\N$, uniformly in $\text{$\mfz$}^{\tau_i}$ and $\mu\in(0,1]$ we have
    \begin{equs}\label{eq:propB}
        \|K_{N,\mu}^{\otimes1+\mfs(\tau)}&\big(\text{$\mfz$}^{\tau_2}\graft^\mu\text{$\mfz$}^{\tau_1}\big)^\tau\|_{L^\infty_t L^1_{s^\tau}}\\&\lesssim \|   \mcP_{2N,\mu} \dot G_\mu\|_{\mcL^{\infty,\infty}}
\|K_{N,\mu}^{\otimes1+\mfs(\tau_1)}\text{$\mfz$}^{\tau_1}\|_{L^\infty_t L^1_{s^{\tau_1}}}
\|K_{N,\mu}^{\otimes1+\mfs(\tau_1)}\text{$\mfz$}^{\tau_2}\|_{L^\infty_t L^1_{s^{\tau_2}}}    \,.\textcolor{white}{a}    
    \end{equs}
\end{lemma}
\begin{proof}
The proof follows from the Definition given in \eqref{eq:defB} of the grafting operation. Writing $\hat N(\btau_1)\eqdef N(\btau_1)\setminus\{n\}$, we have that the convolution of the RHS of \eqref{eq:defB} with $K_{N,\mu}^{\otimes1+\mfs(\tau)}$ is equal to 
\begin{equs}
 {} &\int_{-\infty}^1K_{N,\mu}(s_n-v)\big(K_{N,\mu}^{\otimes\mfs(\tau_1)}\otimes\Id\big)   \text{$\mfz$}^{\btau_1}(t,s_{\hat N(\btau_1)},v)\int_{-\infty}^1\mcP_{N,\mu} \dot G_\mu(v-w )K_{N,\mu}^{\otimes\mfs(\tau_2)+1}(\bell_{\btau_1,\btau_2}\text{$\mfz$}^{\btau_2})(w,s_{N({\btau_2}+\ell_{\btau_1,\btau_2})}) \rmd w\rmd v
\\     
&=\int_{-\infty}^1K_{N,\mu}^{\otimes\mfs(\tau_1)+1}   \text{$\mfz$}^{\btau_1}(t,s_{\hat N(\btau_1)},v)  \mcP_{N,\mu}^\dagger \big(K_{N,\mu}(s_n-\bigcdot)\int_{-\infty}^1
\mcP_{N,\mu}\dot G_\mu( \bigcdot-w)
K_{N,\mu}^{\otimes\mfs(\tau_2)+1}(\bell_{\btau_1,\btau_2}\text{$\mfz$}^{\btau_2})(w,s_{N({\btau_2}+\ell_{\btau_1,\btau_2})})\big)(v)\rmd w  \rmd v\,.
    \end{equs}
When the operator $\big(\mcP_\mu^{N}\big)^\dagger$ hits the kernel of $K_{N,\mu}$
this can create some space-time derivatives of $K_{N,\mu}$ multiplied by $\mu$. Then, by \eqref{eq:spacederivKmu}, these newly created kernels are still in $L^1$ uniformly in $\mu$. This means that there exist some kernels $\big(A_\mu^{(i)}:i\leqslant N\big)$ belonging to $L^1$ uniformly in $\mu$ such that the last line of the above equation is equal to 
    \begin{equs}        \int_{-\infty}^1K_{N,\mu}^{\otimes\mfs(\tau_1)+1}   \text{$\mfz$}^{\btau_1}(t,s_{\hat N(\btau_1)},v)   \sum_{i\leqslant N} A_\mu^{(i)}(s_n-v)\int_{-\infty}^1
\mu^{i}\partial^{i}\mcP_{N,\mu}\dot G_\mu( v-w)
K_{N,\mu}^{\otimes\mfs(\tau_2)+1}(\bell_{\btau_1,\btau_2}\text{$\mfz$}^{\btau_2})(w,s_{N({\btau_2}+\ell_{\btau_1,\btau_2})})\rmd w  \rmd v\,,
    \end{equs}
which concludes the proof.   
\end{proof}

\begin{remark}\label{rem:flow_remark}
The system \eqref{eq:flow2} is hierarchical in the order of the tree, since, for $(\tau_1,\tau_2) \in \Ind(\tau)$, $\mfo(\tau_1)\vee\mfo(\tau_2)\leqslant\mfo(\tau)-1$.
We also have $\Ind(\bignoise)=\emptyset$, so that $\partial_{\mu} \text{$\mfz$}_\epmu^{\noise}(t,s) = 0$ -- that is the noise does not flow. 
\end{remark}

\begin{remark}  \label{rem:2_24}       
Written in terms of labelled trees, \eqref{eq:flow2} reads
  \begin{equs}  \label{eq:flowLabelled}       
 \d_\mu\text{$\mfz$}^{\btau}_\epmu&(t,s_{N(\btau)}) \\&=-\frac{1}{|\Aut(\tau)|}\sum_{\substack{(\tau_1,\tau_2)\in\text{\scriptsize{$\mathrm{Ind}(\tau)$}}}}  
 \sum_{\substack{n\in N(\btau_1)\\
\btau_{1,2,n}=\btau}}
\sum_{ f \in \Aut(\btau)}
 \int_{-\infty}^1\dot G_\mu(s_n-w)\text{$\mathbf{M}$}^{\text{\scriptsize{$\bff$}}}_{\btau_1,\btau_2}\bigotimes_{i=1}^2\text{$\mfz$}_\epmu^{\btau_i}(t_{i},s_{{N(\btau_i+\ell_{\btau_1,\btau_2,i})}})\rmd w
    \,.
\end{equs}
where the choice of $\btau_i\in\tau_i$ for $i\in[2]$ is arbitrary, and where we use the notation
 \begin{equs}       
    \ell_{\btau_1,\btau_2,i}\eqdef \begin{cases}  
    0 & \text{for} \;i= 1\,, \\
   \ell_{\btau_1,\btau_2} & \text{for} \;i= 2\,, \\
\end{cases}
           \quad\text{and}\quad
           t_i \equiv t_i(t,w)\eqdef \begin{cases}  
    t & \text{for} \;i= 1\,, \\
      w & \text{for} \;i= 2\,, \\
\end{cases}
   \end{equs}
and the map 
   \begin{equs}       \text{$\mathbf{M}$}^{\text{\scriptsize{$\bff$}}}_{\btau_1,\btau_2}\eqdef\text{$\mathbf{f}$}  \circ
 \Big(   \Id_{\text{\scriptsize{$\cbC$}}[\btau_1]} \otimes
       \text{$\bell$}_{\btau_1,\btau_2}\Big) : \bigotimes_{i=1}^2\text{$\cbC$}[\btau_i]\rightarrow\text{$\cbC$}[\btau]\,.
       \end{equs}
Note that $\text{$\mathbf{M}$}^{\text{\scriptsize{$\bff$}}}_{\btau_1,\btau_2}$ 
is well-defined, since $\text{$\cbC$}[\btau_{1,2,n}]= \text{$\cbC$}[\btau_1]\otimes\text{$\cbC$}[\btau_2+\ell_{\btau_1,\btau_2}]$.
 \end{remark}

An immediate consequence of the flow equation \eqref{eq:flow2} is the following support property for force coefficients.
\begin{lemma}\label{lem:support} 
Assume that the force coefficient $\text{$\mfz$}_\epmu^{\noise}(t,s)$ is supported on the diagonal $\{t=s\}$. Then, for every $\tau\in\mcT$ and $\epmu\in(0,1]$, we have the following support property of $\text{$\mfz$}^{\tau}_\epmu$:   
\begin{equs}    \text{$\mfz$}^\tau_\epmu(t,s^\tau)=0\;\text{$\mathrm{if}$}\; s^\tau\notin[t-2\mu\mfo(\tau),t]^{\mfs(\tau)}\,.\label{eq:supportXi}
\end{equs}
\end{lemma}
\begin{proof}
Noting Remark~\ref{rem:flow_remark}, we argue by induction on the order of $\tau$, with the base case $\tau = \bignoise$ being immediate since it does not flow with $\mu$. 
The inductive step follows from the support properties of $\dot G_\mu$ stated  in \eqref{eq:suppGmu}. 
\end{proof}

\subsection{Norms on coordinates and force coefficients}
Following Remark~\ref{rem:not_smooth}, we now move to a rough setting. 
From this point on in our argument, we assume that $V \in C^{9}$

Next, for $\btau \in \tau \in  \mcT_*$ , we define $\widecheck{\cbB}[\btau]$ to be the closure of $\cbB[\btau]$ in the $C^6$ norm. As in the smooth case, we define $\widecheck{\cbB}[\tau]$ in terms of sections living in $\widecheck{\cbB}[\btau]$ with $\btau \in \tau$ analogously to what was done in the smooth case in the last section. 
By our assumption that $V \in C^{9}$ and the definition of $\mcT_*$, we certainly have, for $\tau \in \mcT_{*}$, $\Upsilon^{\tau} \in \widecheck{\cbB}[\tau]$. 
We now introduce norms on the $\cbC$-spaces. 

\begin{definition}
   For every $N\geqslant1$, $\tau\in\mcT_*$ and $\text{$\mfz$}^{\tau} \in \cbC[\tau]$, and  $\mu\in(0,1]$, set
    \begin{equs}        K^{\otimes 1+\mfs(\tau)}_{N,\mu}\text{$\mfz$}^\tau(t,s^\tau)\eqdef \big(K_{N,\mu}\otimes\dots\otimes K_{N,\mu}\big)*\text{$\mfz$}^\tau(t,s^\tau)\,.
    \end{equs}
We then endow $\text{$\mfz$}^\tau$ with the norm
    \begin{equs}     \label{eq:defTripleNorm}  \nnorm{\text{$\mfz$}^\tau}_{N,\mu} \eqdef \norm{K^{\otimes1+\mfs(\tau)}_{N,\mu}\text{$\mfz$}^\tau}_{L^\infty_tL^1_{s^\tau}((-\infty,1]\times\R^{\mfs(\tau)})}\equiv \sup_{t\leqslant1}\int_{\R^{\mfs(\tau)}} |K^{\otimes1+\mfs(\tau)}_{N,\mu} \text{$\mfz$}^{\tau}(t,s^\tau)|\rmd s^\tau\,.\textcolor{white}{blabl}
    \end{equs}
\end{definition}

We will later make a global choice of $N \geqslant 1$ in our argument, but $\mu$ will be a running scale parameter. 
For $\btau \in \tau \in \mcT_{*}$ we define $\check{\cbC}_{\mu}[\btau]$ to be the closure of $\check{\cbC}_{\mu}[\btau]$ in the topology $\nnorm{\bigcdot}_{N,\mu}$ and similarly define $\check{\cbC}_{\mu}[\tau]$ analogously to what was done in the smooth case. 
 
While the norm $\nnorm{\bigcdot}_{N,\mu}$ depends on the effective scale $\mu$, along with the space $\check{\cbC}_{\mu}[\tau]$, we are almost always working with some $\text{$\mfz$}^\tau_\epmu$ which comes with an implied scale $\mu$ already. 
Therefore we often drop the scale $\mu$ from some subscripts, just writing $\nnorm{\text{$\mfz$}^\tau_\epmu}_{N}$. 

We also define for $N\in\N_{\geqslant1}$ and $c>0$, 
\begin{equs}
    \nnnorm{\text{$\mfz$}}_{N}
\eqdef
\max_{\tau\in\mcT}\sup_{\eps,\mu\in(0,1]}\mu^{-|\tau|+\eta}\nnorm{\text{$\mfz$}^\tau_{\eps,\mu}}_{N,\mu}\quad\text{and}\quad
\nnnorm{\text{$\mfz$}}_{N,c}
\eqdef
\max_{\tau\in\mcT}\sup_{\eps,\mu\in(0,1]}(\eps\vee\eps')^{-c}\mu^{-|\tau|+\eta}\nnorm{\text{$\mfz$}^\tau_{\eps,\mu}-\text{$\mfz$}^\tau_{\eps',\mu}}_{N,\mu}\,.
\end{equs}

\subsection{Force estimates and remainder construction} \label{sec:2_4}
\auth{comment 3}
As alluded to in Remark~\ref{remark:two_forces}, we will be interested in two families of (truncated) force coefficients $(\xi_\epmu^{\tau})^{\tau\in\mcT}_{\epmu\in(0,1]}$ and $(\zeta_\epmu^{\tau})^{\tau\in\mcT}_{\epmu\in(0,1]}$ which are two solutions to the force coefficient flow equation \eqref{eq:flow2} subject to the two sets of initial data:
\begin{equs}\label{eq:noise_initialisation}
\begin{aligned}
\xi_{\eps,0}^{\circ}(t,s) = \delta(t-s)\xi_{\eps}(s)
\quad&\text{ and }\quad
\xi_{\eps,0}^{\tau}(t,s) = 0
\text{ for } \tau \in \CT \setminus \{\circ\}
\,,\\
\zeta_{\eps,0}^{\circ}(t,s)  = 
\delta(t-s)
\mathbf{1}_{\geqslant0}(s) \xi_{\eps}(s)
\quad&\text{ and }\quad
\zeta_{\eps,0}^{\tau}(t,s) = 0
\text{ for } \tau \in \CT \setminus \{\circ\}\,.
\end{aligned}
\end{equs}
We remark again that the initialisation for $\zeta_\epmu^{\tau}$ is the correct one for solving \eqref{eq:eq1}, namely $\zeta_\epmu^{\tau}$ is an enhancement of the derivative of a mollified fBM (where the fBM vanishes for negative times) and $\xi_{\eps,\mu}^{\tau}$ is an enhancement of a ``stationary, two sided'' analogue of this same derivative. 

\begin{remark}\label{eq:no_renormalisation}
The fact that the two initialisations \eqref{eq:noise_initialisation} vanish for $\tau \not = \circ$ corresponds to the fact that we implement no renormalisation in our equation.

In the flow approach, the need for renormalisation of a force coefficient associated to some $\tau$ is put in by hand as an $\eps$-dependent choice of the initial condition $\zeta_{\eps,0}^{\tau}$ \dash changing this initial condition induces a change in the equation, i.e. renormalisation. 

We can get away without doing this as we can show that expectation of $\zeta^{\tOnesmall}_{\eps,0}$, which looks like it needs to be renormalised due to power counting analysis, actually turns out to give a finite fixed contribution in the fixed point problem for the rough differential equation.  
 \end{remark}

\begin{remark}
We recall that, by Remark~\ref{rem:flow_remark} one has that, for all $\mu$, 
\begin{equ}\label{eq:noise_coeff}
\xi_{\eps,\mu}^{\circ}(t,s) = \delta(t-s)\xi_{\eps}(s)
\quad\text{ and }\quad
\zeta_{\eps,\mu}^{\circ}(t,s)  = 
\delta(t-s)
\mathbf{1}_{\geqslant0}(s) \xi_{\eps}(s)\,.
\end{equ}
\end{remark}

In Section~\ref{sec:proba} we prove the following theorem regarding the stationary force coefficients. 
\begin{theorem}\label{thm:sto} 
For every smooth, compactly supported $\omega:\R\rightarrow\R$, $P{\geqslant1}$, $\eta>0$ and $c>0$ sufficiently small, we have
\begin{equs}\label{eq:boundXi}
\E\big[ \nnnorm{\omega\xi}_{4}^P\big]^{1/P}\lesssim_{P} 1\quad\text{and}\quad\E\big[ \nnnorm{\omega\xi}_{4,c}^P\big]^{1/P}\lesssim_{P} 1\,,
\end{equs}
where we write $\omega\xi^\tau_\epmu(t,s^\tau)\eqdef\omega(t)\xi^\tau_\epmu(t,s^\tau)$
\end{theorem}
In the sequel, we only need $\omega$ to be supported on $[-1,2]$ and equal to one on $[0,1]$. Throughout the rest of the article, we always implicitly assume that these conditions are satisfied.

After turning the moment estimates of \eqref{eq:boundXi} into pathwise estimates, we can use purely\footnote{There is one tiny caveat here,  our control of the noise $\zeta_{\eps,\mu}^{\circ}$ is not obtained deterministically from control over the noise $\xi_{\eps,\mu}^{\circ}$, in fact we get control over $\zeta_{\eps,\mu}^{\circ}$ automatically since the noises do not flow in $\mu$ and it is given in terms of a derivative of the input fBM. See the proof of Theorem~\ref{thm:stoZeta} for details.} deterministic estimates to control the non-stationary force coefficients $(\zeta^\tau_\epmu)^{\tau\in\mcT}_{\epmu\in(0,1]}$. 
This gives us the following estimate.

\begin{theorem}\label{thm:stoZeta} 
Fix $\omega$ supported on $[-1,2]$ and equal to one on $[0,1]$. There exists a smooth function of polynomial growth $P:\R_{\geqslant0}^2\rightarrow\R_{\geqslant0}$ such that for every $\eta>0$ we have
\begin{equs}\label{eq:boundZeta}
\nnnorm{\zeta}_{6}\lesssim P\Big(\nnnorm{\omega\xi}_{4}
,\sup_{\epmu\in(0,1]}\mu^{2-2H+\eta}\norm{Q_\mu(\1_{\geqslant0}\xi_\eps)}_{L^\infty([0,1])}\Big)<\infty
\,.
\end{equs}
Moreover, there are limiting objects $(\zeta^\tau_{0,\mu})^{\tau\in\mcT}_{\mu\in(0,1]}$ such that $\zeta^\tau_\epmu\rightarrow\zeta^\tau_{0,\mu}$ as $\eps\downarrow0$ in the topology induced by the above norm.

Finally, note that the estimate \eqref{eq:boundZeta} combined with \eqref{eq:boundXi} implies moment bounds on the non-stationary force coefficients. 
\end{theorem}
The proof of Theorem~\ref{thm:stoZeta} is essentially the same as \cite[Theorem 10.50]{Duch21}. 
For completeness, we recall the argument in Appendix~\ref{Sec:Sec5}.

The proof of estimate \eqref{eq:boundXi} is the key ``stochastic step'', and nearly all the analysis that follows afterwards is pathwise in $\xi$.
First this deterministic analysis is used to prove Theorem~\ref{thm:stoZeta}, and then one solves the ``remainder fixed point problem'' \eqref{eq:sys2} pathwise in $\xi$. 

\begin{remark}\label{re:2_34}
Note that in our argument we could have dispensed with the introduction of the stationary force coefficients $(\xi_\epmu^{\tau})^{\tau\in\mcT}_{\epmu\in(0,1]}$ entirely and used the same cumulant flow analysis used to prove \eqref{eq:boundXi} to prove an analogous moment estimate for the $(\zeta_\epmu^{\tau})^{\tau\in\mcT}_{\epmu\in(0,1]}$, in fact a previous version of the article did this. 

This creates some changes in our argument however. 
First, the analysis to show that the expectation of $\zeta^{\tOnesmall}_{\eps,0}$ gives a finite contribution in the fixed point problem is more delicate than for $\xi^{\tOnesmall}_{\eps,0}$, we explore this in Section~\ref{sec:expectation}. 
This by itself is a minor issue, a much bigger simplification is that by having our key stochastic input involve a stationary object allows us to use a ``restarting argument'' to prove global in time well-posedness as our stochastic argument doesn't care about shifts in time. 

A minor trade-off of introducing the stationary force coefficients is that the deterministic analysis involved in going from the stationary to the non-stationary force coefficients requires us to need more regularity from $V$. 
If one works directly with the non-stationary force coefficients, then $V \in C^{7}$ suffices. 
\end{remark}

At the end of this section, for small enough $T$ and uniformly in $\eps \in (0,1]$, we solve \eqref{eq:sys2} for the remainder
$(R_{\eps,\bigcdot})_{\mu \in [0,T]}$ and, in Section~\ref{sec:deter}, we use this remainder along with the force ansatz $F_{\eps,\bigcdot}$ to construct $u_\eps$.

\begin{remark}\label{rem:2_31}
While the arguments to construct the remainder and then the solution require the probabilistic estimates~\eqref{eq:boundXi} as an input, they are absolutely deterministic, and besides the probabilistic input, they could be carried out in the entire subcritical regime, i.e. $H>0$. 

While one cannot use the cumulant flow to construct the probabilistic input (that is, the estimate \eqref{eq:boundXi}) in the regime $H \in (0,1/4]$, if one assumed the an analogue of the estimate \eqref{eq:boundXi} now also including force coefficients indexed by trees in $\mfT_{\leqslant k(H)}$ for $k(H)=\lfloor 1/H\rfloor-1$, then the key deterministic step (that is, solving  \eqref{eq:sys2}) would carry over with little change in the regime $H \in (0,1/4]$. 

However, in order to keep notation lighter - in particular, to keep our set of trees fixed and to fix some of the norms we use - we assume $H > 1/4$ in our deterministic analysis. 
\end{remark}

\begin{definition}\label{def:mcG}
For any initial condition ${\tt u}\in\R^\cbn$, we define the increasing function $\mcG_\ttu:\R_{\geqslant0}\rightarrow\R_{\geqslant0}$ by 
\begin{equs}\label{eq:defmcG}
    \mcG_\ttu:t\mapsto  \big(1+t^{36}\big)\norm{V}_{C^{9}(B_0(t+|\ttu|))}^6\,.
\end{equs}
\end{definition}
Recall that, to lighten the notation, we set $K_\mu\equiv K_{6,\mu}$ and $\mcR_\mu=\mcP_{6,\mu}$, and that $\delta=-1+4H>0$.
\begin{corollary}\label{coro:1} 
Fix $T\in(0,1]$, and two families of functions $\theta=(\theta_\mu)_{\mu\in(0,T]}$ and $\tiR=(\tiR_\mu)_{\mu\in(0,T]}$ with $\theta_\mu,\tiR_\mu\in C^\infty([0,T],\R^{\cbn})$, and which for some universal constant $C_{\theta,\tiR}>1$ satisfy, for every $\mu \in (0,T]$, 
\begin{equs}  \label{eq:eqCoro224}  \norm{\theta_\mu}_{L^\infty_{0;T}}\vee
 \mu^{-\delta/2}\norm{\tiR_\mu}_{L^\infty_{0;T}}
\leqslant C_{\theta,\tiR}\,.
\end{equs} 
Writing 
\begin{equs}
    \tF_\epmu[\bigcdot]\eqdef K_\mu F_\epmu[\bigcdot]\,,\quad \tiI_\epmu[\bigcdot]\eqdef K_\mu I_\epmu[\bigcdot]\,,\quad {and}\quad u_\mu \eqdef \theta_\mu+{\tt u}\,,
\end{equs}
we clearly have $K_\mu\theta_\mu+{\tt u}=K_\mu\theta_\mu+K_\mu\mcR_\mu{\tt u}=K_\mu u_\mu$ (since ${\tt u}$ is constant in time, so that $\mcR_\mu{\tt u}={\tt u}$) along with the estimate
\begin{equs}\label{eq:u_mu}
    \norm{u_\mu}_{L^\infty_{0;T}}\leqslant C_{\theta,\tilde R}+|{\tt u}|\,.
\end{equs}
Then, there exists universal constants $K,r\geqslant1$ such that setting $C_F\eqdef K\nnnorm{\omega\xi}^r_{4}$ we have, for every $\eta>0$, 
    \begin{equs}
     \label{eq:boundF}   \norm{ \tF_\epmu[K_\mu\theta_\mu]}_{L^\infty_{0;T}}&\leqslant C_F\mcG_\ttu(C_{\theta,\tiR})\mu^{-1+H-\eta}\,,\\
     \label{eq:boundDF}    \norm{ \D\tF_\epmu[K_\mu\theta_\mu,\mcR_\mu\dot G_\mu \tiR_{\mu}]}_{L_{0;T}^\infty}&\leqslant C_F\mcG_\ttu(C_{\theta,\tiR})\mu^{-1+H-\eta}\,,\\
   \label{eq:boundL}       \norm{\tiI_\epmu[K_\mu\theta_\mu]}_{L_{0;T}^\infty}&\leqslant C_F\mcG_\ttu(C_{\theta,\tiR})\mu^{-1+\delta-\eta}\,.
    \end{equs}
\end{corollary}
\begin{proof}
The only parts of the lemma requiring proof are the estimates \eqref{eq:boundL}, \eqref{eq:boundF}, and \eqref{eq:boundDF}, of which we only prove the hardest \eqref{eq:boundL} since \eqref{eq:boundF} and \eqref{eq:boundDF} are similar and easier.
Recalling the definition \eqref{eq:expL} of $I_\epmu$ and the ansatz \eqref{eq:defFmu} for $F_\epmu$. We place some kernels $K_\mu$ in front of all the variables $s^\tau$ of the force coefficients, which places some operators $\mcR_\mu^\dagger $ is front of all the variables of $\Upsilon^\tau$. Using H\"older's inequality we then have the bound
 \begin{equs}
\norm{\tiI_\epmu[K_\mu\theta_\mu]}_{L_{0;T}^\infty}&\lesssim
\sum_{\substack{ \tau\in\text{\scriptsize{$\mcT_{*}\setminus\mcT$}}\\(\tau_1,\tau_2)\in\text{\scriptsize{$\mathrm{Ind}(\tau)$}}}}  \norm{K^{\otimes1+\mfs(\tau)}_\mu\big( \big(\zeta_\epmu^{\tau_2}   \,\graft^\mu\zeta_\epmu^{\tau_1}\big)^{\tau}\big)}_{{L^\infty L^1([0,T]\times\R^{\mfs(\tau)})}}
\norm{\big(\mcR^\dagger_\mu\big)^{\otimes\mfs(\tau)}\Upsilon^\tau[K_\mu u_\mu]}_{L^\infty(\R^{\mfs(\tau)})}\,. 
 \end{equs}
Using \eqref{eq:defB} and \eqref{eq:propB} we have, for any $\eta>0$,
\begin{equs}    
\norm{K^{\otimes1+\mfs(\tau)}_\mu\big( \big(\zeta_\epmu^{\tau_2}   \,\graft^\mu\zeta_\epmu^{\tau_1}\big)^{\tau}\big)}_{{L^\infty L^1([0,T]\times\R^{\mfs(\tau)})}}
{}&\lesssim\nnorm{\zeta^{\tau_1}_\epmu}_{6}\norm{\mcR_\mu\mcR_\mu^\dagger\dot G_\mu}_{\mcL^{\infty,\infty}}\nnorm{\zeta^{\tau_2}_\epmu}_{6}\\
{}&\lesssim\nnnorm{\omega\xi}^r_4\mu^{|\tau_1|+|\tau_2|-\eta}=\nnnorm{\omega\xi}^r_4\mu^{|\tau|-1-\eta}\,,
\end{equs}
where in the second inequality we used the estimate \eqref{eq:boundZeta} on force coefficients, as well as the fact that the stationary force coefficients are controlled by the non-stationary force coefficients. By \eqref{eq:def_delta} we have $|\tau|\geqslant\delta$ so the desired scaling in $\mu$ of this bound follows.

We turn to the factor involving $\Upsilon^\tau$. 
For any $\tau\in\mcT_*$, $\Upsilon^\tau$ is a product of at most $6$ factors of $V$ and its derivatives - trees in $\mcT$ are of size at most $3$, so they produce trees of size at most $6$ when grafted onto themselves.
Moreover, the derivatives appearing in this product are of at most third order. 
Indeed, the element of $\mcT$ such that $\Upsilon^\tau$ contains a derivative of maximal order is $\tau=\tThree$, which contains a second derivative, and grafting can only increase the order of a derivative by $1$. 

Combining these two observations, we can now control for $i\leqslant3$ 
$  \norm{\mcR^\dagger_\mu \rmd^i V[K_\mu u_\mu]}_{L^\infty(\R)}$ using the Fa\'a di Bruno formula. To do so, observe that $\mcR^\dagger_\mu$ is a polynomial of degree 6 in $\mu\d_t$. Moreover, we have for $k\leqslant8$
\begin{equs}\label{eq:ICvanishes}
    \mu\d_t \rmd^k V[K_\mu u_\mu]=\mu\d_t\big(K_\mu\theta_\mu+\ttu\big) \rmd^{k+1} V[K_\mu u_\mu]= \mu\d_tK_\mu\theta_\mu \rmd^{k+1} V[K_\mu u_\mu]\,,
\end{equs}
so that using \eqref{eq:spacederivKmu} we end up with
\begin{equs}
    \norm{\mcR^\dagger_\mu \rmd^i V[K_\mu u_\mu]}_{L^\infty(\R)}\lesssim\norm{\theta_\mu}^6_{L^\infty_{0;T}}\norm{V}_{C^{9}(B_0(\norm{K_\mu u_\mu}_{L^\infty_{0;T}}))}\lesssim\mcG_\ttu(C_{\theta,\tiR})^{1/6}\,.\label{eq:boundUplsilon}
\end{equs}
To conclude, as mentioned above, \eqref{eq:boundF} and \eqref{eq:boundDF} are proven similarly, using the ansatz \eqref{eq:defFmu} for $F_\epmu$, the stochastic estimate \eqref{eq:boundXi}, and controlling all the $\Upsilon^\tau$ by means of \eqref{eq:boundUplsilon}.
\end{proof}
We can now use Corollary~\ref{coro:1} to argue that, for $T$ sufficiently small, the system~\eqref{eq:sys2} can be solved by a fixed-point argument.
For convenience, we use a change of variables in \eqref{eq:sys2}
\begin{equs}  (\theta_\epmu,\tiR_\epmu)\eqdef(\mcR_\mu v_\epmu , K_\mu R_\epmu)\,,
\end{equs}
so that \eqref{eq:sys2} is equivalent to
\begin{subequations}\label{eq:sys3}
  \begin{empheq}[left=\empheqlbrace]{alignat=2}  
 \tiR_{\eps,\mu}&=-\int_0^\mu \tilde K_{\mu,\nu}\big(\D \tF_{\eps,\nu}[K_\nu\theta_{\eps,\nu},\mcR_\nu\dot G_\nu \tiR_\epnu]+\tilde I_\epnu[K_\nu\theta_\epnu]\big)\rmd\nu\,,\\
        \theta_\epmu&= -\int_\mu^{T}\tilde K_{\nu,\mu}\mcR_\nu^2 \dot G_\nu\big(\tF_\epnu[K_\nu\theta_\epnu]+\tiR_\epnu\big)\rmd\nu\,,
  \end{empheq}
\end{subequations}
Here, we used the fact that $R_{\eps,0}=0$, and set, for $\lambda\geqslant\tau$, $\tilde K_{\lambda,\tau}\eqdef\mcR_\tau K_\lambda$. 
Note that by \eqref{eq:Kmunu}, for $\lambda\geqslant\tau$, $\tilde K_{\lambda,\tau}$ is a bounded operator $L^\infty\rightarrow L^\infty$. Finally, we recall that the pair $(\tiR_{\eps,\bigcdot},\theta_{\eps,\bigcdot})$ is supported on $[0,T]$.\auth{comment 2}

The following proposition solves the fixed point problem~\eqref{eq:sys3}.
\begin{proposition}\label{eq:prop1}
Fix a universal constant $C_R>0$, and define 
\begin{equ}\label{eq:size_constant}
\Xi \equiv \Xi(\xi, V,\ttu ,C_R)
\eqdef
\max \Big(
\nnnorm{\omega\xi}_{4}, C_R,
\|V\|_{C^9(B_0(C_R+|\ttu|))} \Big)\;.
\end{equ}
If there exists constants $C,p\geqslant1$ such that $T=C\Xi^{-p}$, we just write $T = T_{\Xi}$.

There exists $\tilde{T} = \tilde{T}_{\Xi} \in(0,1]$ such that, for any $T \in (0,\tilde{T}]$ and $\eps \in (0,1]$, the map
\begin{equs}
    \Phi:\binom{\theta_{\eps,\bigcdot}}{\tiR_{\eps,\bigcdot}}\mapsto\binom{\Phi^\theta_{\eps,\bigcdot}}{\Phi^{\tiR}_{\eps,\bigcdot}}(\theta_{\eps,\bigcdot},\tiR_{\eps,\bigcdot})\eqdef\binom{-\int_{\bigcdot}^{T}\tilde K_{\nu,\bigcdot}\mcR^2_{\nu} \dot G_{\nu}  \big(\tF_{\eps,\nu}[K_{\nu}\theta_{\eps,\nu}]+\tiR_{\eps,\nu}\big)\rmd\nu}{-\int_0^{\bigcdot}\tilde K_{\bigcdot,\nu}\big(
      \D \tF_{\eps,\nu}[K_{\nu}\theta_{\eps,\nu},\mcR_\nu\dot G_\nu \tiR_\epnu]+\tiI_\epnu[K_\nu\theta_\epnu]\big)
      \rmd\nu}    
 \end{equs}
is a contraction for the norm 
\begin{equs}    \nnorm{\theta_{\eps,\bigcdot},\tiR_{\eps,\bigcdot}}_T\eqdef 
    \sup_{\mu\in(0,T]}\norm{\theta_\epmu}_{L_{0;T}^\infty}\vee\sup_{\mu\in(0,T]}\mu^{-\delta/2}\norm{ \tiR_\epmu}_{L_{0;T}^\infty}\,.
\end{equs}
In particular, the system \eqref{eq:sys2} has a unique solution, denoted $(\theta_{\eps,\bigcdot},\tiR_{\eps,\bigcdot})$, satisfying the estimate
  \begin{equs}\label{eq:boundR}
        \sup_{\mu\in(0,T]}\mu^{-\delta/2}\norm{\tiR_\epmu}_{L^\infty_{0;T}}\leqslant C_R\,.
    \end{equs}
Finally, the solution $(\theta_{\eps,\bigcdot},\tiR_{\eps,\bigcdot})$ is continuous in the data $(\tF_{\eps,\bigcdot},\tiI_{\eps,\bigcdot})$.    
\end{proposition}
\begin{proof}
We show that $\Phi$ maps a ball of radius $C_{\theta,\tiR}$ into itself. 
The proof that it is a contraction, and that it is continuous in the data follows from very standard modifications of the same estimates. 
    
 Suppose that $\nnorm{\theta_{\eps,\bigcdot},\tiR_{\eps,\bigcdot}}_T\leqslant C_{\theta,\tiR}$, we aim to show that $\nnorm{\Phi(\theta_{\eps,\bigcdot},\tiR_{\eps,\bigcdot})}_T\leqslant C_{\theta,\tiR}$. 
 We first observe that the assumption $\nnorm{\theta_{\eps,\bigcdot},\tiR_{\eps,\bigcdot}}_T\leqslant C_{\theta,\tiR}$ implies that $(\theta_\epmu,\tiR_\epmu)$ satisfy the assumptions of Corollary~\ref{coro:1} with $(\theta_\mu,\tiR_\mu)=(\theta_\epmu,\tiR_\epmu)$, so that we can make use of \eqref{eq:boundF}, \eqref{eq:boundDF} and \eqref{eq:boundL}.

We first deal with the $\theta$ component, for which we have the estimates
\begin{equs}    \norm{\Phi_\epmu^\theta}_{L^\infty_{0;T}}&\lesssim
\int_\mu^{T}
\norm{\mcR^2_\nu\dot G_\nu}_{\mcL^{\infty,\infty}}\big(\norm{\tF_\epnu[K_\nu\theta_\epnu]}_{L_{0;T}^\infty}+\norm{\tiR_\epnu}_{L_{0;T}^\infty}\big)
\rmd\nu\\
&\lesssim C_F\mcG_\ttu(C_{\theta,\tiR})\int_\mu^{T}\big(\nu^{-1+H-\eta}+\nu^{\delta/2}\big)\rmd\nu\lesssim C_F\mcG_\ttu(C_{\theta,\tiR})\int_0^{T}\nu^{-1+H-\eta}\rmd\nu\lesssim 
 C_F\mcG_\ttu(C_{\theta,\tiR})T^{H-\eta}
\,.
\end{equs}
To go from the first to the second line, we bounded the $L^\infty_{0;T}$ norm of $\tiR_\epnu$ using the hypothesis that $\nnorm{\theta_{\eps,\bigcdot},\tiR_{\eps,\bigcdot}}_T\leqslant C_{\theta,\tiR}$, and we bounded the $L_{0;T}^\infty$ norm of $\tF_\epnu[K_\mu\theta_\epnu]$ using \eqref{eq:boundF}. In the second line, we used that taking $\eta$ small enough, $\nu^{-1+H-\eta}$ is integrable at $\nu=0$. 
Taking $T$ small enough, we can absorb the constant to get an estimate in terms of $C_{\theta,\tiR}$.

We turn to the $R$ component. 
Using the definition of $\Phi^{\tiR}_\epmu$, and then  \eqref{eq:boundDF} and \eqref{eq:boundL}, we have
\begin{equs}
    \norm{\Phi^{\tiR}_\epmu}_{L^\infty_{0;T}}&\lesssim  \int_0^\mu\big(
    \norm{\D \tF_{\eps,\nu}[K_\nu\theta_{\eps,\nu},\mcR_\nu\dot G_\nu \tiR_\epnu]}_{L^\infty_{0;T}}+\norm{\tiI_\epnu[K_\nu\theta_\epnu]
    }_{L^\infty_{0;T}}  \big)\rmd\nu\\
 &\lesssim C_F\mcG_\ttu(C_{\theta,\tiR})  \int_0^\mu
\big(   \nu^{-1+H-\eta}+ \nu^{-1+\delta-\eta} \big) \rmd\nu\lesssim C_F\mcG_\ttu(C_{\theta,\tiR})\mu^{\delta-\eta}\\&\lesssim C_F\mcG_\ttu(C_{\theta,\tiR})\mu^{3\delta/4}\lesssim C_F\mcG_\ttu(C_{\theta,\tiR})\mu^{\delta/2}T^{\delta/4}\,,
\end{equs}
where on the line line we took $\eta\leqslant\delta/4$. Again, by taking $T$ small enough, we can absorb the constant to get a bound in terms of $C_{\theta,\tiR}$. 
\end{proof}
\subsection{Construction of the solution and proof of Theorem~\ref{thm:main}}\label{sec:deter}
We first prove a lemma allowing us to rewrite \eqref{eq:eq1} in terms of flow data.  
\begin{lemma}
Let $T\in(0,1]$ be as in Proposition~\ref{eq:prop1}. 
Then, for any $\eps \in (0,1]$, 
if $u_\eps$ solves \eqref{eq:eq1} on $[0,T]$, then, for all $t\in[0,T]$, 
    \begin{equs}\label{eq:Phi2}
        u_\eps(t)=(G-G_{T})\big(F_{\eps,T}[0]+R_{\eps,T}\big)(t)+{\tt u}\,.
    \end{equs}
\end{lemma}
\begin{proof} 
\eqref{eq:eq1}, the definition of $F_\eps$, and  \eqref{eq:FmuRmu} for $\mu=T$ imply that 
\begin{equs}   u_{\eps}(t)&=G\big(F_\eps[v_\eps]\big)(t)+{\tt u}=G\big(F_{\eps,T}[v_{\eps,T}]+R_{\eps,T}\big)(t)+{\tt u}\,,\label{eq:somestephere}
\end{equs}
where $v_{\eps,T}$ is defined by \eqref{eq:psimu} with $\mu=T$. The support properties of $G_{T}$ (see \eqref{eq:suppGmu}) and the fact that $F_\eps$ is supported on positive times then imply that $v_{\eps,T}$ is supported outside $ [0,T]$, so that we can replace it by $0$ in \eqref{eq:somestephere}. Similarly, we have the same support property for $G_T\big(F_{\eps,T}[0]+R_{\eps,T}\big)$, which concludes the proof.
\end{proof}
Next, we need the following consequence of \eqref{eq:boundF} and \eqref{eq:boundR}.
\begin{lemma}\label{lem:Lem45} 
Pick $\varsigma\in(0,1)$ satisfying\footnote{Note the $4$ appearing below stems from \eqref{eq:fix_kernel}.}  
\begin{equs}\label{eq:constrainsigma}
    (2/\varsigma-1)^4\leqslant2\,. 
\end{equs}
Then, there exists $\tilde{T} = \tilde{T}_{\Xi} \in(0,1]$, such that, for any $T \in (0,\tilde{T}]$, $\eta > 0$,  $p\in\N$, and $\eps \in (0,1]$, 
    \begin{equs}
         \norm{K_{\varsigma^pT}\big( F_{\eps,T}[0]+R_{\eps,T}\big)}_{L^\infty_{0;T}}\leqslant 2 C_F(\varsigma^pT)^{-1+H-\eta}\mcG_\ttu(1)\,,
    \end{equs}
where $C_F$ is as per Corollary~\ref{coro:1}. Consequently, by interpolation, we have that 
\begin{equs}\label{eq:Lem45}
     \norm{K_{\mu}\big( F_{\eps,T}[0]+R_{\eps,T}\big)}_{L^\infty_{0;T}}\lesssim C_F \mcG_\ttu(1)\mu^{-1+H-\eta}
\end{equs}
uniformly in $\mu\in(0,1]$.
\end{lemma}
\begin{proof} 
We argue by induction in $p$. 
By \eqref{eq:boundF} and \eqref{eq:boundR} we have, for sufficiently small $T>0$, 
\begin{equ}
    \norm{K_{T}\big( F_{\eps,T}[0]+R_{\eps,T}\big)}_{L^\infty_{0;T}}\leqslant C_FT^{-1+H-\eta}\mcG_\ttu(0)+C_RT^{\delta/2}\leqslant2 C_FT^{-1+H-\eta}\mcG_\ttu(0)\,.
\end{equ}
We now prove the induction step, assuming that 
 \begin{equs}\label{eq:stepprooflem}
         \norm{K_{\varsigma^pT}\big( F_{\eps,T}[0]+R_{\eps,T}\big)}_{L^\infty_{0;T}}\leqslant 2 C_F(\varsigma^pT)^{-1+H-\eta}\mcG_\ttu(1)\,.
    \end{equs}
We use \eqref{eq:Kmunu} to replace $K_{\varsigma^{p}T}$ by $K_{\varsigma^{p+1}T}$ in \eqref{eq:stepprooflem}, which, thanks to our constraint \eqref{eq:constrainsigma} on $\varsigma$, gives 
\begin{equs}    \norm{K_{\varsigma^{p+1}T}\big( F_{\eps,T}[0]+R_{\eps,T}\big)}_{L^\infty_{0;T}}
 &\leqslant 4C_F(\varsigma^{p+1}T)^{-1+H-\eta}\mcG_\ttu(1)
    \,.\textcolor{white}{blb}\label{eq:intertau}
\end{equs}
 In particular, by interpolation, \eqref{eq:intertau} also holds with $\varsigma^{p+1}T$ replaced by any $\nu\in[\varsigma^{p+1}T,T]$.
 
Next, recall that, for all $t\in[0,T]$ and $\mu\in(0,T]$,
\begin{equs}\label{eq:FmuRmuBis}
   \big( F_{\eps,T}[0]+R_{\eps,T}\big)(t)=\big(F_{\eps,\varsigma^{p+1}T}[v_{\eps,\varsigma^{p+1}T}]+R_{\eps,\varsigma^{p+1}T}\big)(t)\,.
\end{equs}
Using \eqref{eq:vphimu}, \eqref{eq:FmuRmuBis} and the support properties of $G_{T}$, we obtain that 
\begin{equs}    \norm{v_{\eps,\varsigma^{p+1}T}}_{L^\infty_{0;T}}&=\norm{\big(G_{\varsigma^{p+1}T} -G_{T}\big) \big( F_{\eps,T}[0]+R_{\eps,T}\big)}_{L^\infty_{0;T}}\\
{}&\leqslant\int^{T}_{\varsigma^{p+1}T}\norm{\mcR_\nu\dot G_{\nu}}_{\mcL^{\infty,\infty}}\norm{K_\nu \big( F_{\eps,T}[0]+R_{\eps,T}\big)}_{L^\infty_{0;T}}\rmd\nu\,.
\end{equs}
Then, using \eqref{eq:heat1} and \eqref{eq:intertau} with $\varsigma^{p+1}T$ replaced by $\nu\in[\varsigma^{p+1}T,T]$ yields
\begin{equs}    \norm{v_{\eps,\varsigma^{p+1}T}}_{L^\infty_{0;T}}&\leqslant 4C_GC_F\mcG_\ttu(1)
\int^{T}_{\varsigma^{p+1}T}\nu^{-1+H-\eta}\rmd\nu \leqslant 4C_GC_{H,\eta}C_F\mcG_\ttu(1)T^{H-\eta}\,,
\end{equs}
where $C_G>0$ is the implicit constant in \eqref{eq:heat1} and $C_{H,\eta}>0$ is universal and only depends on $H,\eta$. 
By taking the scale $T$ small enough, we have 
$$ 4C_GC_F\mcG_\ttu(1)T^{H-\eta}\leqslant 1\,.$$ $\theta_{\varsigma^{p+1}T}=\mcR_{\varsigma^{p+1}T}v_{\eps,\varsigma^{p+1}T}$ thus verifies the hypothesis of Corollary~\ref{coro:1} with $C_{\theta,\tiR}=1$, and we can make use of \eqref{eq:boundF} to control $F_{\eps,\varsigma^{p+1}T}[v_{\eps,\varsigma^{p+1}T}]$ in \eqref{eq:FmuRmuBis}. 
Using \eqref{eq:boundF} and \eqref{eq:boundR}, we obtain
\begin{equs}
   {} &\norm{K_{\varsigma^{p+1}T}\big( F_{\eps,T}[0]+R_{\eps,T}\big)}_{L^\infty_{0;T}}=\norm{\tF_{\eps,\varsigma^{p+1}T}[v_{\eps,\varsigma^{p+1}T}]+\tiR_{\eps,\varsigma^{p+1}T}}_{L^\infty_{0;T}}
    \\ &\quad\leqslant    \norm{\tF_{\eps,\varsigma^{p+1}T}[v_{\eps,\varsigma^{p+1}T}]}+\norm{\tiR_{\eps,\varsigma^{p+1}T}}_{L^\infty_{0;T}}\leqslant   C_F(\varsigma^{p+1}T)^{-1+H-\eta}\mcG_\ttu(1)+C_R(\varsigma^{p+1}T)^{\delta/2}\,.
\end{equs}
Once more, taking $T$ small enough, the second term of the r.h.s. can be bounded by the first one, which yields the desired result.  
\end{proof}
We are now ready to construct the solution $u_\eps$ in the space $\mcC^{H-\eta}([0,T])$ for any $\eta>0$. 
\begin{lemma}\label{lem:solutionconstruction}
There exists $\tilde{T} = \tilde{T}_{\Xi} \in(0,1]$, such that, for any $T \in (0,\tilde{T}]$, $\eta > 0$, and $\eps \in (0,1]$, 
    \begin{equs}
\norm{u_\eps}_{\mcC^{H-\eta}([0,T])} \equiv 
\sup_{\mu\in(0,1]}  \mu^{-H+\eta}   \norm{Q_\mu u_\eps-u_\eps}_{L^\infty_{0;T}}\lesssim |\ttu|+ C_F \mcG_\ttu(1)\,.
    \end{equs}
\end{lemma}
\begin{proof}
We first consider the more difficult case of $\mu\leqslant T$.
Starting from \eqref{eq:Phi2}, we have
    \begin{equs} 
 {}&\norm{Q_\mu u_\eps-u_\eps}_{L^\infty_{0;T}}\\
 {}&\lesssim   \norm{(Q_\mu-\Id)*{\tt u}}_{L^\infty_{0;T}}+\int_0^{T}   \norm{(Q_\mu-\Id)\dot G_\nu\big(F_{\eps,T}[0]+R_{\eps,T}\big)}_{L^\infty_{0;T}}\rmd\nu  \\  
 {}&\lesssim   \norm{(Q_\mu-\Id)*{\tt u}}_{L^\infty_{0;T}}+\int_0^{\mu}  \norm{\dot G_\nu\big(F_{\eps,T}[0]+R_{\eps,T}\big)}_{L^\infty_{0;T}}\rmd\nu\\
 {}& \qquad +\int_\mu^{T}   \norm{(Q_\mu-\Id)Q_\nu\mcP_\nu\dot G_\nu\big(F_{\eps,T}[0]+R_{\eps,T}\big)}_{L^\infty_{0;T}}\rmd\nu\,.
    \end{equs}
The first term of the last line is bounded by $\mu^{H-\eta}|\ttu|$ since the constant function is certainly $(H-\eta)$-H\"{o}lder continuous. 
For the second term we use \eqref{eq:Lem45} from Lemma~\ref{lem:Lem45} and \eqref{eq:heat1} to obtain
\begin{equs}
\int_0^{\mu} &  \norm{\dot G_\nu\big(F_{\eps,T}[0]+R_{\eps,T}\big)}_{L^\infty_{0;T}}\rmd\nu \\  
    \lesssim &
      \int_0^{\mu} \norm{\mcR_\nu\dot G_\nu}_{\mcL^{\infty,\infty}_{T}}\norm{K_\nu\big(F_{\eps,T}[0]+R_{\eps,T}\big)}_{L^\infty_{0;T}}\rmd\nu \lesssim C_F \mcG_\ttu(1)
      \int_0^\mu \nu^{-1+H-\eta}\rmd\nu\lesssim C_F \mcG_\ttu(1)\mu^{H-\eta}\,.
\end{equs}
For the remaining term we use \eqref{eq:comKmuKnu}:
\begin{equs}
    \int_\mu^{T}   \norm{(Q_\mu-\Id)Q_\nu\mcP_\nu\dot G_\nu\big(F_{\eps,T}[0]+R_{\eps,T}\big)}_{L^\infty_{0;T}}\rmd\nu&\lesssim\mu\int_\mu^{T}\nu^{-1} \norm{\mcP_\nu\mcR_\nu\dot G_\nu K_\nu\big(F_{\eps,T}[0]+R_{\eps,T}\big)}_{L^\infty_{0;T}}\rmd\nu\\
    &\lesssim C_F \mcG_\ttu(1)\mu\int_\mu^{T}\nu^{-2+H-\eta}\rmd\nu  \lesssim C_F \mcG_\ttu(1)\mu^{H-\eta}\,,
\end{equs}
where on the second line we used \eqref{eq:Lem45} and \eqref{eq:heat1}.

The case $\mu\geqslant T$ is easier since we do not have to take advantage of the presence of the operator $Q_\mu-1$. Proceeding exactly as for the second term of the case $\mu\leqslant T$ yields
\begin{equs}
    \norm{Q_\mu u_\eps-u_\eps}_{L^\infty_{0;T}}&\lesssim C_F \mcG_\ttu(1) T^{H-\eta}\lesssim C_F \mcG_\ttu(1)\mu^{H-\eta}\,.
\end{equs}
\end{proof}

\begin{proof}[Theorem~\ref{thm:main}]

We first note that Lemma~\ref{lem:solutionconstruction} gives uniform in $\eps$ control over $u_\eps$ in $\mcC^{H-}([0,\tilde T])$.    
    
To show convergence of $u_\eps$ as $\eps\downarrow 0$, the key step is to show convergence in probability of the stationary force coefficients $(\xi^\tau_{\eps,\mu})^{\tau\in\mcT}_{\mu\in(0,1]}$ (and thus of the non-stationary force coefficients built from them) to a limiting $(\xi^\tau_{0,\mu})^{\tau\in\mcT}_{\mu\in(0,1]}$ in the topology given by \eqref{eq:boundXi}. This stems from the second estimate in \eqref{eq:boundXi}.

This $\eps \downarrow 0$ convergence in probability of force coefficients is passed on to the force ansatz $F_{\eps,\bigcdot}$ given in \eqref{eq:defFmu} along with $I_{\eps,\bigcdot}$, the error term generated in by our force ansatz which is given in \eqref{eq:expL} - the topology for this convergence is with respect to the estimates of Corollary~\ref{coro:1}. 

Proposition~\ref{eq:prop1} then solves a fixed point problem for the remainder field $R_{\eps,\bigcdot}$ using the data  $\tF_{\eps,\bigcdot}$ and $\tilde I_{\eps,\bigcdot}$ (built from $F_{\eps,\bigcdot}$ and $I_{\eps,\bigcdot}$). 
Thanks to continuity of the fixed point problem with respect to this data, we obtain that $R_{\eps,\bigcdot}$ converges to $R_{0,\bigcdot}$ in $L^\infty([0,\tilde T])$ as $\eps\downarrow0$.

Finally, using the convergence of $F_{\eps,\bigcdot}$ and $R_{\eps,\bigcdot}$ in Lemma~\ref{lem:solutionconstruction} gives convergence as $\eps \downarrow 0$ of $u_\eps$ to a limiting $u$ in $\mcC^{H-}([0,\tilde T])$. 
    
Finally, regarding the regularity of $V$, the bounds in Corollary~\ref{coro:1} and the definition of $\mcG$ in Definition~\ref{def:mcG} show that $V \in C^9$ suffices.    
\end{proof}

\subsection{Proof of Theorem~\ref{thm:thm2}}
This section is devoted to the proof of Theorem~\ref{thm:thm2}. We therefore restrict our attention to the case where \eqref{eq:assumpV} holds, that is $V$ and its first nine derivatives are globally bounded on $\R^\cbn$.
The key ingredient of the proof of Theorem~\ref{thm:thm2} is the fact that we have control on the time of existence that is uniform in the initial condition $\ttu$.
\begin{lemma}\label{lem:2_35}
  There exist constants $C_{\mathfrak{t}}, C_{1}>0$ and $q \geqslant 1$ such that the following holds. 
  
  For every initial condition $\ttu\in\R^\cbn$ and under assumption that \eqref{eq:assumpV} holds, the statement of Lemma~\ref{lem:solutionconstruction} holds with the existence time $\tilde{T}_{\Xi}$ replaced by  $\tilde{T}_{\hat\Xi}$, with $\tilde{T}_{\hat\Xi} = C_{\mathfrak{t}}\hat\Xi^{-p}$ for   
\begin{equ}\hat\Xi =
\max \Big(
\nnnorm{\omega\xi}_{4}, C_R,
\|V\|_{C^9(\R^\cbn)} \Big)\,.
\end{equ}
In particular, $\hat{\Xi}$ is independent of $\ttu$.

Moreover, on the event $\{T\leqslant \tilde T_{\hat\Xi}\}$, 
    \begin{equs}
       \sup_{t\in[0,T]} |{u_\eps}(t)| \leqslant |\ttu|+C_1T^{-q}\,.
    \end{equs}
Finally, in view of the $\eps \downarrow 0$ convergence in probability of $u_\eps$ to a limit $u$ stated in Theorem~\ref{thm:main}, on the event $\{T\leqslant \tilde T_{\hat\Xi}\}$, one has
    \begin{equs}\label{eq:boundsol}
  \sup_{t\in[0,T]} |{u}(t)|\leqslant    \liminf_{\eps\downarrow0} \sup_{t\in[0,T]} |{u_{\eps}}(t)| \leqslant |u(0)|+C_1T^{-q}\,.
    \end{equs} 
\end{lemma}
\begin{proof}
We crucially use the observation made in the second equality of \eqref{eq:ICvanishes}: since the initial condition is a constant in time, it vanishes when applying the chain rule. This is the step that breaks when trying to apply this argument to singular SPDEs such as gPAM -- see Remark~\ref{rem:compawithgPAM}.

Therefore, the initial condition only appears in the upper bounds on the force inside $\rmd^kV$, and under the assumption \eqref{eq:assumpV}, $\mcG_\ttu(t)$ can be bounded uniformly in $\ttu$ by $\big(1+t^{36}\big)\norm{V}_{C^9(\R^\cbn)}^6$, which establishes the first part of the lemma. 

The estimate \eqref{eq:boundsol} is an immediate consequence of Lemma~\ref{lem:solutionconstruction} and of
    \begin{equs}
        C_F\mcG_\ttu(1)\lesssim K\nnnorm{\omega\xi}_4^{r}\norm{V}_{C^9(\R^\cbn)}^6\lesssim \hat\Xi^{pq}=(C/\tilde T_{\hat\Xi})^{q}\lesssim T^{-q}\,.
    \end{equs}     
\end{proof}
With this observation in hand, we can now give a precise formulation of Theorem~\ref{thm:thm2}.
\begin{lemma}\label{lem:Thm21precise}
There exist constants $C,c > 0$ such that the following holds. 

Fix any $\mcb N\geqslant C_R\vee
\|V\|_{C^9(\R^\cbn)} $ and initial condition $\ttu \in \R^\cbn$. 
Then, on the event $\{\nnorm{\omega\xi}_4\leqslant\mcb N\}$, the solution $u$ to \eqref{eq:eq0} is global, and one has the bound 
\begin{equs}
     \sup_{t\in[0,1]}  |u(t)|\leqslant |{\tt u}|+C\mcb N^{c}\,.
\end{equs}
\end{lemma}
 \begin{proof}
In view of Lemma~\ref{lem:2_35}, on the event $\{\nnorm{\omega\xi}_4\leqslant\mcb N\}$, the solution to \eqref{eq:eq0} exists up to the time $T_{\mcb N}\eqdef C_{\mathfrak{t}}\mcN^{-p}$ for every initial condition, where $C_{\mathfrak{t}}$ is as in Lemma~\ref{lem:2_35}. We can therefore split the interval $[0,1]$ into $T_{\mcb N}^{-1}$ intervals of length $T_{\mcb N}$. Restarting from the previous final value and iteratively solving the equation on the next interval of length $T_{\mcb N}$ yields, using \eqref{eq:boundsol}
\begin{equs}
  \sup_{t\in[0,1]}  |u(t)|\leqslant |u(0)|+C_1 T_{\mcb N}^{-q-1}\,.
\end{equs}

 \end{proof}

\begin{remark}\label{rem:compawithgPAM}
    The global well-posedness result stated in Theorem~\ref{thm:thm2} stems from the fact that assuming~\eqref{eq:assumpV}, we were able to derive some bounds on the effective force that are uniform in the initial condition -- see Corollary~\ref{coro:1}. This is made possible by the conjunction of two facts: Firstly, in our ansatz for the effective force, the initial condition is located inside $V$ or one of its derivative. Secondly, while in order to control the solution it is necessary to control some time derivatives of $V^{(i)}\big(\theta(t)+\ttu\big)$, the chain rule acts nicely on on the initial condition, in the sense that it keeps it inside some derivative of $V$: 
    \begin{equs}
        \d_t V^{(i)}\big(\theta(t)+\ttu\big) =\d_t\theta(t)V^{(i+1)}\big(\theta(t)+\ttu\big)\,.
    \end{equs}
A natural question is to ask whether the same kind of argument can be used to study the generalised parabolic Anderson model 
\begin{equs}\label{eq:gPAM}
    (\d_t-\Delta)\psi=V(\psi)\xi\,,
\end{equs}
in the case where $V$ and its derivatives are uniformly bounded. Global existence of solutions to \eqref{eq:gPAM} was proved in \cite{chandra2024prioribounds2dgeneralised,shen2024globalwellposedness2dgeneralized} in the case where the regularity of the noise is in a small window below $-1$, but relying on techniques that are not uniform in the distance to criticality.

It is therefore interesting to ask if the flow approach can help give a proof holding in the full subcritical regime. It turns out that in this case the situation is much more complicated. Indeed, we consider \eqref{eq:gPAM} in \cite{gKPZFlow}, where we showed that in the ansatz for the effective force used to study \eqref{eq:gPAM}, the initial condition $\psi_0$ is still present only inside $V$ and its derivative, but that this dependence is of the form 
\begin{equs}    V^{(i)}\big(\theta(t,x)+e^{t\Delta}\psi_0(x)\big)\,.
\end{equs}
Here, when acting with $\d_t$, $\Delta$, or even $\d_t-\Delta$ on the above expression, we will get contributions coming from $\psi_0$, which rules out the possibility of extending the argument of the proof of Theorem~\ref{thm:thm2} to the case of \eqref{eq:gPAM}.
\end{remark}

\section{Probabilistic analysis}\label{sec:proba}
This section is devoted to the proof of Theorem~\ref{thm:sto}. 
The main step is an inductive estimate on the cumulants of the force coefficients which is stated in Lemma~\ref{lem:cumul}.

We do this by applying one of the major innovations of \cite{Duch21}, which is to use the Polchinski flow, in the form of \eqref{eq:flow2}, to generate a hierarchy of equations for cumulants of force coefficients. 
\auth{comment 1: Remark 3.1 was removed}

We first introduce notation for working with cumulants. 
\begin{definition}
Fix a finite subset $J\subset \N_{\geqslant1}$. 
We denote by $\mcP(J)$ the set of all partitions of $J$. For $\rho\in \mcP(J)$, we write $|\rho|$ for the number of elements of $\rho$. We denote the elements of $\rho$ by $(\rho_q)_{q\in[|\rho|]}$: they are non-empty subsets of $J$, non-overlapping, and their union is $J$.
We adopt the convention of ordering these subsets by the order of their minima, writing $\rho=(\rho_q)_{q\in[|\rho|]}$ where $k<p\Rightarrow\min\rho_k<\min\rho_p$. 

Moreover, for two finite non-empty subsets $I,J \subset \N_{\geqslant1}$ we define 
\begin{equs}
    \mcQ(I,J)\eqdef\Big\{(\pi,\rho):\rho\in\mcP(J)\;\text{and}\;\pi:I\rightarrow[|\rho|]\Big\}\,,
\end{equs}
and we adopt the convention that if $I=\emptyset$, then we set $  \mcQ(I,J)=\mcP(J)$.

Finally, given $(\pi,\rho)\in\mcQ(I,J)$, for every $q\in[|\rho|]$, we use the shorthand $\pi_q\eqdef \pi^{-1}(q)$. The $\pi_q$'s are therefore some possibly empty subsets of $I$, non-overlapping, and whose union is $I$.
\end{definition}

We let $\kappa_{|I|}\big((X_i)_{i\in I}\big)$ denote the joint cumulant of the family of random variables $(X_i)_{i\in I}$. 
In the next lemma we state a standard identity for joint cumulants that we will use in what follows. 
\begin{lemma}
Let $I, J \subset \N_{\geqslant1}$ with $|J| > 0$ and $|I| \geqslant 0$, and suppose we are given two families of random variables $(X_i)_{i\in I}$ and $(\hat{X}_j)_{j\in J}$.
Then we have
    \begin{equs}   \label{eq:relcum0} \kappa_{|I|+1}\big((X_i)_{i\in I},\prod_{j\in J} \hat{X}_j\big)=\sum_{(\pi,\rho)\in\mcQ(I,J)}\prod_{q=1}^{|\rho|}\kappa_{|\pi_q|+|\rho_q|}\big((X_i)_{i\in \pi_q},(\hat{X}_j)_{j\in \rho_q}\big)\,.
\end{equs}
\end{lemma}
Note that we will in fact want a vector form of \eqref{eq:relcum0}. 
Suppose we are given $I,J$ as above, families of finite dimensional inner product spaces $(E_{i})_{i \in I}$ and $(\bar{E}_{j})_{j \in J}$ and two families of random vectors  $(X_i)_{i\in I}$ and $(\hat{X}_j)_{j\in J}$ where for each $i \in I$, the $X_{i}$ is a random element of $E_{i}$ and for each $j \in J$ the $\hat{X}_{j}$ is a random element of $\bar{E}_{j}$. 
Then we first recall that, by multi-linearity of cumulants, the joint cumulant $\kappa_{|I|}\big( (X_{i})_{i \in I} \big)$ can be viewed as an element of $\bigotimes_{i \in I} E_{i}$. 

In particular, \eqref{eq:relcum0} can be written as 
\begin{equs}\label{eq:relcum} 
\kappa_{|I|+1}\big((X_i)_{i\in I},\bigotimes_{j\in J} \hat{X}_j\big)
=
\sum_{(\pi,\rho)\in\mcQ(I,J)}\bigotimes_{q=1}^{|\rho|}\kappa_{|\pi_q|+|\rho_q|}\big((X_i)_{i\in \pi_q},(\hat{X}_j)_{j\in \rho_q}\big)\,.
\end{equs}

We now introduce the indexing set for our cumulants.
Throughout this subsection we fix an arbitrary $P \in \N_{\geqslant 1}$, this corresponds to the $P$ taken in Theorem~\ref{thm:sto}. 

\begin{definition}\label{def:forests}
Define $\mcT_{1,2}\eqdef\mcT\times\{0,1\}\times\{0,1,2\}$, and denote by $\tilde\tau=(\tau,r,u)$ an element of $\mcT_{1,2}$. 
We introduce a set of (ordered, unlabelled) forests
\begin{equs}
\text{$   \widehat\mcT$}\eqdef
\Big\{\stau=(\tilde {\tau}_1,\dots,\tilde{\tau}_{p(\stau)})
=
\Big( (\tau_1,r_1,u_1), \dots, (\tau_{p(\stau)},r_{p(\stau)},u_{p(\stau)}) 
\Big)\in\text{$\mcT$}_1^{p(\stau)}:p(\stau)\in[P]
\Big\}\,. 
\end{equs}

For two forests $\stau,\ssigma\in\widehat\mcT$ and a tuple $I=(i_1,\dots,i_{|I|})$ consisting of distinct elements of $[p(\stau)]$, we define the forest $\stau_{I}$ by
\begin{equs}    \stau_I\eqdef\big(\tilde\tau_{i_1},\dots,\tilde\tau_{i_{|I|}}\big)\,,
\end{equs}
and we define the forest  $\stau\sqcup\ssigma$, the concatenation of $\stau$ and $\ssigma$, by
\begin{equs}    \stau\sqcup\ssigma\eqdef(\tilde\tau_1,\dots,\tilde\tau_{p(\stau)},\tilde\sigma_1,\dots,\tilde\sigma_{p(\ssigma)})\,.
\end{equs}
Fix $\stau=(\tilde\tau_1,\dots,\tilde\tau_{p(\stau)})\in\widehat\mcT$ and
suppose that for every $i\in[p(\stau)]$, we are given tuples of times $(t_{i},s^{\tau_i})\in(-\infty,1]\times\R^{\mfs(\tau_i)}$.
We then write
\begin{equs}
    t_{\stau}\eqdef (t_{\tau_1},\dots,t_{\tau_{p(\stau)}})  \in (-\infty,1]^{p(\stau)} \,,\;\text{and}\;s^{\stau}\eqdef(s^{\tau_1},\dots,s^{\tau_{p(\stau)}}) \in \R^{\mfs(\stau)} \,,
\end{equs}
where $\mfs(\stau)\eqdef\sum_{i\in[p(\stau)]}\text{$\mfs$}(\tau_i)$ is the \textit{size} of $\stau$. We also write $\mfl(\stau)\eqdef p(\stau)+\mfs(\stau)$. 

Finally, we define the \textit{scaling} of $\stau$ by
\begin{equs}    |\stau|\eqdef\big(p(\stau)-1\big)+\sum_{i=1}^{p(\stau)}|\tau_i|\,,\label{eq:defScalingCum}
\end{equs}
the \textit{order} of $\stau$ by $\text{$\mfo$}(\stau)\eqdef \sum_{i\in[p(\stau)]}\text{$\mfo$}(\tau_i) $, and set $r(\stau)\eqdef\sum_{i\in[p(\stau)]}r_i$ and $u(\stau)\eqdef\sum_{i\in[p(\stau)]}u_i$. 
\end{definition}
The individual trees $\tau_i$ in $\stau$ keep track of, with multiplicity, the particular force coefficients appearing in the cumulant. 
The indices $r_{i}$ and $u_i$ determine whether the corresponding force coefficient comes with a scale derivative in respectively $\partial_{\mu}$ and $\d_\eps$.

\begin{definition}
We introduce function spaces for our $\widehat{\CT}$-indexed cumulants. 
Analogously to how we worked with trees in Section~\ref{subsec:tree_coord}, it will be convenient to define operations at the level of labelled forests but have the cumulants themselves indexed by (unlabelled) forests. 

Given  $\stau = \big( (\tau_{i},r_{i},u_i) \big)_{i=1}^{p(\stau)} \in \text{$   \widehat\mcT$}$, a labelling $\underline{\stau}$ of $\stau$  is a tuple
$\big( (\underline{\tau}_{i},r_{i},u_i) \big)_{i=1}^{p(\stau)}$ where, for any $i \in [p(\stau)]$, we have $\underline{\tau}_{i} \in \tau_{i}$ (that is, $\underline{\tau}_{i}$ is a labelling of the tree $\tau_i$) and for any distinct $i,j \in [p(\stau)]$, 
\begin{equs}\label{eq:restrictionforest}
  N(\btau_i)\cap N(\btau_j)=\emptyset
\end{equs}
We call $\underline{\stau}$ a labelled forest, and also set 
\begin{equs}
 N(\underline\stau)\eqdef\bigcup_{i\in[p(\stau)]}N(\btau_i)\,.
\end{equs}

Note that every forest $\stau$ has at least one labelling, and we can also view elements of ${\widehat{\CT}}$ as sets of all the labelled forests $\underline{\stau} \in \stau$ that label them.

Given a labelled forest $\underline{\stau}$ as above, we define the functional space 
\begin{equs} 
    \text{$\cbC$}\text{$[\underline{\stau}]$}
\eqdef
\bigotimes_{i \in p(\stau)}
\text{$\cbC[\underline{\tau}_i]$}\,.
\end{equs}
Note that, thanks to \eqref{eq:restrictionforest}, one can really view elements of $\cbC[\underline{\stau}]$ are really functions of time variables $\big(t_{v}: v \in N(\underline\stau) \big)$. 

Analogously to what we did for labelled vs unlabelled trees, we define for each unlabelled forest $\stau \in \widehat{\CT}$ a corresponding spaces  $\cbC[{\stau}]$ defined using relabelling-invariant sections taking value in $\cbC[{\stau}]$

\end{definition}

\begin{definition}
For $N\geqslant1$ and $\mu \in (0,1]$ we write, for any  $W^{\stau} \in \cbC[\stau]$, 
\begin{equs}
K_{N,\mu}^{\otimes\text{\scriptsize{$\mfl$}} (\text{\scriptsize{$\stau$}})}    &W^{\stau}( t_\stau,s^\stau)\eqdef (K_{N,\mu}\otimes\dots\otimes K_{N,\mu})*W^{\stau}\text{$( t_\stau,s^\stau)$}
\end{equs}
for the convolution of $\text{$\Kappa$}_{\eps,\mu}^{\stau}$ with $K_{N,\mu}$ at the level of all its $\mfl(\stau)$ arguments $t_\stau$ and $s^\stau$.

We also endow $\text{$\cbC$}[\text{{$\stau$}}]$ with the norm 
\begin{equs}   
\text{$\nnorm{W^{\text{\scriptsize{$\stau$}}}}_{N,\mu}$}&\eqdef\|\text{$K_{N,\mu}^{\otimes \text{\scriptsize{$\mfl$}}(\text{\scriptsize{$\stau$}})}$}W^{\text{\scriptsize{$\stau$}}}\|_{L^\infty_{t_{\tau_1}}L^1_{t_{\tau_2}}\cdots L^1_{t_{\tau_{p(\text{\tiny{$\stau$}})}}}((-\infty,1]^{p(\text{\tiny{$\stau$}})})
L^1_{s^{\text{\tiny{$\stau$}}}}(\R^{\mfs(\text{\tiny{$\stau$}})})}\\
&\equiv \sup_{t_{\tau_1}\leqslant1}
\int_{(-\infty,1]^{p(\text{\tiny{$\stau$}})-1}\times\R^{\text{\tiny{$\mfs$}}(\text{\tiny{$\stau$}})}}
| K_{N,\mu}^{\otimes \text{\scriptsize{$\mfl$}}(\text{\scriptsize{$\stau$}})}    \Kappa_{\eps,\mu}^{\text{\scriptsize{$\stau$}}}( t_{\text{\scriptsize{$\stau$}}},s^{\text{\scriptsize{$\stau$}}})|\text{$\rmd$} s^{\text{\scriptsize{$\stau$}}}\rmd t_{\tau_2}\cdots\rmd t_{\tau_{{p(\stau)}}}
\,.
\end{equs}
We write $\widecheck{\cbC}_{\mu}[\stau]$ for the closure of $\widecheck{\cbC}[\stau]$ with respect to $\nnorm{\bigcdot}_{N,\mu}$, again dropping the scale $\mu$ from notation when it is clear from context. 
\end{definition}

We now define the joint cumulant associated to $\text{$\stau$}\in\widehat\mcT$. 
\begin{definition}\label{def:joint_cumulant}
For $\eps,\mu\in(0,1]$, $\text{$\stau$}\in\widehat\mcT$ 
we define $\text{$\Kappa$}_{\eps,\mu}^{\stau} \in \widecheck{\cbC}[\stau]$
by setting, for any $(t_{\stau},s^\stau)\in(-\infty,1]^{p(\stau)}\times\R^{\mfs(\stau)}$, 
\begin{equ}
\Kappa_{\eps,\mu}^{\stau}( t_\stau,s^\stau)
\eqdef
\kappa_{p(\stau)}\Big(\d_\eps^{u_i}\d_\mu^{r_i}\xi_{\eps,\mu}^{\tau_i}(t_{i},s^{\tau_i}):{i\in[p(\stau)]}\Big)\,.
\end{equ} 
\end{definition} 

\begin{definition}
We say that $\stau \in \text{$   \widehat\mcT$}$ is \textit{relevant} if $|\stau|\leqslant0$, and that it is \textit{irrelevant} if $|\stau| > 0$. 
\end{definition}
The terms ``relevant'' and ``irrelevant'' are from quantum field theory -  relevant cumulants will have boundary conditions fixed at large scales and propagated backwards, while irrelevant cumulants will have (vanishing) boundary conditions fixed at small scales that are propagated to larger scales. 
\begin{definition}
   In the same way \eqref{eq:flow2} is simply a shorthand notation for \eqref{eq:flowLabelled} justified by the fact that \eqref{eq:flow2} does not depend on the choice of labelling, while in Definition~\ref{def:forests} we expressed the forests in terms of unlabelled trees, in order to write the flow equation for cumulants, we will actually use labelled trees, and give an analogue to \eqref{eq:flowLabelled}.

   Therefore, for $\underline{\stau}\in\stau\in\widehat \CT$, we denote by $\Kappa_\epmu^{\underline{\stau}}(t_{\underline{\stau}},s_{N(\underline{\stau})})$ the cumulant of the force coefficients $\big(\d_\mu^{r_i}\xi_{\eps,\mu}^{\btau_i}(t_{i},s_{N(\btau_i)}):{i\in[p(\stau)]}\big)$ indexed by the labelled trees composing $\underline{\stau}$.    
\end{definition}

We now present a hierarchy of equations for cumulants indexed by $\widehat\mcT$. 
This hierarchy is generated as follows: 
\begin{itemize}
\item For a given $\stau \in \widehat\mcT$ with $r(\stau) > 0$, we look at the first $\tau_{k}$ appearing in $\stau$ with $r_{k} = 1$ - which corresponds to a term $\partial_{\mu}\xi^{\tau_k}_{\epmu}$ in the cumulant - and replace this term with the corresponding right hand side of \eqref{eq:flow2}. 
This gives an expression like \eqref{eq:relcum} with $|J|=2$, this $J$-indexed term coming from the  quadratic grafting term from the right hand side of \eqref{eq:flow2}.  

\item For an irrelevant $\stau$ with $r(\stau) = 0$, we write $\kappa_{\epmu}^{\stau} = \int_{0}^{\mu} \partial_{\nu} \kappa_{\epnu}^{\stau}\rmd\nu$, rewrite $\partial_{\nu} \kappa_{\epnu}^{\stau}$ using Leibniz rule as a sum of cumulants each with a single derivative, and then use the first bullet point to rewrite these derivative terms. 

\item We will treat relevant  $\stau$ by hand, see Section~\ref{sec:3_2}.
\end{itemize}

In the first bullet point above, inserting \eqref{eq:flow2} will correspond to choosing a tree in an ordered forest, rewriting the corresponding force coefficient as grafted-product of two force coefficients, and then using \eqref{eq:relcum} to rewrite the result in terms of cumulants indexed by $\widehat{\mcT}$. 
The following definition introduces notation for describing this. 
 
\begin{definition}
For any $\stau \in\text{$\widehat{\mcT}$}$ such that $r(\stau)\geqslant1$, we set $k\equiv k(\stau)\eqdef\min\{i\in[p(\stau)]:r_i=1\}$. Moreover, to have notation for writing Leibniz rule for scale derivatives of cumulants, for $\stau = \big( (\tau_{i},0,u_i) \big)_{i=1}^{p(\stau)} \in \text{$\widehat{\mcT}$} $ (note that $r(\stau) = 0$) and $1 \leqslant k \leqslant p(\stau)$, we define \textcolor{white}{$\mathbf d$}
\begin{equs}
    \mathbf{d}_{k} \stau \eqdef \Big( (\tau_{i}, \mathbf{1}\{i = k\},u_i) : 1 \leqslant i \leqslant p(\stau) \Big) \in \widehat{\mcT}\,.
\end{equs}
Note that is holds $k(\mathbf{d}_k\stau)=k$.

Fix $\stau \in\widehat\mcT$ such that $r(\stau)\geqslant1$, and recall the definition of $\Ind(\tau_k)$ given in Definition~\ref{def:defBmu}. Moreover, fix $v_1,v_2\in\{0,1,2\}$ such that $v_1+v_2=u_k$. We combine $v_1,v_2$ and $(\sigma_1,\sigma_2)\in\Ind(\tau_k)$ in the forest $\ssigma = \big((\sigma_1,v_1),(\sigma_2,v_2) \big)$. 

We then define 
\begin{equs}
\text{$ {\mathrm{Ind}}$}(\text{${\stau}$})\eqdef
\Bigg\{ 
\big(\ssigma, \pi,\rho \big):
\begin{array}{c}
(\sigma_1,\sigma_2) \in \Ind(\tau_k)\,,\\
 v_1,v_2\in\{0,1,2\}:\,v_1+v_2=u_k\,,
\\
(\pi,\rho) \in \mcQ\big([p(\text{${\stau}$})]\setminus\{k\},[2]\big)
\end{array}
\Bigg\}\;.
\end{equs}
The set $\mathrm{Ind}(\stau)$ will appear as an indexing set for a sum in our cumulant equation. 
To motivate this notation, we note that upon fixing $\ssigma$, we are working with a cumulant of the form (suppressing time arguments) 

\[
\kappa_{p(\stau)}
\Big( \big(\xi^{\tau_i}_{\epmu}: i \in [p(\stau)] \setminus \{k\} \big), 
\xi^{\sigma_2} \graft^{\mu} \xi^{\sigma_1}
\Big)
\]
We then generate a sum over $(\pi,\rho)$ using \eqref{eq:relcum}, each term corresponding to a product of joint cumulants indexed by the $|\rho|$ different blocks of  $\rho$. 

We fix $(\ssigma,\pi,\rho)\in\mathrm{Ind}(\stau)$. For each $1 \leqslant q \leqslant |\rho|$, we define $\slambda_{q} \in \text{$\widehat{\mcT}$}$ via the concatenation of ordered forests $\slambda_q \eqdef {\stau}_{\pi_q}\sqcup{\ssigma}_{\rho_q}$, where we set  $\slambda_q=\ssigma_{\rho_q}$ if $p(\stau)=1$. 

Now, in order to define our analogue of the r.h.s. of \eqref{eq:flowLabelled} for the cumulant flow, we fix some representative $\underline\stau\in \stau$ verifying \eqref{eq:restrictionforest}. For $j\in[2]$, we also fix some representatives $\bsigma_j\in\sigma_j$, along with $f\in\Aut(\btau_k)$ and $n\in N(\bsigma_1)$ such that $\btau_k=\bsigma_2\graft_n\bsigma_1$.

Then, following Remark~\ref{rem:2_24}, we introduce the short-hand notations
 \begin{equs}       
    \ell_{\bsigma_1,\bsigma_2,j}\eqdef \begin{cases}  
  0 & \text{for} \;j= 1\,, \\
   \ell_{\bsigma_1,\bsigma_2}& \text{for} \;j= 2\,, \\
\end{cases}
           \quad\text{and}\quad
           t_{\bsigma_i} \equiv t_{\bsigma_i}(t_{\btau_k},w)\eqdef \begin{cases}  
    t_{\btau_k} & \text{for} \;j= 1\,, \\
      w & \text{for} \;j= 2\,. \\
\end{cases}
   \end{equs}
Recall that $f$ was promoted to a map
$\text{$\bff$}:\bigotimes_{j=1}^2\text{$\cbC$}[\bsigma_j+\ell_{\bsigma_1,\bsigma_2,j}]\longrightarrow\cbC[\btau_k]$. We further promote it to a map $$\text{$\bff^k$}:\bigotimes_{q=1}^{|\rho|}  \Big(\big( \bigotimes_{i \in \pi_{q}}{\cbC[\btau_i]}\big) \otimes \big( \bigotimes_{j \in \rho_q} \text{$\cbC$}[\bsigma_j+\ell_{\bsigma_1,\bsigma_2,j}] \big)\Big)\longrightarrow\text{$\cbC$}[\underline\stau]\eqdef\bigotimes_{i\in[p(\stau)]}\text{$\cbC$}[\btau_i]$$ acting as the identity on all factors $\cbC[\btau_i]$ for $i\neq  k$. Finally, following the definition of the map $\text{$\mathbf{M}$}^{\text{\scriptsize{$\bff$}}}_{\btau_1,\btau_2}$ in Remark~\ref{rem:2_24}, we define a map
\begin{equs}  \text{$\mathbf{M}$}^{\text{\scriptsize{$\bff,k$}}}_{\bsigma_1,\bsigma_2}\eqdef \text{$\mathbf{f}^k$} \circ\bigg( \bigotimes_{q=1}^{|\rho|}  \Big(\big( \bigotimes_{i \in \pi_{q}}\Id_{\cbC[\btau_i]}\big) \otimes \big( \bigotimes_{j \in \rho_q} \bell_{\bsigma_1,\bsigma_2,j} \big)\Big)\bigg):\bigotimes_{q=1}^{|\rho|}\text{$\cbC$}[\underline\slambda_q]  \longrightarrow\text{$\cbC$}[\underline\stau]\,.
\end{equs}
Here $\ell_{\bsigma_1,\bsigma_2,j}$ is promoted to a map $\bell_{\bsigma_1,\bsigma_2,j} $ following the set-up of Definition~\ref{def:iso_action}.

With this notation, our analogue of the summand in the last line of \eqref{eq:flowLabelled} for the cumulant flow equation will be given by the operator 
\begin{equs}    
\label{eq:defA} 
\text{A}^{\underline\stau, (\ssigma, \pi,\rho)}_\mu
\Big(&
\big(\Kappa_{\eps,\mu}^{\underline\slambda_q}\big)_{q=1}^{|\rho|}
\Big)
(t_{\underline\stau},s_{N(\underline\stau)})
\\
{}&  
\eqdef\frac{1}{|\text{$\Aut$}(\tau_k)|}  \frac{u_k!}{v_1!v_2!}\sum_{\substack{n\in N(\bsigma_1)\\
\bsigma_{1,2,n}=\btau_k}}
\sum_{ f \in \Aut(\btau_k)}
\int_{-\infty}^1\dot G_\mu( s_{n}-w)   
\text{$\mathbf{M}$}^{\text{\scriptsize{$\bff,k$}}}_{\bsigma_1,\bsigma_2}
\bigotimes_{q=1}^{|\rho|}\Kappa_{\eps,\mu}^{\underline\slambda_q}(t_{\underline\slambda_q},s_{N(\widetilde{\underline\slambda}_q)}) \,\text{$\rmd$} w\,.\textcolor{white}{bllabl}
\end{equs}
Here
\begin{equs}    \widetilde{\underline\slambda}_q\eqdef\stau_{\pi_q}\sqcup\widetilde{\underline\ssigma}_{\rho_q} \quad\text{where}\quad\widetilde{\underline\ssigma} \eqdef\big((\bsigma_j+\ell_{\bsigma_1,\bsigma_2,j},0):j\in[2]\big)\,.
\end{equs}
Finally, one can check as before that our definition of the maps $\A_\mu^{\underline\stau,(\ssigma,\pi,\rho)}$ on labelled spaces allows us to define a corresponding operator $\A_\mu^{\stau,(\ssigma,\pi,\rho)}$ in the unlablelled setting, namely by choosing some representatives and acting with $\A_\mu^{\underline\stau,(\ssigma,\pi,\rho)}$, in the same way the definition of the notation \eqref{eq:flow2} is \eqref{eq:flowLabelled}.

\end{definition}

With the above definitions, and based on the bullet points above the definition, we immediately have the following lemma. 
\begin{lemma}\label{lem:cumulant_flow}
Suppose that the force coefficients satisfy the flow equation \eqref{eq:flow2}, we then have the following equations for their joint cumulants as defined in Definition~\ref{def:joint_cumulant}. 

\begin{enumerate}
\item For $\text{$\stau$}\in\widehat\mcT$ such that $r(\text{$\stau$})\geqslant1$, and recall that $k
=
\min\{i\in[p(\stau)]:r_i=1\}$. We then have
\begin{equs}\label{eq:flowcumuldmu}
\text{$\Kappa$}_{\eps,\mu}^{\text{\scriptsize{$\stau$}}}=&-  \sum_{(\text{\scriptsize{$\ssigma$}},\pi,\rho)\in\text{\scriptsize{$\mathrm{Ind}$}}(\text{\scriptsize{$\stau$}})     }  
\mathrm{A}^{\text{\scriptsize{$\stau$}}, (\text{\scriptsize{$\ssigma$}},\pi,\rho)}_\mu
\Big(\big(\Kappa_{\eps,\mu}^{\text{\scriptsize{$\slambda$}}_q}\big)_{q=1}^{|\rho|}\Big)
       \,.
\end{equs}
\item Next, we turn to $\text{$\stau$}\in\text{$\widehat\mcT$}$ such that $r(\stau)=0$, and $|\stau|>0$. 
We then have 
\begin{equs}\label{eq:floweqcumulIrrel}
\text{$\Kappa$}_{\eps,\mu}^{\stau}=&-  
\sum_{k = 1}^{p(\stau)} 
\sum_{(\text{\scriptsize{$\ssigma$}},\pi,\rho)\in\text{\scriptsize{$\mathrm{Ind}$}}(\text{\scriptsize{$\mathbf{d}_{k}$}}\text{\scriptsize{$\stau$}})     } 
\int_0^\mu
\mathrm{A}^{\text{\scriptsize{$\mathbf{d}$}}_{k}\stau, (\ssigma,\pi,\rho)}_\nu \Big(\big(\Kappa_{\eps,\nu}^{\text{\scriptsize{$\slambda_q$}}}\big)_{q=1}^{|\rho|}
\Big)\rmd\nu
       \,.
\end{equs}
\end{enumerate}
\end{lemma}

Note that the flow equations of the cumulants are hierarchical in the order of the cumulants, in the sense that we have $\text{$\mfo(\slambda_q)\leqslant\mfo(\stau)-1$}$ for every $q\in[|\rho|]$.
Therefore, we can use these flow equations to formulate an inductive scheme for estimating cumulants. 
The following estimate is a straightforward consequence of \eqref{eq:defA}. 

\begin{lemma}\label{lem:boundA} 
For any partition $\rho\in\mcP([2])$, we introduce the short-hand notation $${\tt I}(\rho)\eqdef \begin{cases}  
    1 & \text{if} \;|\rho|= 1\,, \\
       \infty & \text{if} \;|\rho|= 2\,. \\
\end{cases} $$
Moreover, fix $N\geqslant1$, $\text{$\stau$}\in\widehat\mcT$, $k\in[p(\stau)]$, and $(\text{$\ssigma$},n,f\pi,\rho)\in\Ind(\stau)$. Then, for any collection $(\psi_q)_{q\in[|\rho|]}$ of functions such that, for every $q\in[|\rho|]$, $\psi_q\in\mcD'(\R^{p(\slambda_q)}\times\R^{\mfs(\slambda_q)})$,  
    \begin{equs}        \nnorm{\A^{\stau,(\ssigma,\pi,\rho)}_\mu\big((\psi_q)_{q\in[|\rho|]}\big)}_N\lesssim\norm{\mcP_{N,\mu} \dot G_\mu}_{\mcL^{{\tt I}(\rho),\infty}}\prod_{q\in[|\rho|]}\nnorm{\psi_q}_N
\end{equs}
uniformly in $\mu\in(0,1]$.
\end{lemma}
The main estimate on cumulants we will obtain is as follows.
\begin{lemma}\label{lem:cumul}
For every $\text{$\stau$}\in\widehat\mcT$, uniformly in $(\eps,\mu)\in(0,1]$, we have for $c>0$ small enough
    \begin{equs} \label{eq:Boundcum}\nnorm{\Kappa_{\eps,\mu}^{\stau}}_{2}\lesssim \eps^{u(\stau)(-1+c)}\mu^{|\stau|-u(\stau)c-r(\stau)}\,.
    \end{equs}
\end{lemma}
The flow equations \eqref{eq:flowcumuldmu} and \eqref{eq:floweqcumulIrrel} will allow us to deal with all the induction steps except the initialisation and the first step, both of which involve relevant cumulants. We postpone the study of the relevant cumulants to the next section but, admitting this input, we are ready to give the proof of Lemma~\ref{lem:cumul}.  
\begin{proof}[of Lemma~\ref{lem:cumul}]
We argue by induction on the order of $\stau$, the base case being handled by Lemma~\ref{lem:basecaseCum} below. 
We now deal with the induction step. 

In the first two induction steps, one can either construct an irrelevant cumulant, a cumulant with $r(\stau)\geqslant1$, or the expectation of $\xi_\epmu
^{\tau}$ for $\tau=\tOne,\tTwo$, or $\tThree$. By parity, the expectation of $\xi_\epmu
^{\tau}$ vanishes when $\tau=\tTwo,\tThree$, and in view of Lemma~\ref{lem:expDumbbell} below, we have for $\stau=\big((\tOne,0,u)\big)$
\begin{equs}
  \nnorm{\Kappa^\stau_\epmu}_2  \lesssim\eps^{u(-1+c)}\mu^{|\stau|-u(\stau)c-r(\stau)}\,.
\end{equs}
It remains to deal with case of cumulants with $r(\stau)\geqslant1$ or $|\stau|>0$. This is done using equations \eqref{eq:flowcumuldmu} and \eqref{eq:floweqcumulIrrel}, and the property of the operator $\A_\mu^{\stau,(\ssigma,\pi,\rho)}$ stated in Lemma~\ref{lem:boundA}. 
For example, using \eqref{eq:Kmunu}, we have in the case where $r(\stau)=0$
\begin{equs}
\nnorm{\Kappa_{\eps,\mu}^{\stau}}_{2}\lesssim  &\int_0^\mu\sum_{k=1}^{p(\stau)}\sum_{\substack{(\text{\scriptsize{$\ssigma$}},\pi,\rho)\in\text{\scriptsize{$\mathrm{Ind}$}}(\text{\scriptsize{$\mathbf{d}_k\stau$}})  } } \label{eq:inter4_15}\norm{\mcP_{2,\nu}\dot G_\nu}_{\mcL^{{\tt I}(\rho),\infty}}\prod_{q\in[|\rho|]}\nnorm{\Kappa^{\slambda_q}_\epnu}_{2} \rmd\nu\,.\textcolor{white}{bl}
\end{equs}
At this stage, the desired result now follows using the induction hypothesis and \eqref{eq:heat1}. We end up with 
\begin{equs}    \nnorm{\Kappa_{\eps,\mu}^{\stau}}_{2}&\lesssim \int_0^\mu\sum_{k=1}^{p(\stau)}\sum_{\substack{(\text{\scriptsize{$\ssigma$}},\pi,\rho)\in\text{\scriptsize{$\mathrm{Ind}$}}(\text{\scriptsize{$\mathbf{d}_k\stau$}})    } }\eps^{u(\stau)(-1+c)}\Big(\nu^{\sum_{j\neq k}|\tau_j|+|\sigma_1|+|\sigma_2|+(p(\stau)-1)-u(\stau)c}+\nu^{\sum_{j\neq k}|\tau_j|+|\sigma_1|+|\sigma_2|-1+p(\stau)-u(\stau)c}\Big)\rmd\nu\\
    &\lesssim   \eps^{u(\stau)(-1+c)}\int_0^\mu\nu^{|\stau|-1-u(\stau)c}\rmd\nu\lesssim\eps^{u(\stau)(-1+c)}\mu^{|\stau|-u(\stau)c}\,.
\end{equs}
In the last step, we have use the fact that $|\stau|>0$, so that the integral is convergent, which yields the desired result. Finally, note that in \eqref{eq:inter4_15}, we do not need to increase the value of $N$ in the kernel $K_{N,\mu}$, so that after the first induction step we can always take $N=2$.
\end{proof}
\subsection{Relevant cumulants}\label{sec:3_2}
Recall that the noise is contained in the effective force coefficient $\xi_\epmu^{\noise}(t,s)=\delta(t-s)\xi_\eps(s)$, so that $\xi_\epmu^{\noise}$ is constant along the flow, and its only non-vanishing cumulant is its covariance. 

Then, the relevant cumulants are the covariance of $\xi_\epmu^{\noise}$, along with the expectations of the relevant force coefficients $\xi^\tau_\epmu$ for $\tau=\bignoise$, $\tOne$, $\tTwo$, and $\tThree$. Moreover, by parity, the expectations for $\tau=\bignoise$, $\tTwo$ and $\tThree$ vanish, so that the only two non-trivial divergent cumulants are $$\kappa_2\big(\d_\eps^{u_1}\xi^{\noise}_\epmu(t_1,s_1),\d_\eps^{u_2}\xi^{\noise}_\epmu(t_2,s_2)\big) \quad\text{and}\quad\E[\d_\eps^{u}\xi^{\tOnesmall}_\epmu(t,s^{\tOnesmall}_1,s^{\tOnesmall}_2)]\,.$$
\subsubsection{The covariance of the noise}
We first deal with the covariance of $\xi_\epmu^{\noise}$. Since it does not flow, we simply have to handle it as an initialisation step. To do so, we need the following notation: for $u_1,u_2\in\{0,1,2\}$, we set
\begin{equs}
    \Cov^{u_1,u_2}_\eps(t-s)\eqdef \d_\eps^{u_1}\rho_\eps*\d_\eps^{u_2}\rho_\eps*\Cov(t-s)\,,
\end{equs}
and when $u_1=u_2=0$, we just write $\Cov_\eps\eqdef \Cov_\eps^{0,0}$.

The initialisation step is done in the following lemma.
\begin{lemma}\label{lem:basecaseCum}
For $\text{$\stau$}=\big((\bignoise,0,u_1),(\bignoise,0,u_2)\big)$ and $c>0$ small enough, uniformly in $\epmu\in(0,1]$, we have
   \begin{equs}   \nnorm{\Kappa_{\eps,\mu}^{\stau}}_{2}\lesssim \eps^{u(\stau)(-1+c)}\mu^{|\stau|-u(\stau)c-r(\stau)}\,.
    \end{equs}
\end{lemma}
\begin{proof}
First, note that
\begin{equs}    \Kappa^\stau_\epmu(t_1,t_2,s_1,s_2)=\delta(t_1-s_1)\delta(t_1-s_2)\Cov_\eps^{u_1,u_2}(s_1-s_2)\,,
\end{equs}
Therefore, using \eqref{eq:trick1} proven in Lemma~\ref{lem:tricks}, we have
\begin{equs} {}&\nnorm{\Kappa^{\stau}_\epmu}_{2}
\lesssim\norm{K_{4,\mu}*\Cov_\eps^{u_1,u_2}(r_1-r_2)
}_{L^\infty_{r_1}L^{1}_{r_2}}=
\norm{K_{4,\mu}*\Cov_\eps^{u_1,u_2}
}_{L^{1}}
\,.
\end{equs}
Then, we have
\begin{equs}    K_{4,\mu}\Cov_\eps^{u_1,u_2}=Q_\mu*\d^{u_1}_\eps\rho_\eps* Q_\mu*\d^{u_2}_\eps\rho_\eps*K_{2,\mu}\Cov\,,
\end{equs}
which implies that by Young's inequality for convolution, 
\begin{equs}   \nnorm{\Kappa^{\stau}_\epmu}_{2}
&\lesssim  \prod_{i=1}^2\norm{Q_\mu*\d^{u_i}_\eps\mcS_\eps\rho}_{L^1}\times \norm{K_{2,\mu}\Cov}_{L^1}\,.
\end{equs}
We first eliminate the product over $i$ using the bound \begin{equs}\label{eq:duch4_19}
\norm{Q_\mu*\d_\eps\mcS_\eps\rho}_{L^1}\lesssim\eps^{-1+c}\mu^{-c}
\end{equs}
proven in \cite[Lemma 4.19 (B)]{Duch21}, and if $u_i=2$ the fact that $\d_\eps^{2}\mcS_\eps\rho=\eps^{-1}\d_\eps\mcS_\eps\tilde \rho$ for some other smooth compactly supported function $\tilde\rho$. 

It remains to evaluate $\norm{K_{2,\mu}\Cov}_{L^1}$. We set $Q=Q_1$, and recall that the kernel of $Q_\mu$ verifies $Q_\mu(t)=\mcS_\mu Q(t)$. Here, we view $Q(t-s)$, $Q_\mu(t-s)$ as bilinear kernels, and therefore $Q_\mu(t-s)=\mu\mcS_\mu^{\otimes2}Q(t-s)$. With this observation, we have 
\begin{equs}
\norm{K_{2,\mu}\Cov}_{L^1}&= \mu^2\norm{\langle\mcS_\mu^{\otimes2}Q({r_1-\bigcdot})\otimes\mcS_\mu^{\otimes2}Q({r_2-\bigcdot}),\Cov\rangle}_{L^\infty_{r_1}L^1_{r_2}}\,.
\end{equs}
Here, in view of the homogeneity property of $\Cov$ stated in \eqref{eq:homogeneity}, we obtain that
\begin{equs}
    \norm{Q_\mu^{\otimes2}\Cov(r_1,r_2)}_{L^\infty_{r_1}L^1_{r_2}}&= \mu^{2H}\norm{\langle\mcS_\mu^{\otimes(1,0)}Q({r_1-\bigcdot})\otimes\mcS_\mu^{\otimes(1,0)}Q({r_2-\bigcdot}),\Cov\rangle}_{L^\infty_{r_1}L^1_{r_2}}\,,
\end{equs}
where we used the shorthand $\mcS_\mu^{\otimes(1,0)}=\mcS_\mu\otimes \Id$. Note that to compensate the fact that $\Cov$ is not in $L^1$, we had to take the value of $N$ in the kernel $K_{N,\mu}$ large enough for the pairing to exist: $N=2$ is sufficient since indeed $\Cov$ can be tested against $Q^{\otimes2}$. Recalling the definition \eqref{eq:defCov} of the covariance of the noise, we see that $Q$ can indeed absorb one time derivative and yields a kernel that behaves like a Dirac on short scales, which in turn can be tested against $\Cov^{\tt S}_W$ that is continuous in time. $\langle\mcS_\mu^{\otimes(1,0)}Q({r_1-\bigcdot})\otimes\mcS_\mu^{\otimes(1,0)}Q({r_2-\bigcdot}),\Cov\rangle$ is thus equal to $\mcS_\mu^{\otimes 2}f(r_1,r_2)$ where $f$ is a bona fine bounded function.

We can now conclude the proof, using \eqref{eq:boundSmu} that implies that the $L^1$ norm in $r_2$ is bounded by a constant while the $L^\infty$ norm in $r_1$ yields a factor $\mu^{-1}$, so that we end up with 
\begin{equs}    \label{eq:bound_cov}\norm{K_{2,\mu}^{\otimes2}\Cov_\eps^{u_1,u_2}(r_1,r_2)}_{L^\infty_{r_1}L^1_{r_2}}\lesssim\eps^{(u_1+u_2)(-1+c)}\mu^{-1+2H-(u_1+u_2)c}\,.
\end{equs}
\end{proof}
\begin{remark}
An argument similar to the derivation of  \eqref{eq:bound_cov} gives 
    \begin{equs}\label{eq:bound_cov2}
      \norm{K_{2,\mu}^{\otimes2}\Cov_\eps^{u_1,u_2}(r_1,r_2)}_{L^\infty_{r_1,r_2}} \lesssim\eps^{(u_1+u_2)(-1+c)}\mu^{-2+2H-(u_1+u_2)c}
    \end{equs}
uniformly in $\eps,\mu\in(0,1]$.
\end{remark}

\begin{remark}
Note that in the present article, we are focused on obtaining  probabilistic convergence as $\epsilon \downarrow 0$ of the effective solution at scales $\mu > 0$ and so we want to prove the probabilistic convergence of the force ansatz at scales $\mu > 0$. 
The covariance computation in the last lemma does this for the noise. 

However, a modification of the argument above would also allow one to compare forces constructed with different mollifiers, allowing one to also infer stability with respect to the choice of mollifier. 
\end{remark}

The following lemma is necessary in order to handle the Dirac distributions appearing in the definitions of the revelant cumulants.
\begin{lemma}\label{lem:tricks}
    For $\varphi\in C^\infty_c(\R^2)$ and $\psi\in C_c^\infty(\R)$, we define $A^\varphi (t_1,t_2,s_1,s_2)\eqdef\delta(t_1-s_1)\delta(t_2-s_2)\varphi(s_1,s_2)$, $B^\varphi(t_1,t_2,s_1)\eqdef\delta(t_1-s_1)\varphi(s_1,t_2)$, and $C^\psi(t_1,t_2,s_1)\eqdef\delta(t_1-s_1)\delta(t_2-s_1)\psi(s_1)$.

    Then, for every $N\geqslant1$, uniformly in $\varphi$, $\psi$ and $\mu\in(0,1]$, we have
    \begin{equs}
        \label{eq:trick1}\norm{K_{N,\mu}^{\otimes4}A^\varphi(t_1,t_2,s_1,s_2)}_{L^\infty_{t_1}L^1_{t_2}L^1_{s_1}L^1_{s_2}}&\lesssim \norm{K_{N,\mu}^{\otimes2}\varphi(s_1,s_2)}_{L^\infty_{s_1}L^1_{s_2}}\,,\\   \label{eq:trick2}   \norm{K_{N,\mu}^{\otimes3}B^\varphi(t_1,t_2,s_1)}_{L^\infty_{t_1}L^1_{t_2}L^1_{s_1}}&\lesssim \norm{K_{N,\mu}^{\otimes2}\varphi(s_1,s_2)}_{L^\infty_{s_1}L^1_{s_2}}\,,\\\norm{K_{N,\mu}^{\otimes3}C^\psi(t_1,t_2,s_1)}_{L^\infty_{t_1}L^1_{t_2}L^1_{s_1}}&\lesssim \norm{K_{N,\mu}\psi(s_1)}_{L^\infty_{s_1}}\,.\label{eq:trick3}
    \end{equs}
\end{lemma}
\begin{proof}
   We first prove \eqref{eq:trick1}. Using the identity $\mcP_{N,\mu} K_{N,\mu}=\Id$, we have
\begin{equs}    K_{N,\mu}^{\otimes4}A^\varphi(t_1,t_2,s_1,s_2)&=\int_{\R^2}  \prod_{i=1}^2 \big(K_{N,\mu}(s_i-r_i)  K_{N,\mu}(t_i-r_i)\big)\varphi(r_1,r_2)\rmd r_1\rmd r_2\\
&=\int_{\R^2} \prod_{i=1}^2\mcP_{N,\mu}^\dagger\big( K_{N,\mu}(s_i-\bigcdot)   K_{N,\mu}(t_i-\bigcdot)\big)(r_i)K_{N,\mu}^{\otimes2}\varphi(r_1,r_2)\rmd r_1\rmd r_2\,.
\end{equs}
By Hölder's inequality, we can conclude that
\begin{equs}  \label{eq:stepinproofcovnoise} {}&\norm{K_{N,\mu}^{\otimes4}A^\varphi(t_1,t_2,s_1,s_2)}_{L^\infty_{t_1}L^1_{t_2}L^1_{s_1}L^1_{s_2}}
\\&\lesssim
\norm{
\mcP^\dagger_{N,\mu}\big(K_{N,\mu}(s_1-\bigcdot)K_{2,\mu}(t_1-\bigcdot)\big)(r_1)
}_{L^\infty_{t_1}L^1_{s_1}L^1_{r_1}}\norm{\mcP^\dagger_{N,\mu}\big(K_{N,\mu}(s_2-\bigcdot) K_{N,\mu}(s_2-\bigcdot)\big)(r_2)
}_{L^{\infty}_{r_{2}}L^1_{t_2}L^1_{s_2}}
\\&\quad\times\norm{K^{\otimes2}_{N,\mu}\varphi(r_1,r_2)
}_{L^\infty_{r_1}L^{1}_{r_2}}\,.
\end{equs}
The first two terms of the r.h.s. of \eqref{eq:stepinproofcovnoise} are very similar. The aim is to show that they are bounded by a constant. Indeed, when the operator $\mcP^\dagger_{N,\mu}$ hits one of the kernels $K_{N,\mu}$, this possibly creates some time derivatives of the kernel $K_{N,\mu}$ multiplied by $\mu$. Then, by \eqref{eq:spacederivKmu}, the newly created kernel is also in $L^1$. Overall, this entails that there exists a finite set $I$ and some kernels $\big(R_\mu^{(i)},S_\mu^{(i)}:i\in I\big)$ belonging to $L^1$ uniformly in $\mu$ such that 
\begin{equs}
    \mcP_{N,\mu}^{\dagger}\big(K_{N,\mu}(s-\bigcdot)K_{N,\mu}(t-\bigcdot)\big)(r)
=\sum_{i\in I} R_\mu^{(i)}(s-r)S_\mu^{(i)}(t-r)\,.
\end{equs}
By translation invariance, this kernel really depends on two variables of the three variables $t$, $s$ and $r$, and by taking the $L^1$ norm in two of the three variables, we can indeed conclude that the first two terms are bounded uniformly in $\mu\in(0,1]$. 
    
    \eqref{eq:trick2} is an immediate consequence of \eqref{eq:trick1}. Indeed, $B^\varphi(t_1,t_2,s_1)=\int_\R A^\varphi(t_1,t_2,s_1,$ $s_2)\rmd s_2$, so that $K_{N,\mu}^{\otimes3}B^\varphi(t_1,t_2,s_1)=\int_\R K_{N,\mu}^{\otimes4}A^\varphi(t_1,t_2,s_1,$ $s_2)\rmd s_2$, and thus
    \begin{equs}        \norm{K_{N,\mu}^{\otimes3}B^\varphi(t_1,t_2,s_1)}_{L^\infty_{t_1}L^1_{t_2}L^1_{s_1}}\leqslant   \norm{K_{N,\mu}^{\otimes4}A^\varphi(t_1,t_2,s_1,s_2)}_{L^\infty_{t_1}L^1_{t_2}L^1_{s_1}L^1_{s_2}}\,.
    \end{equs}

\eqref{eq:trick3} stems from \eqref{eq:trick2}: applying \eqref{eq:trick2} to $\varphi(s_1,t_2)=\delta(t_2-s_1)\psi(s_1)  $ allows to get rid of the term $\delta(t_1-s_1)$. Applying \eqref{eq:trick2} a second time to $\varphi(s_1,t_2)=\psi(s_1)$ yields \eqref{eq:trick3}. Here, note that in the first step we took $\varphi$ to be a distribution. This is not a problem since $\delta(t_1-s_1)\delta(t_2-s_1)$ is a well-defined distribution, since the null section does not belong to the sum of the wavefront sets of the two Dirac distributions.    
\end{proof}
\subsubsection{The expectation of $\tOne$}\label{sec:expectation}
We now fix $u\in\{0,1,2\}$, and turn to the case of $\E[\d_\eps^u\xi^{\tOnesmall}_\epmu]$ which is the only relevant cumulant that has a non-trivial flow equation. 

Since we aim to show that solutions to \eqref{eq:eq1} are stable and converge as $\eps \downarrow 0$ \textit{without} any possibly divergent counterterm, we have to show, that despite being relevant, it is possible to construct $\big(\E[\xi^{\tOnesmall}_\epmu]\big)_{\epmu\in[0,1]}$ so that $\E[\xi^{\tOnesmall}_\epmu]$ vanishes when $\mu=0$. 

To do so, we do not use the general scheme of estimates we developed for the cumulant flow equations of Lemma~\ref{lem:cumulant_flow}. We instead study $\partial_{\nu} \E[\d_\eps^u\xi^{\tOnesmall}_\epnu]$ by taking the expectation of \eqref{eq:flow2} for $\tau = \tOne$, and we decompose $\E[\d_\eps^u\xi^{\tOnesmall}_\epmu]=\int_0^\mu\partial_{\nu} \E[\d_\eps^u\xi^{\tOnesmall}_\epnu]\rmd\nu$ into several terms, following the spirit of the decomposition performed in the proof of \cite[Theorem 5.3]{Duch23}.

Taking the expectation of \eqref{eq:flow2} for $\tau = \tOne$ first gives us 
\begin{equs}\label{eq:dumbell_flow}
\d_\nu \E \big[\d_\eps^u\xi^{\tOnesmall}_\epnu \big] &(t,s^{\tOnesmall}_1,s^{\tOnesmall}_2)\\
{}&=-\sum^{\noise}_{v_1,v_2}\delta(t-s_1^{\tOnesmall})\E[\d_\eps^{v_1}\xi_\eps(s_1^{\tOnesmall})\int_\R\dot G_\nu(s_1^{\tOnesmall}-w)\delta(w-s_2^{\tOnesmall})\d_\eps^{v_2}\xi_\eps(s_1^{\tOnesmall})\rmd w \big]
\\
&=-\sum^{\noise}_{v_1,v_2}\delta(t-s_1^{\tOnesmall})\E[\d_\eps^{v_1}\xi_\eps(s_1^{\tOnesmall})\dot G_\nu(t-s_2^{\tOnesmall})\d_\eps^{v_2}\xi_\eps(s_2^{\tOnesmall})]\\
   &=-\sum^{\noise}_{v_1,v_2}\delta(t-s_1^{\tOnesmall})\dot G_\mu(s_1^{\tOnesmall}-s_2^{\tOnesmall})\Cov^{v_1,v_2}_\eps
   (s_1^{\tOnesmall}-s_2^{\tOnesmall})
\,,
\end{equs}
where $s^{\tOnesmall}_1$ is the time variable associated to the root of $\tOne$, and we set
\begin{equs}
    \sum^{\noise}_{v_1,v_2}f(u,v_1,v_2)\eqdef  \sum_{\substack{v_1,v_2\in\{0,1,2\}\\v_1+v_2=u}}\frac{u!}{v_1!v_2!}f(u,v_1,v_2)\,.
\end{equs}
To further lighten the notion, we write $(s,r)$ instead of $ (s_1^{\tOnesmall},s_2^{\tOnesmall})$ and, given $v_1,v_2$ we write $\vec{v} =(v_1,v_2)$. Finally, we have 
\begin{equs}
    \d_\nu \E[\d^u_\eps\xi^{\tOnesmall}_\epnu](t,s,r)=-\delta(t-s)\sum_{v_1,v_2}^{\noise}\dot\lambda^{\tOnesmall,\vec{v}}_{\epnu}(s-r)\,,\quad\text{where}\quad\dot\lambda^{\tOnesmall,\vec{v}}_{\epnu}(t)\eqdef \dot G_\nu(t)\Cov^{v_1,v_2}_\eps
   (t)\,.
\end{equs} 

As is common in renormalisation, we will want to extract its ``local parts'', for which the following notation will be helpful. 
\begin{lemma}\label{def:local}
We define for any smooth function $f:\R^{2}\rightarrow\R$ 
\begin{equs}
   L^1(\R^2)\ni    \bfL_\tau f:(s,r)\mapsto \tau^{-1}f(s,s+(r-s)/\tau)\,.
\end{equs}
Then, setting $\X(s,r)\eqdef(s-r)$, the following equality holds true in a distributional sense:
\begin{equs}\label{eq:localisation}
f(s,r)=\delta (s-r) \int_\R f(s,w)\rmd w +\int_0^1\partial_{r}\text{$\bfL$}_\tau \big(\X f\big)(s,r)\rmd\tau\,.
\end{equs}
\end{lemma}



Applying Lemma~\ref{def:local}, we have 
\begin{equs}
      \d_\nu &\E[\d^u_\eps\xi^{\tOnesmall}_\epnu](t,s,r)=-\delta(t-s)\sum_{v_1,v_2}^{\noise}\dot\lambda^{\tOnesmall,\vec{v}}_{\epnu}(s-r)\\
      &=-\delta(t-s)\sum_{v_1,v_2}^{\noise}\Big(
      \big(\Id^{\otimes2}-K_{2,\nu}^{\otimes2}\big) \dot\lambda^{\tOnesmall,\vec{v}}_{\epnu}(s-r)+
      K_{2,\nu}^{\otimes2} \dot\lambda^{\tOnesmall,\vec{v}}_{\epnu}(s-r)           \Big)\\
      &=
-\delta(t-s)\sum_{v_1,v_2}^{\noise}\Big(
      \big(\Id^{\otimes2}-K_{2,\nu}^{\otimes2}\big) \dot\lambda^{\tOnesmall,\vec{v}}_{\epnu}(s-r)+
\int_0^1\partial_{r}\text{$\bfL$}_\tau \big(\X K_{2,\nu}^{\otimes2} \dot\lambda^{\tOnesmall,\vec{v}}_{\epnu} \big)(s-r)\rmd\tau 
      +\delta(s-r)\int_\R      K_{2,\nu}^{\otimes2} \dot\lambda^{\tOnesmall,\vec{v}}_{\epnu}(s-w)\rmd w           \Big) \\
      &=-\delta(t-s)\sum_{v_1,v_2}^{\noise}\Big(
      \big(\Id^{\otimes2}-K_{2,\nu}^{\otimes2}\big) \dot\lambda^{\tOnesmall,\vec{v}}_{\epnu}(s-r)+
\int_0^1\partial_{r}\text{$\bfL$}_\tau \big(\X K_{2,\nu}^{\otimes2} \dot\lambda^{\tOnesmall,\vec{v}}_{\epnu} \big)(s-r)\rmd\tau 
      \\&\qquad\qquad\qquad\qquad+\delta(s-r)\int_\R      \big(K_{2,\nu}^{\otimes2}-\Id^{\otimes2}\big) \dot\lambda^{\tOnesmall,\vec{v}}_{\epnu}(s-w)\rmd w        +\delta(s-r)\int_\R      \dot\lambda^{\tOnesmall,\vec{v}}_{\epnu}(s-w)\rmd w        \Big)
\end{equs}

Recall that we want to define $\E[\xi^{\tOnesmall}_\epmu]$ so that it goes to zero as $\mu$ goes to zero. Therefore, we integrate the previous equation between $0$ and $\mu$ to identify the five terms that contribute to $\E[\d_\eps^u\xi_\epmu^{\tOnesmall}]$, and set
\begin{equs}
\E[\d_\eps^u\xi_\epmu^{\tOnesmall,(\mathrm{I})}](t,s,r)&\eqdef-\delta(t-s)\sum_{v_1,v_2}^{\noise}\int_0^\mu\big(\Id^{\otimes2}-K_{2,\nu}^{\otimes2}\big) \dot\lambda^{\tOnesmall,\vec{v}}_{\epnu}(s-r)\rmd\nu\,,
\\
\E[\d_\eps^u\xi_\epmu^{\tOnesmall,(\mathrm{II})}](t,s,r)&\eqdef-\delta(t-s)\sum_{v_1,v_2}^{\noise}
\int_0^\mu\int_0^1\partial_{r}\text{$\bfL$}_\tau \big(\X K_{2,\nu}^{\otimes2} \dot\lambda^{\tOnesmall,\vec{v}}_{\epnu} \big)(s-r)\rmd\tau\rmd \nu\,,\\
\E[\d_\eps^u\xi_\epmu^{\tOnesmall,(\mathrm{III})}](t,s,r)&\eqdef-\delta(t-s)\delta(s-r)\sum_{v_1,v_2}^{\noise}\int_0^\mu
\int_\R \big( K_{2,\nu}^{\otimes2}-\Id^{\otimes2}\big) \dot\lambda^{\tOnesmall,\vec{v}}_{\epnu}(s-w) \rmd w\rmd\nu\,,\\
\E[\d_\eps^u\xi_\epmu^{\tOnesmall,(\mathrm{IV})}](t,s,r)&\eqdef\delta(t-s)\delta(s-r)\sum_{v_1,v_2}^{\noise}\int^1_\mu
\int_\R  \dot\lambda^{\tOnesmall,\vec{v}}_{\epnu}(s-w) \rmd w\rmd\nu\,.
\end{equs}
The above quantities verify 
\begin{equs}\label{eq:alltheterms}
    \E[\d^u_\eps\xi^{\tOnesmall}_\epnu](t,s,r)=\,&\E[\d_\eps^u\xi_\epmu^{\tOnesmall,(\mathrm{I})}](t,s,r)+\E[\d_\eps^u\xi_\epmu^{\tOnesmall,(\mathrm{II})}](t,s,r)+\E[\d_\eps^u\xi_\epmu^{\tOnesmall,(\mathrm{III})}](t,s,r)\\&+\E[\d_\eps^u\xi_\epmu^{\tOnesmall,(\mathrm{IV})}](t,s,r)+\delta(t-s)\delta(s-r)\d_\eps^u\mfc_\eps\,,
\end{equs}
where the \textit{counterterm} $\mfc_\eps$ is defined by 
\begin{equs}\label{eq:defCT}
\mfc_\eps\equiv\mfc_\eps(s)&\eqdef -\int_0^1
\int_\R  \dot\lambda^{\tOnesmall,\vec{v}}_{\epnu}(s-w) \rmd w\rmd\nu=-\int_0^1\int_\R \dot G_\nu(s-w)\Cov_\eps(s-w)\rmd w\rmd\nu\\&=\int_\R (G- G_1)(s-w)\Cov_\eps(s-w)\rmd w=(G-G_1)*\Cov_\eps(0)\,.
\end{equs}
Note thath by stationarity it does not depend on $s\in\R$.
\begin{remark}\label{rem:explanation_of_allterms}
We now discuss the five terms appearing in \eqref{eq:alltheterms}. 

\begin{itemize}
\item The first term is a remainder contribution coming from the fact that we want to convolve $\dot\lambda_\epnu^{\tOnesmall,\vec{v}}$ with $K^{\otimes 2}_{2,\nu}$ before localising. The convolution with $\big(\Id^{\otimes 2}-K^{\otimes 2}_{2,\nu}\big)$ yields a good factor that makes that the integral over $\mu$ is actually convergent -- see Lemma~\ref{lem:term1}.

\item The second term is a non-local remainder contribution resulting from the localisation. It comes with a linear time weight, and the good factor from the weight will allow us to show $\text{$\bfL$}_\tau \big(\X K_{2,\nu}^{\otimes2} \dot\lambda^{\tOnesmall,\vec{v}}_{\epnu} \big)$ is actually irrelevant -- see Lemma~\ref{lem:term2}.

\item  The third term is a remainder contribution coming from the fact that we do not want the kernel $K^{\otimes 2}_{2,\nu}$ to be present in the local part, in order to simplify the analysis of the counterterm in Appendix~\ref{sec:B}. As for the first term, the presence of $\big(\Id^{\otimes 2}-K^{\otimes 2}_{2,\nu}\big)$ yields a good factor ensuring that the integral over $\mu$ is convergent -- see Lemma~\ref{lem:term3}.

\item The fourth term is a local contribution with an integral in the scale $\nu$ being performed \textit{above} the effective scale $\mu$, which will thus be absolutely convergent -- see Lemma~\ref{lem:term4}. 

\auth{comment 1: this point was rewritten.}
\item The fifth term is the value of the local part of the expectation at $\mu=1$. This is the quantity\footnote{Modulo a choice of finite $\eps$-independent constant}, which, if it gave a divergent contribution as $\eps \downarrow 0$, would have to be subtracted in our equation - that is renormalised. This is why we call this term a counterterm. 

However, in the end we will show we do not need to renormalise and indeed show that this contribution is finite. By naive power counting, one might expect that $\mfc_\eps$ diverges like $\eps^{-1+2H}$. However, in Lemma~\ref{lem:CT}, we show that the counterterm is actually finite uniformly in $\eps>0$ when working with the stationary force. 
\end{itemize}    
\end{remark}

\begin{remark}\label{rem:explanation_of_last_term}
In Lemma~\ref{lem:CT} we also show that for the non-stationary force coefficients (the enhancement of the derivative of the fBM), this last term gives a time-dependent counterterm, which we denote $\mfc^+_\eps(t)$, and which would satisfy a point-wise in $t$ estimate $|\mfc^+_\eps(t)| \lesssim t^{-1+2H}$. 
In particular, the counterterm would produce an integrable divergence for small times.
\end{remark}

\begin{lemma}\label{lem:expDumbbell}
For $\text{$\stau$}=\big((\tOne,0,u)\big)$ and $c>0$ small enough, uniformly in $\epmu\in(0,1]$, we have
   \begin{equs}   \nnorm{\Kappa_{\eps,\mu}^{\stau}}_{2}\lesssim \eps^{u(\stau)(-1+c)}\mu^{|\stau|-u(\stau)c-r(\stau)}\,.
    \end{equs}
\end{lemma}
\begin{proof}
The first four terms in \eqref{eq:alltheterms} are handled in Appendix~\ref{sec:C} -- see Lemmas~\ref{lem:term1}, \ref{lem:term2}, \ref{lem:term3}, and \ref{lem:term4}. It remains to deal with the norm $\nnorm{\bigcdot}_2$ of $\delta(t-s)\delta(s-r)\d_\eps^u\mfc_\eps(s)$. Since this is a local contribution, we can use \eqref{eq:trick3} to eliminate the two Dirac distributions, so that this last term is bounded by
\begin{equs}
    \| K_{2,\mu}\d_\eps^u\mfc_\eps\|_{L^\infty}\,.
\end{equs}
The proof that the counterterm is finite uniformly is $\eps>0$ is given in Appendix~\ref{sec:B} (see Lemma~\ref{lem:CT}).
\end{proof}

\subsection{Kolmogorov argument}
With the cumulant analysis in hand, we are now ready to conclude the proof of Theorem~\ref{thm:sto}. 
This is done by post-processing the estimate \eqref{eq:Boundcum}, combined with the following Kolmogorov type argument (see~\cite[Lemma 13.6]{Duch22} for a similar statement).
\begin{lemma}\label{lem:Kolmo}
Pick $\tau\in\mcT$, and suppose that for some $\eta>0$ and $c>0$ small enough we have
\begin{equs}
\max_{u\in\{0,1,2\}}\max_{r\in\{0,1\}}\sup_{\mu\in(0,1]}  \E\Big[\Big( \eps^{u(1-c)}\mu^{-|\tau|+uc+r+\eta}\norm{\d_\mu^r\big(K^{\otimes1+\mfs(\tau)}_{4,\mu}(\omega\d_\eps^u\xi^{\tau}_{\eps,\mu})\big)}_{L^\infty_tL^1_{s^\tau}}\Big)^P\Big]^{1/P}\lesssim_{P,\tau} 1\,.
\end{equs}
Then, it holds
  \begin{equs}
\max_{u\in\{0,1\}} \E\Big[\Big(\sup_{\mu\in(0,1]}\eps^{u(1-c)} \mu^{-|\tau|+uc+r+2\eta}\norm{K^{\otimes1+\mfs(\tau)}_{4,\mu}(\omega\d_\eps^u\xi^{\tau}_{\eps,\mu})}_{L^\infty_tL^1_{s^\tau}}\Big)^P\Big]^{1/P}\lesssim_{P,\tau} 1\,.
\end{equs}
\end{lemma}
\begin{proof}[of Theorem~\ref{thm:sto}]
In view of Lemma~\ref{lem:Kolmo}, to prove the theorem, it suffices to show that for every $P$ even, $\tau\in\mcT$, $r\in\{0,1\}$ and $u\in\{0,1,2\}$, 
\begin{equs}\label{eq:preresult}
    \E[\|\d_\mu^r\big(K^{\otimes1+\mfs(\tau)}_{4,\mu}(\omega\d_\eps^u\xi^\tau_{\eps,\mu})\big)\|^P_{L^\infty_tL^1_{s^\tau}}]^{1/P}&\lesssim \eps^{u(-1+c)}\mu^{|\tau|-uc-r-1/P}\,. 
\end{equs}
Let us prove \eqref{eq:preresult}. Fix $P\geqslant1$ even, $\tau\in\mcT$ and $r\in\{0,1\}$, and recall the support property of the force coefficient $\xi_\epmu^\tau(t,s^\tau)$ stated in Lemma~\ref{lem:support} (the hypothesis of Lemma~\ref{lem:support} is indeed verified, since $\xi^{\noise}_\epmu(t,s)=\delta(t-s)\xi_\eps(s)$ is supported on the diagonal). It implies that we can apply \eqref{eq:expdecay} in all the variables $s^\tau$ to gain a good factor when bounding the $L^1$ norms in the variables $s^\tau$ by some $L^\infty$ norms. We thus have
\begin{equs}    \E[\|K^{\otimes1+\mfs(\tau)}_{4,\mu}(\omega\d_\eps^u\d_\mu^r\xi^\tau_{\eps,\mu})\|^P_{L^\infty_tL^1_{s^\tau}}]&\lesssim \mu^{P\mfs(\tau)}\E[\|K^{\otimes1+\mfs(\tau)}_{4,\mu}(\omega\d_\eps^u\d_\mu^r\xi^\tau_{\eps,\mu})\|^P_{L^\infty_{t,s^\tau}}]\\
&\lesssim\mu^{(P-1)\mfs(\tau)-1}\E[ \|K^{\otimes1+\mfs(\tau)}_{3,\mu}(\omega\d_\eps^u\d_\mu^r\xi^\tau_{\eps,\mu})\|^P_{L^P_{t,s^\tau}}]
\\&= \mu^{(P-1)\mfs(\tau)-1}   \int_{(-\infty,1]\times\R^{\mfs(\tau)}}\E[\big(K^{\otimes1+\text{\scriptsize{$\mfs$}}(\tau)}_{3,\mu}(\omega\d_\eps^u\d_\mu^r\xi^\tau_\epmu)(t,s^\tau)\big)^P]\rmd t\rmd s^\tau\,,
\end{equs}
where to go from the first to the second line, we used the Sobolev embedding type inequality \eqref{eq:KmuLp}. 
Remember that the kernel of $K_{3,\mu}$ is supported on positive times, and that $\omega$ has compact support on say $[a,b]$. Therefore, the integral over $t$ is really an integral over $[a,1]$, so that it can be bounded by a supremum. Moreover, again, using the support properties of the force coefficient, we can bound the integral in $s^\tau$ by a supremum using \eqref{eq:expdecay} and gain a good factor. Combining these two observations yields
\begin{equs}    
\E[\|K^{\otimes1+\mfs(\tau)}_{4,\mu} &(\omega\d_\eps^u\d_\mu^r\xi^\tau_{\eps,\mu})\|^P_{L^\infty_tL^1_{s^\tau}}]\\
{}&\lesssim 
\mu^{P\mfs(\tau)-1}   \|\E[\big(K^{\otimes1+\mfs(\tau)}_{3,\mu}(\omega\d_\eps^u\d_\mu^r\xi^\tau_\epmu)(t,s^\tau)\big)^P]\|_{L_{t,s^\tau}^\infty}\\
&\lesssim\mu^{P\mfs(\tau)-1}   \|\E[\prod_{i=1}^P K^{\otimes1+\mfs(\tau)}_{3,\mu}(\omega\d_\eps^u\d_\mu^r\xi^\tau_\epmu)(t_{\tau_i},s^{\tau_i})]\|_{L_{t_{\tau_1},s^{\tau_1}}^\infty\cdots L_{t_{\tau_P},s^{\tau_P}}^\infty}\,.    
\end{equs}
Here, we will use the same trick as in the proof of Lemma~\ref{lem:tricks} to eliminate the weights $\omega$. Indeed, we can write for any smooth function $\lambda:R\rightarrow\R$ it holds
\begin{equs}
 K_{3,\mu}(\omega\lambda)(t)=\int_\R K_{3,\mu}(t-s)\omega(s)\lambda(s)\rmd s =\int_\R\mcP^\dagger_{3,\mu}\big(K_{3,\mu}(t-\bigcdot)\omega(\bigcdot)\big)K_{3,\mu}\lambda(s)\rmd s\,.
\end{equs}
When the derivatives in $\mcP_{3,\mu}^\dagger$ act on $K_{\mu,3}(t-\bigcdot)$, this will create some new kernels that are still in $L^1$ uniformly is $\mu$. Therefore, there exist some kernels $\big( A^{(i)}:i\in\{0,1,2,3\}  \big)$ belonging to $L^1$ uniformly in $\mu$ such that 
\begin{equs}
 K_{3,\mu}(\omega\lambda)(t)=\int_\R\sum_{i=0}^3 A^{(i)}(t-s) \mu^i\d_s^i\omega(s)K_{3,\mu}\lambda(s)\rmd s\,.
\end{equs}
Applying this principle $P$ times yields
\begin{equs}
     \|\E[\prod_{i=1}^P K^{\otimes1+\mfs(\tau)}_{3,\mu}(\omega\d_\eps^u\d_\mu^r\xi^\tau_\epmu)(t_{\tau_i},s^{\tau_i})]\|_{L_{t_{\tau_1},s^{\tau_1}}^\infty\cdots L_{t_{\tau_P},s^{\tau_P}}^\infty}\lesssim
      \|\E[\prod_{i=1}^P K^{\otimes1+\mfs(\tau)}_{3,\mu}\d_\eps^u\d_\mu^r\xi^\tau_\epmu(t_{\tau_i},s^{\tau_i})]\|_{L_{t_{\tau_1},s^{\tau_1}}^\infty\cdots L_{t_{\tau_P},s^{\tau_P}}^\infty}\,.
\end{equs}
Therefore, setting $\tilde\tau\eqdef(\tau,r,u)\in\mcT_{1,2}$ and introducing for any $p\in[P]$ the list of trees 
\begin{equs}
    \tilde\tau^p\eqdef\underbrace{(\tilde\tau,\cdots,\tilde\tau)}_{p\,\text{\scriptsize{$\mathrm{times}$}}}\in\widehat\mcT\,,
\end{equs}
we end up with
\begin{equs}
\E[\|K^{\otimes1+\mfs(\tau)}_{4,\mu} &(\omega\d_\eps^u\d_\mu^r\xi^\tau_{\eps,\mu})\|^P_{L^\infty_tL^1_{s^\tau}}]
\lesssim\mu^{P\mfs(\tau)-1} \sum_{\pi \in\mcP([P])}\prod_{k=1}^{|\pi|}
\|K^{\otimes\mfl(\tilde\tau^{|\pi_k|})}_{3,\mu}\Kappa_\epmu^{\tilde\tau^{|\pi_k|}}(t_{\tilde\tau^{|\pi_k|}},s^{\tilde\tau^{|\pi_k|}})\|_{L^\infty(\R^{|\pi_k|}\times\R^{\mfs(\tilde\tau^{|\pi_k|})})}
\end{equs}
We are now ready to conclude, since using again the Sobolev embedding type inequality \eqref{eq:KmuLp}, we can bound the $L^\infty$ norms in $s^{\tilde\tau^{|\pi_k|}}$ and in all the variables in $t_{\tilde\tau^{|\pi_k|}}$ but one by some $L^1$ norms. We obtain
\begin{equs}
    \E[\|K^{\otimes1+\mfs(\tau)}_{4,\mu}(\omega\d_\mu^r\xi^\tau_{\eps,\mu})\|^P_{L^\infty_tL^1_{s^\tau}}]&\lesssim 
    \mu^{P\mfs(\tau)-1} \sum_{\pi \in\mcP([P])}\mu^{-P\mfs(\tau)-(P-|\pi|)}\prod_{k=1}^{|\pi|}
\nnorm{\Kappa_\epmu^{\tilde\tau^{|\pi_k|}}}_2\\
&\lesssim \eps^{Pu(-1+c)}\mu^{P|\tau|-Puc-Pr-1}\,,
\end{equs}
where on the second line we used \eqref{eq:Boundcum}. This is the desired result for $r=0$, and for $r=1$, we can also conclude using \eqref{eq:derivKmu} that allows us to plug the kernels $K_{4,\mu}$ inside the scale derivative.    
\end{proof}

\begin{appendix}

\section{Properties of regularizing kernels}\label{app:A}
We recall here some very important properties of the operators $K_{N,\mu}$ stated in \cite{Duch21}. 

We start with some general properties of the regularizing kernels.
\begin{lemma}
Let $N \geqslant 1$, then 
\begin{equs} \label{eq:normKmu}\norm{K_{N,\mu}}_{\mcL^{\infty,\infty}}&\leqslant1   \,,\\
\norm{\partial_t^{M}K_{N,\mu}}_{\mcL^{\infty,\infty}}&\lesssim\mu^{-M}\;\mathrm{for}\;\mathrm{all} \;M\leqslant N\,,
\label{eq:spacederivKmu}
\\
\label{eq:KmuLp} \norm{K_{N,\mu}}_{\mcL^{p,\infty}}&\lesssim \mu^{-1/p} \;\mathrm{for}\;\mathrm{all} \;p\geqslant1/ N\,,\\\norm{\mcP_{N,\mu}\d_\mu K_{N,\mu}}_{\mcL^{\infty,\infty}}&\lesssim\mu^{-1}\,,\label{eq:derivKmu} \end{equs}
uniformly in $\mu\in(0,1]$. 

Moreover, there exists $C_{N} < \infty$ such that for all $\mu,\nu\in(0,1]$, $\psi\in\mcD(\R)$ and , we have 
\begin{equs}  
\norm{K_{N,\mu}\psi}_{L^\infty}&\leqslant  \bigg(1\vee\Big(2\big(\nu/\mu\big)-1\Big)^N\bigg)  \norm{K_{N,\nu}\psi}_{L^\infty}\,,\label{eq:Kmunu}
\\
    \norm{K_{N,\mu}(\Id-K_{N,\nu})\psi}_{L^\infty}&\leqslant  N\big(\nu/\mu\big)\norm{ K_{N,\nu}\psi}_{L^\infty}\;\mathrm{for}\;\nu\leqslant\mu\,,\label{eq:comKmuKnu}
\end{equs}

\end{lemma}
\begin{proof}    
\eqref{eq:normKmu} is an immediate consequence of the definition of $K_{N,\mu}$. \eqref{eq:spacederivKmu} is \cite[Lemma~4.14 (A)]{Duch21}. \eqref{eq:KmuLp} is a minor modification of \cite[Lemma~4.14 (D)]{Duch21} (see also \cite[Lemma~10.52 (C)]{Duch21}). \eqref{eq:derivKmu} is \cite[Lemma~4.17]{Duch21}, \eqref{eq:Kmunu} follows from
\begin{equs}
    Q_\mu=\bigg(\frac{\nu}{\mu}+\Big(1-\frac{\nu}{\mu}\Big)Q_\mu\bigg)Q_\nu
\end{equs}
and \eqref{eq:comKmuKnu} was first observed in \cite{Duch23} in an elliptic context. 
The idea is that we have
\begin{equs}
Q_\mu(\Id-Q_\nu)&=(1+\mu\partial_{t})^{-1}\big(1-  (1+\nu\partial_{t})^{-1}\big)=\nu\partial_{t}  (1+\mu\partial_{t})^{-1}  (1+\nu\partial_{t})^{-1}\\
{}&=(\nu/\mu)   (1+\nu\partial_{t})^{-1}\big(1-  (1+\mu\partial_{t})^{-1}\big)=(\nu/\mu)Q_\nu(\Id-Q_\mu)\,,
\end{equs}
so that 
\begin{equs}
 K_{N,\mu}(\Id-K_{N,\nu}) & =Q_\mu^{*N}(\Id-Q_\nu^{*N})\\
 {}&=Q_\mu(\Id-Q_\nu)K_{N-1,\mu}\sum_{M=0}^{N-1}K_{M,\nu}=(\nu/\mu)(\Id-Q_\mu)K_{N-1,\mu}\sum_{M=1}^N K_{M,\nu}\,,
\end{equs}
and we can conclude, using \eqref{eq:Kmunu} to replace $K_{N-1,\mu}$ by $K_{N-M,\nu}$ for $M\in[N]$.
\end{proof}
The following lemma is a technical result regarding the kernels $K_{N,\mu}$. It turns out that, despite the fact that $K_{N,\mu}$ is not compactly supported, thanks to its exponential decay, the convolution $K_{N,\mu}\lambda$ inherits certain properties from a  compactly supported function $\lambda$. 
\begin{lemma}
   Fix $t_0\in[0,1]$, $N\in\N$, and consider $\lambda\in C^\infty(\R)$ with $\mathrm{supp} \,\lambda\subset[t_0-c\mu,t_0]$ for some universal constant $c>0$. Then, the following holds uniformly in $t_0\in[0,1]$ and $\mu\in[0,1]$:
   \begin{equs}
         \norm{K_{N,\mu}\lambda}_{L^p}&\lesssim_c  \mu^{1/p} \norm{K_{N,\mu}\lambda}_{L^\infty}\;\;\text{for every}\;p\in[1,\infty]\,,\label{eq:expdecay}\\         
       \norm{(t_0-\bigcdot)K_{N,\mu}\lambda}_{L^1}&\lesssim_c  \mu \norm{K_{N,\mu}\lambda}_{L^1}\,. \label{eq:expdecayPol}
   \end{equs}   
\end{lemma}
\begin{remark}
    The observation that the exponential decay of $K_{N,\mu}$ allows to transfer some support properties was made in \cite[Lemma 5.6]{Duch23}, whose statement corresponds to \eqref{eq:expdecayPol}. In the following, we adapt the proof of \cite[Lemma 5.6]{Duch23} to the case of \eqref{eq:expdecay} and, for the readers convenience, we also recall how \eqref{eq:expdecayPol} is proven.  Moreover, the proof of \eqref{eq:expdecayPol} generalises in the exact same way to the case of $\lambda\in C^\infty(\R^2)$ supported on $s,t$ such that $|t-s|\leqslant c\mu$, and in this case we have
\begin{equs}
           \norm{(t-s)K^{\otimes2}_{N,\mu}\lambda(t,s)}_{L^\infty_t L_s^1}&\lesssim_c  \mu \norm{K^{\otimes2}_{N,\mu}\lambda(t,s)}_{ L^\infty_t L_s^1}\,. \label{eq:expdecayPol1}
\end{equs}

\end{remark}
\begin{proof}
We first prove \eqref{eq:expdecay}. 
Let $\chi$ be as per Definition~\ref{def:G_mu}, and for $\tau\geqslant0$, set $\psi_\tau(t)\eqdef1-\chi(t/\tau)$, note $\psi_{\tau}$ is supported before $2\tau$ and is non-constant on $[\tau,2\tau]$ only. 
For $\tau\in[1,\infty)$, define the operator $\psi_{\tau\mu}K_{N,\mu}$ with kernel given by $\psi_{\tau\mu}(t)K_{N,\mu}(t)$, so that $K_{N,\mu}(t)=\lim_{\tau\uparrow\infty}\psi_{\tau\mu}K_{N,\mu}(t)$. 
Then, 
    \begin{equs}
           {}\norm{ K_{N,\mu}\lambda}_{L^p}\lesssim     \norm{ (\psi_\mu K_{N,\mu})\lambda}_{L^p}+\int_1^\infty   \norm{ \partial_\tau(\psi_{\tau\mu} K_{N,\mu})\lambda}_{L^p}\rmd \tau\,.
    \end{equs}
By the support properties of $\psi_\mu$, $(\psi_\mu K_{N,\mu})\lambda$ is supported on $[t_0-c\mu,t_0+2\mu]$, giving
\begin{equs}
     \norm{ (\psi_\mu K_{N,\mu})\lambda}_{L^p}&\lesssim\mu^{1/p}
      \norm{ (\psi_\mu K_{N,\mu})\lambda}_{L^\infty}=\mu^{1/p}
      \norm{ \mcP_\mu^N(\psi_\mu K_{N,\mu})*K_{N,\mu}\lambda}_{L^\infty}\\
      &\lesssim \mu^{1/p}\norm{\mcP_\mu^N(\psi_\mu K_{N,\mu})}_{\mcL^{\infty,\infty}}\norm{K_{N,\mu}\lambda}_{L^\infty}
\end{equs}
Similarly, $\partial_\tau(\psi_{\tau\mu} K_{N,\mu})\lambda$ is supported on $[t_0-c\mu+\tau\mu,t_0+2\tau\mu]$, so that
\begin{equs}
     \norm{ \partial_\tau(\psi_{\tau\mu} K_{N,\mu})\lambda}_{L^p}&\lesssim(\tau\mu)^{1/p}
      \norm{ \partial_\tau(\psi_{\tau\mu} K_{N,\mu})\lambda}_{L^\infty}=(\tau\mu)^{1/p}
      \norm{ \mcP_\mu^N\partial_\tau(\psi_{\tau\mu} K_{N,\mu})*K_{N,\mu}\lambda}_{L^\infty}\\
      &\lesssim (\tau\mu)^{1/p}\norm{\mcP_\mu^N\partial_\tau(\psi_{\tau\mu} K_{N,\mu})}_{\mcL^{\infty,\infty}}\norm{K_{N,\mu}\lambda}_{L^\infty}\,.
\end{equs}
Now, by definition,
\begin{equs}\label{eq:somestep1}
    \norm{\mcP_\mu^N(\psi_\mu K_{N,\mu})}_{\mcL^{\infty,\infty}}\lesssim1\,,
\end{equs}
while the exponential decay of $K_{N,\mu}$ combined with the fact that $\partial_\tau(\psi_{\tau\mu})$ is supported on $[\tau\mu,2\tau\mu]$ imply that
\begin{equs}\label{eq:somestep2}
    \norm{\mcP_\mu^N\partial_\tau(\psi_{\tau\mu} K_{N,\mu})}_{\mcL^{\infty,\infty}}\lesssim\tau^{-r}
\end{equs}
for every $r\geqslant1$. Combining all the previous estimates, one finally obtains
  \begin{equs}
           {}\norm{ K_{N,\mu}\lambda}_{L^1}\lesssim     \mu^{1/p}\Big(1+\int_1^\infty\tau^{-2}\rmd\tau\Big)\norm{ K_{N,\mu}\lambda}_{L^\infty}\,,
    \end{equs}
which gives \eqref{eq:expdecay}. 

We turn to \eqref{eq:expdecayPol}, and start with the estimate
    \begin{equs}
           {}\norm{ (t_0-\bigcdot)K_{N,\mu}\lambda}_{L^1}\lesssim     \norm{(t_0-\bigcdot) (\psi_\mu K_{N,\mu})\lambda}_{L^1}+\int_1^\infty   \norm{ (t_0-\bigcdot)\partial_\tau(\psi_{\tau\mu} K_{N,\mu})\lambda}_{L^1}\rmd \tau
    \end{equs}
The support properties of $\psi_\mu$, $(\psi_\mu K_{N,\mu})\lambda$ and $\partial_\tau(\psi_{\tau\mu} K_{N,\mu})\lambda$  then allow us to trade the weights $(t_0-\bigcdot)$ for some factors of $\tau$ and $\mu$, giving 
 \begin{equs}
           {}\norm{ (t_0-\bigcdot)K_{N,\mu}\lambda}_{L^1}&\lesssim   \mu  \norm{ (\psi_\mu K_{N,\mu})\lambda}_{L^1}+\mu\int_1^\infty  \tau \norm{ \partial_\tau(\psi_{\tau\mu} K_{N,\mu})\lambda}_{L^1}\rmd \tau\\
           &\lesssim\mu\Big(\norm{\mcP_\mu^N(\psi_\mu K_{N,\mu})}_{\mcL^{1,1}}+\int_1^\infty\tau
           \norm{\mcP_\mu^N\partial_\tau(\psi_{\tau\mu} K_{N,\mu})}_{\mcL^{1,1}}\Big)\norm{K_{N,\mu}\lambda}_{L^1}\,.
    \end{equs}
Using that, since the kernels are absolutely integrable, the $\mcL^{1,1}$ norm is the same as the $\mcL^{\infty,\infty}$ norm, along with the estimates \eqref{eq:somestep1} and \eqref{eq:somestep2}, then gives us the desired result.  
\end{proof}

\section{Estimation of the counterterms}\label{sec:B}
In this section, we estimate both the counterterms of the stationary and non-stationary force coefficients $\xi^{\tOnesmall}_\epmu$ and $\zeta^{\tOnesmall}_\epmu$. 
As mentioned in the last point of Remark~\ref{rem:explanation_of_allterms} and in Remark~\ref{rem:explanation_of_last_term}, we show that the stationary counterterm is indeed finite, while the non-stationary counterterm yields a finite Itô-Stratonovich correction of the form $\rmd t^{2H}$.

First, recall that in view of \eqref{eq:defCT}, the stationary counterterm is equal to 
\begin{equs}
\mfc_\eps(s)\equiv\mfc_\eps^{\tt S}(s)&=\rho_\eps^{\otimes2}*\big(\Id\otimes (G-G_1)\big)*\Cov^{\tt S}(s,s)=
\rho_\eps^{\otimes2}*\big(\Id\otimes (G-G_1)\big)*\d^{\otimes2}\Cov_W^{\tt S}(s,s)\,,
\end{equs}
which does not depend on $s\in\R$.

The counterterm for $\zeta^{\tOnesmall}_\epmu$, which for $s\in(0,1]$ we denote by $\mfc^+_\eps(s)$, is defined in the same way, but with $\Cov$ replaced by the non-stationary covariance $\Cov^+(t,s)=\d_t\d_s\Cov^+_W(t,s)$.

\begin{lemma}\label{lem:CT} 
   For every $u\in\{0,1,2\}$ and $c>0$ small enough, uniformly in $\eps\in(0,1]$ and $s\in\R$, we have
\begin{equs}
    |\d^u_\eps\mfc_\eps^{\tt S}(s)|\lesssim \eps^{(-1+c)u}\quad\text{and}\quad|\d^u_\eps\mfc^+_\eps(s)|\lesssim \eps^{(-1+c)u}(1\wedge |s|+\eps)^{-1+2H-cu}\,.
\end{equs}
\end{lemma}
\begin{proof}
The terms involving $G_1$ are necessarily bounded. Indeed, the kernel of $G_1$ is smooth and supported outside the diagonal, and the divergences of the covariances are supported on the diagonal, so that $(\Id\otimes G_1)*\Cov^{\tt S/+}$ is therefore a continuous function of its arguments.

We therefore focus on the terms with $G$. Because $G$ is the inverse of the derivative, they read
\begin{equs}
    \rho_\eps^{\otimes2}*(\Id\otimes G)*\d^{\otimes2}\Cov_W^{\tt S/+}(s,s)&=
     \rho_\eps^{\otimes2}*(\d\otimes\Id)*\Cov_W^{\tt S/+}(s,s)\\
&=\int_{\R^2}\rho_\eps(s-w_1)\rho_\eps(s-w_2)\d_{w_1}\Cov^{\tt S/+}_W(w_1,w_2)\rmd w_1\rmd w_2\\
    &=\int_{\R^2}\d_s\rho_\eps(s-w_1)\rho_\eps(s-w_2)\Cov^{\tt S/+}_W(w_1,w_2)\rmd w_1\rmd w_2\\
    &=\frac12\frac{\rmd}{\rmd s}\int_{\R^2}\rho_\eps(s-w_1)\rho_\eps(s-w_2)\Cov^{\tt S/+}_W(w_1,w_2)\rmd w_1\rmd w_2
    \,.
\end{equs}
In going from the second to the third line, we first integrated by part in $w_1$, and then used the relation $\d_{w_1}\rho_\eps(s-w_1)=-\d_{s}\rho_\eps(s-w_1)$, and to go to the last line we use the fact that $\Cov^{\tt S/+}_W$ is symmetric in its two arguments.
Note that is the case of $\Cov^{\tt S}_W$, by translation invariance, this expression is the derivative of a constant, and is thus $0$. 
This concludes the proof of the finiteness of $\mfc_\eps^{\tt S}(s)$, and we thus focus on $\mfc_\eps^{+}(s)$. Note that for this counterterm, we also control the derivatives in $\eps$, control on the $\eps$ derivatives of $\mfc_\eps^{\tt S}(s)$ can be obtained in the same way.

We have obtained
\begin{equs}
     {}& \rho_\eps^{\otimes2}*(\Id\otimes G)*\Cov^+(s,s)
    =\frac12\frac{\rmd}{\rmd s}\int_{\R^2}\rho(r_1)\rho(r_2)\Cov^+_W(s-\eps r_1,s-\eps r_2)\rmd r_1\rmd r_2\\
&=\frac{C_H}2\frac{\rmd}{\rmd s}\int_{\R^2}\rho(r_1)\rho(r_2)\1_{\geqslant0}(s-\eps r_1)\1_{\geqslant0}(s-\eps  r_2)\big(|s-\eps r_1|^{2H}+|s-\eps r_2|^{2H}-\eps^{2H}|r_1-r_2|^{2H}\big)\rmd r_1\rmd r_2    
    \,.
\end{equs}
Here, two situations can occur. Either the derivative in $s$ hits an indicator function, and turns it into a Dirac, or it hits the covariance term. In the first case, one obtains two contributions that are actually equal to zero. Indeed, they are of the form
\begin{equs}    {}&\int_{\R^2}\rho(r_1)\rho(r_2)\1_{\geqslant0}(s-\eps r_1)\delta(s-\eps  r_2)\big(|s-\eps r_1|^{2H}+|s-\eps r_2|^{2H}-|\eps r_1-\eps r_2|^{2H}\big)\rmd r_1\rmd r_2  \\
&\qquad=\int_{\R}\rho(r_1)\rho(s/\eps)\1_{\geqslant0}(s-\eps r_1)\big(|s-\eps r_1|^{2H}+|s-s|^{2H}-|\eps r_1-s|^{2H}\big)\rmd r_1=0\,.
\end{equs}
In the second case, observe that the third term in the covariance of the fractional Brownian motion is independent of $s$, and thus vanishes when the derivative in $s$ is taken. This yields
\begin{equs}
\rho_\eps^{\otimes2}*(&\Id\otimes G)*\Cov^+(s,s)\\
&=HC_H\int_{\R^2}\rho(r_1)\rho(r_2)\1_{\geqslant0}(s-\eps r_1)\1_{\geqslant0}(s-\eps  r_2)\big(|s-\eps r_1|^{-1+2H}+|s-\eps r_2|^{-1+2H}\big)\rmd r_1\rmd r_2  \\
    &=2HC_HP(s/\eps)\int_{-\infty}^{s/\eps}\rho(r)|s-\eps r|^{-1+2H}\rmd  r =2HC_H s^{-1+2H}P(s/\eps)\int_{-\infty}^{s/\eps}\rho(r)|1-\eps r/s|^{-1+2H}\rmd   r\\
    &=2HC_H s^{-1+2H} P(1/x)\rmI_{-\infty}^1(x)
    \,.
\end{equs}
Here, we set $x\equiv \eps/s$, denoted by $P(b)\eqdef\int_{-\infty}^b\rho(h)\rmd h$ the primitive of $\rho$, and introduced the integral
\begin{equs}
    \rmI_a^b(x)\eqdef \int_a^b \mcS_x\rho(h)|1-h|^{-1+2H}\rmd h=\int_{a/x}^{b/x}\rho(r)|1-xr|^{-1+2H}\rmd r\,.
\end{equs}
In order to conclude, it remains to bound the integral $\rmI_{-\infty}^1(x)$.\\
Moreover, remember that we want to keep track of the derivatives in $\eps$ of the counterterm. To do so, note that with the notation $x=\eps/s$, we have $\d_\eps= s^{-1}\frac{\rmd}{\rmd x}$. The derivative in $x$ can act on the integral $\rmI_{-\infty}^1(x)$, but also on $P(1/x)$. This term is easy to handle since 
\begin{equs}
    s^{-1}\frac{\rmd}{\rmd x} P(1/x)=-s^{-1}x^{-2}\rho(1/x)\,,
\end{equs}
and since $\rho$ is compactly supported, one has $|\rho(h)|\lesssim (1+|h|)^{-N}$ for every $N\geqslant0$, so that  
\begin{equs}
  \Big|  s^{-1}\frac{\rmd}{\rmd x} P(1/x)\Big|\lesssim s^{-1}x^{-2}(1+1/x)^{-2}\lesssim s^{-1}(1+x)^{-1}=(s+\eps)^{-1}\,.
\end{equs}

\noindent$\bullet\,$\underline{Case 1): $\eps\gg s$.} It corresponds to the regime $|x|$ large, and we observe that $\rmI_{-\infty}^1(x)$ decays at large $x$:
\begin{equs}
|\rmI_{-\infty}^1(x)|=\Big|\int_{-\infty}^{1/x}\rho(r)|1-xr|^{-1+2H}\rmd r\Big|=\Big|\int_0^\infty\rho(1/x-{q})|x{q}|^{-1+2H}\rmd {q}\Big|\lesssim |x|^{-1+2H}\,.
\end{equs}
In the same way, we obtain 
\begin{equs}
|s|^{-u}\Big|\frac{\rmd^u}{\rmd x^u}\rmI_{-\infty}^1(x)\Big|\lesssim |s|^{-u} |x|^{-1+2H-u}=\eps^{-u}|x|^{-1+2H}\lesssim( |s|+\eps)^{-1+2H}|x|^{-1+2H}\,.
\end{equs}

The case $\eps\ll s$, which corresponds to the regime $x$ small, is more subtle, and we need to further split $\rmI^1_{-\infty}(x)=\rmI^1_{1/2}(x)+\rmI^{1/2}_{-\infty}(x)$, and to study the two contributions separately.

\noindent$\bullet\,$\underline{Case 2a): $\eps\ll s$ and $h\geqslant1/2$.} The integral between $1/2$ and $1$ is easier to handle since, in this case, we can benefit from the fact that $\rho$ is compactly supported:
\begin{equs}
|\rmI^{1}_{1/2}(x)|=|x|^{-1}\Big|\int^{1}_{1/2}\rho(h/x)|1-h|^{-1+2H}\rmd h\Big|\lesssim |x|^{-1}\int_{1/2}^1\Big(1+\frac{1}{2|x|}\Big)^{-N}|1-h|^{-1+2H}\rmd {h}\,.    
\end{equs}
Setting $N=1$ yields $|\rmI^{1}_{1/2}(x)|\lesssim (1+|x|)^{-1}\lesssim1$.\\
Similarly, when acting with the derivative in $\eps$, we obtain for some other smooth compactly supported function $\tilde \rho$
\begin{equs}
|s|^{-u}\Big|\frac{\rmd^u}{\rmd x^u}\rmI_{1/2}^1(x)\Big|&= |s|^{-u} |x|^{-1-u}
\Big|\int^{1}_{1/2}\tilde\rho(h/x)|1-h|^{-1+2H}\rmd h\Big|\lesssim |s|^{-u}|x|^{-1-u}(1+1/|x|)^{-N}\\
&\lesssim |s|^{-u}(1+|x|)^{-1-u}\lesssim (|s|+\eps)^{-u}
\,,
\end{equs}
where on the second line we set $N=1+u$.

\noindent$\bullet\,$\underline{Case 2b): $\eps\ll s$ and $h\leqslant1/2$.} The integral between $-\infty$ and $1/2$ is the most difficult term. Indeed, here, we do not benefit from the decay of $\rho$. On the other hand, we are in the sector where the term $|1-h|^{-1+2H}$ is smooth. Therefore, we have
\begin{equs}
    |\rmI^{1/2}_{-\infty}(x)|=\Big|\int_{-\infty}^{1/2}\mcS_x\rho(h)|1-h|^{-1+2H}\rmd h\Big|\lesssim  \int_{-\infty}^{1/2}|\mcS_x\rho|(h)\rmd h \lesssim 1\,.
\end{equs}
However, such a naive approach breaks down when the derivative in $\eps$ is introduced, since using $\partial_x\mcS_x\rho=x^{-1}\mcS_x\tilde\rho $ for some other smooth compactly supported function $\tilde \rho$, we will never be able to eliminate the $\eps^{-1}$. Instead, observe that computing carefully $\tilde \rho$, one has $\tilde\rho=\theta^\prime$ for $\theta=-\Id\rho$. On the other hand, it holds
\begin{equs}
    \d_h \big(\mcS_x f\big)(h)=x^{-1}\big(\mcS_x f^\prime\big)(h)\,,
\end{equs}
so that we have obtained the relation 
\begin{equs}
    \d_x \mcS_x\rho(h)=\frac{\rmd}{\rmd h}\mcS_x\theta(h)\,.
\end{equs}
Now, taking advantage that we are integrating on $(-\infty,1/2]$ where $|1-h|^{-1+2H}$ is smooth, we can integrate by parts. This yields
\begin{equs}
    |s|^{-1}\Big|\frac{\rmd}{\rmd x}\rmI^{1/2}_{-\infty}(x)\Big|&=|s|^{-1}\Big|\int_{-\infty}^{1/2}\frac{\rmd}{\rmd h}\mcS_x\theta(h)|1-h|^{-1+2H}\rmd h\Big|
    \\&\lesssim |s|^{-1}|\mcS_x\theta(1/2)|+|s|^{-1}
    \Big|\int_{-\infty}^{1/2}\mcS_x\theta(h)|1-h|^{-2+2H}\rmd h\Big|\\
    &\lesssim |s|^{-1}|x|^{-1}(1+|x|)^{-N}+s^{-1}\int_{-\infty}^{1/2}|\mcS_x\eta|(h)\rmd h\\
    &\lesssim |s|^{-1}\lesssim (|s|+\eps)^{-1}\,.
\end{equs}
In the case $u=0$, we do not seek to optimise in the same way in the second derivative, and just obtain in a similar way
\begin{equs}
   | s|^{-2}\Big|\frac{\rmd^2}{\rmd x^2}\rmI^{1/2}_{-\infty}(x)\Big|&\lesssim \eps^{-1}(|s|+\eps)^{-1}\,,
\end{equs}
which is sufficient for our purpose.
\end{proof}
\section{Estimates on the expectation of \TitleEquation{\tOne}{tOne}}\label{sec:C}

\begin{lemma}\label{lem:term1}
  Uniformly in $\epmu\in(0,1]$, we have for $c>0$ small enough
    \begin{equs}        \norm{K_{2,\mu}^{\otimes3}\E[\d_\eps^u\xi_\epmu^{\tOnesmall,(\mathrm{I})}](t,s,r)}_{L^\infty_t L_s^1 L_{r}^1}\lesssim 
    \eps^{u(-1+c)}\mu^{-1+2H-uc}\,.
    \end{equs}
\end{lemma}
\begin{proof}
First, the Dirac distribution in the definition of $\E[\d_\eps^u\xi_\epmu^{\tOnesmall,(\mathrm{I})}]$ can be eliminated using \eqref{eq:trick2}, which yields

    \begin{equs}    \label{eq:stepfromlemterm3}    \norm{K_{2,\mu}^{\otimes3}\E[\d_\eps^u\xi_\epmu^{\tOnesmall,(\mathrm{I})}](t,s,r)}_{L^\infty_t L_s^1 L_{r}^1}&\lesssim
   \sum_{v_1,v_2}^{\noise} \int_0^\mu    \norm{K_{2,\mu}^{\otimes2}(\Id^{\otimes2}-K_{2,\nu}^{\otimes2})\dot\lambda_\epnu^{\tOnesmall,\vec{v}}(s-r)}_{L_s^\infty L_{r}^1}\rmd\nu\\  &\lesssim\mu^{-1}\sum_{v_1,v_2}^{\noise} \int_0^\mu\nu\norm{K_{2,\nu}^{\otimes2}\dot\lambda_\epnu^{\tOnesmall,\vec{v}}(s-r)}_{L_s^\infty L_{r}^1}\rmd\nu\,,
    \end{equs}
where in the second inequality we used \eqref{eq:comKmuKnu}.

Here, we repeat the argument of the proof of Lemma~\ref{lem:tricks} in order to move the operator $K_{2,\nu}^{\otimes2}$ to the covariance, and write 
\begin{equs}    K_{2,\nu}^{\otimes2}\dot\lambda_\epnu^{\tOnesmall,\vec{v}}(s-r)&=\int_{\R^2}K_{2,\nu}(s-w_1)K_{2,\nu}(r-w_2) \dot G_\mu(w_1-w_2)\Cov^{v_1,v_2}_\eps
   (w_1-w_2)\rmd w_1\rmd w_2\\   &=\int_{\R^2}\big(\mcP^\dagger_{2,\nu,w_1}\otimes\mcP^\dagger_{2,\nu,w_2}\big)\big(K_{2,\nu}(s-w_1)K_{2,\nu}(r-w_2) \dot G_\mu(w_1-w_2)\big)K^{\otimes2}_{2,\nu}\Cov^{v_1,v_2}_\eps
   (w_1-w_2)\rmd w_1\rmd w_2\,,
\end{equs}
where $\mcP^\dagger_{2,\nu,w_i}$ denotes the action of $\mcP^\dagger_{2,\nu}$ at the level of the variable $w_i$. Again, as in the proof of Lemma~\eqref{lem:tricks}, in view of \eqref{eq:heat1} and \eqref{eq:spacederivKmu}, the kernel $\big(\mcP^\dagger_{2,\nu,w_1}\otimes\mcP^\dagger_{2,\nu,w_2}\big)\big(K_{2,\nu}(s-w_1)K_{2,\nu}(r-w_2) \dot G_\mu(w_1-w_2)\big)$ is integrable uniformly in $\nu$. Sacrificing the integrals over $w_1,w_2$ and the $L^1$ norm in $r$ to integrate this kernel, we end up with 
\begin{equs}    \label{eq:astepinthelemmas}\norm{K_{2,\nu}^{\otimes2}\dot\lambda_\epnu^{\tOnesmall,\vec{v}}(s-r)}_{L_s^\infty L_{r}^1}\lesssim \norm{K_{2,\nu}^{\otimes2}\Cov_\eps^{v_1,v_2}(s-r)}_{L_{s,r}^\infty}\lesssim\eps^{u(-1+c)}\nu^{-2+2H-uc}\,,
\end{equs}
where on the second inequality we used \eqref{eq:bound_cov2}. Thus, finally, 
\begin{equs}    \norm{K_{2,\mu}^{\otimes3}\E[\d_\eps^u\xi_\epmu^{\tOnesmall,(\mathrm{I})}](t,s,r)}_{L^\infty_t L_s^1 L_{r}^1}&\lesssim\eps^{u(-1+c)}\mu^{-1}\int_0^\mu\nu^{-1+2H-uc}\rmd\nu\lesssim\eps^{u(-1+c)}\mu^{-1+2H-uc}\,,
\end{equs}
where to conclude we use the fact that the integral over $\nu$ is convergent.
\end{proof}
\begin{lemma}\label{lem:term2}
  Uniformly in $\epmu\in(0,1]$, we have for $c>0$ small enough
    \begin{equs}        \norm{K_{2,\mu}^{\otimes3}\E[\d_\eps^u\xi_\epmu^{\tOnesmall,(\mathrm{II})}](t,s,r)}_{L^\infty_t L_s^1 L_{r}^1}\lesssim   
    \eps^{u(-1+c)}\mu^{-1+2H-uc}\,.
    \end{equs}
\end{lemma}
\begin{proof}
    Again we start by getting rid of the Dirac distribution using \eqref{eq:trick2}, which yields
\begin{equs}    \norm{K_{2,\mu}^{\otimes3}\E[\d_\eps^u\xi_\epmu^{\tOnesmall,(\mathrm{II})}](t,s,r)}_{L^\infty_t L_s^1 L_{r}^1}&\lesssim\sum_{v_1,v_2}^{\noise} 
\int_0^\mu\int_0^1
\norm{
\big(K_{2,\mu}\otimes \partial_tK_{2,\mu}\big)
\text{$\bfL$}_\tau \big(\X K_{2,\nu}^{\otimes2} \dot\lambda^{\tOnesmall,\vec{v}}_{\epnu} \big)(s-r)}_{L_s^\infty L_{r}^1}\rmd\tau\rmd \nu\\
&\lesssim
\mu^{-1}\sum_{v_1,v_2}^{\noise} \int_0^\mu\int_0^1
\norm{
\text{$\bfL$}_\tau \big(\X K_{2,\nu}^{\otimes2} \dot\lambda^{\tOnesmall,\vec{v}}_{\epnu} \big)(s-r)}_{L_s^\infty L_{r}^1}\rmd\tau\rmd \nu\\
&\lesssim
\mu^{-1}\sum_{v_1,v_2}^{\noise} \int_0^\mu
\norm{
 \big(\X K_{2,\nu}^{\otimes2} \dot\lambda^{\tOnesmall,\vec{v}}_{\epnu} \big)(s-r)}_{L_s^\infty L_{r}^1}\rmd \nu\,.
\end{equs}
To go from the first to the second inequality, we used \eqref{eq:spacederivKmu}, and to go from the second to the third inequality, we use the fact that $\bfL_\tau$ is a rescaling that preserves the $L^1$ norm.

Then, using the support property \eqref{eq:suppGmu} of the kernel $\dot G_\nu$ and the estimate \eqref{eq:expdecayPol1}, that allows to benefit from the presence of the weight $\bfX$ despite the convolution with $K^{\otimes2}_{2,\nu}$, we end up with
\begin{equs}    \norm{K_{2,\mu}^{\otimes3}\E[\d_\eps^u\xi_\epmu^{\tOnesmall,(\mathrm{II})}](t,s,r)}_{L^\infty_t L_s^1 L_{r}^1}&\lesssim\sum_{v_1,v_2}^{\noise}
    \mu^{-1}\int_0^\mu \nu
\norm{
 K_{2,\nu}^{\otimes2} \dot\lambda^{\tOnesmall,\vec{v}}_{\epnu} (s-r)}_{L_s^\infty L_{r}^1}\rmd \nu\,.
 \end{equs}
 Using, \eqref{eq:astepinthelemmas}, we thus are in the same situation as in the previous lemma, which allows us to conclude.
\end{proof}

\begin{lemma}\label{lem:term3}
  Uniformly in $\epmu\in(0,1]$, we have for $c>0$ small enough
    \begin{equs}        \norm{K_{2,\mu}^{\otimes3}\E[\d_\eps^u\xi_\epmu^{\tOnesmall,(\mathrm{III})}](t,s,r)}_{L^\infty_t L_s^1 L_{r}^1}\lesssim   
    \eps^{u(-1+c)}\mu^{-1+2H-uc}\,.
    \end{equs}
\end{lemma}
\begin{proof}
This time, we deal with a local contribution, and get rid of the Dirac distributions using \eqref{eq:trick3}. This yields
\begin{equs}    \norm{K_{2,\mu}^{\otimes3}\E[\d_\eps^u\xi_\epmu^{\tOnesmall,(\mathrm{III})}](t,s,r)}_{L^\infty_t L_s^1 L_{r}^1}&\lesssim\sum_{v_1,v_2}^{\noise} 
\int_0^\mu
\Big\Vert
\int_\R K_{2,\mu}^{\otimes2} \big( K_{2,\nu}^{\otimes2}-\Id^{\otimes2}\big) \dot\lambda^{\tOnesmall,\vec{v}}_{\epnu}(s-w) \rmd w\Big\Vert_{L^\infty_ s}\rmd\nu\\
&\lesssim\sum_{v_1,v_2}^{\noise} 
\int_0^\mu
\norm{K_{2,\mu}^{\otimes2} \big( K_{2,\nu}^{\otimes2}-\Id^{\otimes2}\big) \dot\lambda^{\tOnesmall,\vec{v}}_{\epnu}(s-w) }_{L^\infty_ s L^1_w}\rmd\nu
\,.
\end{equs}
Here, note that in the first inequality, we used the freedom we have to introduce by hand a kernel $K_{2,\mu}$ inside the integral over $w$. We now are in the same situation as in Lemma~\ref{lem:term1} (see \eqref{eq:stepfromlemterm3}), so that the proof follows.
\end{proof}
\begin{lemma}\label{lem:term4}
  Uniformly in $\epmu\in(0,1]$, we have for $c>0$ small enough
    \begin{equs}        \norm{K_{2,\mu}^{\otimes3}\E[\d_\eps^u\xi_\epmu^{\tOnesmall,(\mathrm{IV})}](t,s,r)}_{L^\infty_t L_s^1 L_{r}^1}\lesssim   
    \eps^{u(-1+c)}\mu^{-1+2H-uc}\,.
    \end{equs}
\end{lemma}
\begin{proof}
Again, we use \eqref{eq:trick3} to eliminate the two Dirac distributions, which yields 
\begin{equs}    \norm{K_{2,\mu}^{\otimes3}\E[\d_\eps^u\xi_\epmu^{\tOnesmall,(\mathrm{IV})}](t,s,r)}_{L^\infty_t L_s^1 L_{r}^1}&\lesssim\sum_{v_1,v_2}^{\noise} 
\int_\mu^1
\Big\Vert
\int_\R K_{2,\mu}^{\otimes2}\dot\lambda^{\tOnesmall,\vec{v}}_{\epnu}(s-w) \rmd w\Big\Vert_{L^\infty_ s}\rmd\nu\\
&\lesssim\sum_{v_1,v_2}^{\noise} 
\int^1_\mu
\norm{K_{2,\mu}^{\otimes2} \dot\lambda^{\tOnesmall,\vec{v}}_{\epnu}(s-w) }_{L^\infty_ s L^1_w}\rmd\nu
\,.
\end{equs}  
As in the proof of the previous lemma, in the first inequality, we inserted by hand a kernel $K_{2,\mu}$ inside the integral over $w$. We are now ready to conclude, since using \eqref{eq:astepinthelemmas} we end up with 
\begin{equs}    \norm{K_{2,\mu}^{\otimes3}\E[\d_\eps^u\xi_\epmu^{\tOnesmall,(\mathrm{IV})}](t,s,r)}_{L^\infty_t L_s^1 L_{r}^1}&\lesssim\eps^{u(-1+c)}\int_\mu^1\nu^{-2+2H-uc}\rmd\nu\lesssim \eps^{u(-1+c)}\mu^{-1+2H-uc} \,.
\end{equs}
\end{proof}

\section{Construction of the non-stationary force coefficients}\label{Sec:Sec5}
This section is devoted to the proof of Theorem~\ref{thm:stoZeta}.
We construct the non-stationary force coefficients $(\zeta^\tau_\epmu)^{\tau\in\mcT}_{\epmu\in(0,1]}$ starting from the stationary force coefficients $(\xi^\tau_\epmu)^{\tau\in\mcT}_{\epmu\in(0,1]}$ using deterministic arguments, and with the input that the stationary force coefficients verify the estimate \eqref{eq:boundXi}, which was proven in Section~\ref{sec:proba}. 

Recall that we can express the stationary effective force $S_\epmu$ defined as
\begin{equs}
   S_{\epmu}[v](t)=\sum_{\tau\in \mcT} \int_{\R^{\mfs(\tau)}}\langle \xi_\epmu^{\tau}(t,s^{\tau}),\Upsilon^{\tau}[v+{\tt u}](s^{\tau})\rangle_{\mcH^{\tau}}\rmd s^{\tau}\,,
\end{equs}
and that it verifies $S_\epmu=\Pi S_\epmu$ and solves $\Pol_\mu(S_\epmu)=0$ with initial condition $S_\eps$.

We now aim to make sense of the solution $F_\epmu$ to $\Pol_{\mu}(F_\epmu)=0$ with initial condition $F_{\eps,0}=\1_{\geqslant0}S_\eps[\bigcdot]$. The strategy is to show that far from 0, $F_\epmu$ coincides with $S_\epmu$, and that it is therefore constructed with the same renormalisation, which ultimately turns out to be the same renormalisation necessary to make sense of $S_\epmu$.

Moreover, since for $\psi\in\mcD(\R)$, $F_\eps[\psi](t)$ is supported on positive times, one is only interested in constructing $F_\epmu[\psi](t)$ supported on positive times. We therefore make the ansatz that $F_\epmu$ is of the form
\begin{equs}\label{eq:Fmu_ansatz}
    F_{\epmu}[v](t)=\sum_{\tau\in \mcT} \int_{\R^{\mfs(\tau)}}\langle \zeta_\epmu^{\tau}(t,s^{\tau}),\Upsilon^{\tau}[v+{\tt u}](s^{\tau})\rangle_{\mcH^{\tau}}\rmd s^{\tau}\,,
\end{equs}
for $(\zeta^\tau_\epmu)^{\tau\in\mcT}_{\epmu\in(0,1]}$ a collection of force coefficients such that $\zeta^\tau_\epmu(t,s^\tau)=\1_{\geqslant0}(t)\zeta^\tau_\epmu(t,s^\tau)$.

The crucial observation in Section 10.4 of \cite{Duch21} is that $\xi^\tau_\epmu$ and $\zeta^\tau_\epmu$ coincide when evaluated sufficiently far (depending on $\mu$) from the the zero time plane in their first arguments. 
This is the content of the following lemma, analogous to \cite[Lemma~10.47]{Duch21}.  
\begin{lemma}\label{lem:supportzeta}
  Fix $\tau\in\mcT$, $\epmu\in(0,1]$ and $s^\tau\in\R^{\mfs(\tau)}$. For all $t\in (2\mu\mfo(\tau),1]$, it holds
    \begin{equs}
        \zeta^\tau_\epmu(t,s^\tau)=\xi^\tau_\epmu(t,s^\tau)\,.
    \end{equs}
\end{lemma}
This lemma stems from the support properties of $\dot G_\mu$ \eqref{eq:suppGmu}, and from the fact that the initial conditions $S_{\eps,0}$ and $F_{\eps,0}$ are equal after time 0.

Since $\zeta^\tau_\epmu(t,s^\tau)$ is now well-defined for $t > 2\mu\mfo(\tau)$, it remains to construct it for $t\in(0,2\mu\mfo(\tau)]$. 
Here we follow the proof of  \cite[Theorem~10.50]{Duch21} and leverage the fact that the region $(0,2\mu\mfo(\tau)]$ has Lebesgue measure of order $\mu $ which will result in a good factor allowing one to integrate the flow equation for $\zeta^\tau_\epmu(t,s^\tau)$, provided one works with some $L^1$ norms in time instead of the usual $L^\infty$ norms. \cite[Lemma~10.52]{Duch21} is necessary to leveraging the good factors coming from the use of the $L^1$ norm. We therefore translate it with our notation:
\begin{lemma}
 For any $N\in\N$, $p\in[1,\infty]$ and $\mu\in(0,1]$, we denote by $\mcW^p_{N,\mu}$ the set of all time weights $u_\mu\in \mcD(\R)$ verifying
\begin{equs}
  \max_{i\leqslant N} 
    \mu^{-1/p+i}
    \norm{\d^i_t u_\mu}_{L^p((-\infty,1])}< \infty\,.
\end{equs}
Then, for every $\psi\in \mcD(\R)$, $M\geqslant N$ and $q\in[1,\infty]$, uniform in $\mu\in(0,1]$ we have
    \begin{equs}\label{eq:1052a}      \norm{K_{M,\mu}\big(u_\mu\psi)}_{L^{q}}&\lesssim \norm{K_{M,\mu}\psi}_{L^{q}}\,,\;\text{provided} \; u_\mu \in\mcW^\infty_{N,\mu}\,,\\
  \label{eq:1052b}        \norm{K_{M,\mu}\big(u_\mu\psi)}_{L^{1}}&\lesssim\mu \norm{K_{M,\mu}\psi}_{L^\infty}\,,\;\text{provided} \; u_\mu \in\mcW^1_{N,\mu}\,.
    \end{equs} 
\end{lemma}

\begin{proof}[of Theorem~\ref{thm:stoZeta}]
 The proof is by induction on the order of the force coefficient. 
 
We start by discussing the initialisation. in view of \eqref{eq:trick2}, it suffices to control the non-stationary noise. Setting $Q=Q_1$, recall that the kernel of $Q_\mu$ verifies $Q_\mu(t)=\mcS_\mu Q(t)$. Here, we view $Q(t-s)$, $Q_\mu(t-s)$ as bilinear kernels, and therefore $Q_\mu(t-s)=\mu\mcS_\mu^{\otimes2}Q(t-s)$. With this observation, we have 
 \begin{equs}
     \E[Q_\mu(\1_{\geqslant0}\xi)^2(t)]=\mu^2 \langle\mcS_\mu^{\otimes2}Q({t-\bigcdot})\otimes\mcS_\mu^{\otimes2}Q({t-\bigcdot}),\Cov^+\rangle\,.  
 \end{equs}
Here, in view of the homogeneity property of $\Cov$ stated in \eqref{eq:homogeneity}, we obtain that
\begin{equs}
    \E[Q_\mu(\1_{\geqslant0}\xi)^2(t)]&= \mu^{2H}\langle\mcS_\mu^{\otimes(1,0)}Q({t-\bigcdot})\otimes\mcS_\mu^{\otimes(1,0)}Q({t-\bigcdot}),\Cov^+\rangle\,,
\end{equs}
where we used the shorthand $\mcS_\mu^{\otimes(1,0)}\eqdef\mcS_\mu\otimes \Id$. Note that $\Cov$ can be tested against $Q^{\otimes2}$: indeed, recalling the definition \eqref{eq:defCov} of the covariance of the noise, we see that $Q$ can indeed absorb one time derivative and yields a kernel that behaves like a Dirac on short scales, which in turn can be tested against $\Cov^+_W$ that is continuous in time. $\langle\mcS_\mu^{\otimes(1,0)}Q({t-\bigcdot})\otimes\mcS_\mu^{\otimes(1,0)}Q({t-\bigcdot}),\Cov\rangle$ is thus equal to $\mcS_\mu^{\otimes 2}f(t,t)$ where $f$ is a bona fine bounded continuous function. This quantity is therefore bounded by $\mu^{-2}$ uniformly in $t$. Therefore, by a classical Kolmogorov argument and hyper-contractivity, for $\eta>0$, we have
\begin{equs}
    \norm{Q_\mu(\1_{\geqslant0}\xi)}_{L^\infty([0,1])}\lesssim\mu^{-2+2H-\eta}\,,
\end{equs}
where the implicit constant is an $L^P$ random variable for every $P\geqslant1$. A similar computation would yield
\begin{equs}
    \norm{Q_\mu(\1_{\geqslant0}(1-\rho_\eps)\xi)}_{L^\infty([0,1])}\lesssim\eps^\eta\mu^{-2+2H-2\eta}\,.
\end{equs}

We now turn to the induction step. To lighten notation, we only prove the bound for $\Norm{\zeta}_{P,\eta}$, the proof of the convergence as $\eps\downarrow0$ being totally similar, and only requiring a heavier notation.

Two distinguish between the cases $t\geqslant 2\mu\mfo(\tau)$ and $t\geqslant 2\mu\mfo(\tau)$, we introduce two collections of time weights $v=(v_\mu)_{\mu\in(0,1]}$, $w=(w_\mu)_{\mu\in(0,1]}$ defined in such a way that we have $v_\mu\in\mcW^\infty_{N,\mu}$ and $ w_\mu\in\mcW^1_{N,\mu}\cap\mcW^\infty_{N,\mu}$ for all $N\in\N$ and that it holds 
\begin{equs}    \text{supp}\, v_\mu\subset[4\mu,\infty)\,,\;\text{and}\;\text{supp}\,\hat w_\mu\subset[-6\mu,6\mu]\,,
\end{equs}
as well as $ v_\mu(t)+ w_\mu(t)=1$ if $t\geqslant0$.

Using Lemma~\ref{lem:supportzeta}, for $t\geqslant0$, we thus have
\begin{equs} \label{eq:zetamu}   \zeta^\tau_\epmu(t,s^\tau)= v_\mu\xi_\epmu^\tau(t,s^\tau)+ w_\mu\zeta^\tau_\epmu(t,s^\tau)\,.
\end{equs}


We deal with the first term in \eqref{eq:zetamu} as follows: we want to control
\begin{equs}    \nnorm{ v_\mu\xi^\tau_\epmu}_{4+\mfo(\tau)}=
  \nnorm{ v_\mu\omega\xi^\tau_\epmu}_{4+\mfo(\tau)}\lesssim
    \nnorm{ v_\mu\omega\xi^\tau_\epmu}_{4}
\,,
\end{equs}
where we inserted a smooth function $\omega$ supported on $(-4\mu-1,2]$ and which is equal to one on $(-4\mu^2,1]$. Since $ v_\mu\in\mcW^\infty_{N,\mu}$, we therefore use \eqref{eq:1052a} and \eqref{eq:boundXi} to obtain
\begin{equs}    \nnorm{v_\mu\xi^\tau_\epmu}_{4+\mfo(\tau)}\lesssim\nnorm{\omega\xi^\tau_\epmu}_{4}\lesssim\mu^{|\tau|-\eta}\,.
\end{equs}
Let us turn to the second term in \eqref{eq:zetamu}. The proof is by induction on the order of $\tau$, assuming we already have
\begin{equs}\label{eq:rechypo}
    \nnorm{\zeta_\epmu^{\tau}}_{4+\mfo(\tau)}\lesssim\mu^{|\tau|-\eta}
\end{equs}
for all the trees of lower order. We rely on the following flow equation:
\begin{equs}
    \d_\mu\big( w_\mu\zeta^\tau_\epmu\big)(t,s^\tau)&= w_\mu\d_\mu\zeta^\tau_\epmu(t,s^\tau)+(\d_\mu \hat w_\mu)\zeta^\tau_\epmu(t,s^\tau)\\
    &= w_\mu\d_\mu\zeta^\tau_\epmu(t,s^\tau)-(\d_\mu  v_\mu)\xi^\tau_\epmu(t,s^\tau)\,.\label{eq:dmuwmuzeta}
\end{equs}
While we use \eqref{eq:boundXi} to control the second term in the r.h.s. of \eqref{eq:dmuwmuzeta}, we use the flow equation for the force coefficients to handle the first one. Indeed, taking the weight $ w_\mu$ into account, the flow equation rewrites
\begin{equs}\label{eq:flowwmuzetamu}
  w_\mu    \d_\mu\zeta_\epmu^{\tau}(t,s^\tau)&=-\sum_{\substack{(\tau_1,\tau_2)\in\text{\scriptsize{$\mathrm{Ind}(\tau)$}}}}    \big(\zeta_\epmu^{\tau_2}   \,\graft^\mu w_\mu\zeta_\epmu^{\tau_1}\big)^{\tau}(t,s^\tau)\,.
\end{equs}

Recall the notation $        K_{N,\mu}^{1+\mfs(\tau)}\mfz^\tau(t,s^\tau)=\big(K_{N,\mu}\otimes\dots\otimes K_{N,\mu}\big)*\mfz^\tau(t,s^\tau)$. For $N\in\N_{\geqslant1}$, we define a new norm $\nnorm{\bigcdot}_{1,N}$ inspired by $\nnorm{\bigcdot}_{N}\equiv\nnorm{\bigcdot}_{N,\mu}$ by setting
    \begin{equs}        \nnorm{\mfz^\tau}_{1,N}\eqdef \norm{K^{\otimes1+\mfs(\tau)}_{N,\mu}\mfz^\tau}_{L^{1}_{t}L^1_{s^\tau}((-\infty,1]\times\R^{\mfs(\tau)})}\equiv \int_{(-\infty,1]\times \R^{\mfs(\tau)}} |K^{\otimes1+\mfs(\tau)}_{N,\mu} \mfz^\tau(t,s^\tau)|\rmd t\rmd s^\tau\,.
    \end{equs}    
The interest of this norms lies in the fact that using \eqref{eq:dmuwmuzeta} and the Sobolev type estimate \eqref{eq:KmuLp} we  have 
\begin{equs}    \nnorm{w_\mu\zeta^\tau_\epmu}_{4+\mfo(\tau)}\lesssim\mu^{-1}\nnorm{ w_\mu\zeta^\tau_\epmu}_{1,3+\mfo(\tau)}\lesssim\mu^{-1}
\int_0^\mu\Big(\nnorm{w_\nu\d_\nu\zeta^\tau_\epnu}_{1,3+\mfo(\tau)}+\nnorm{\d_\nu  v_\nu\xi^\tau_\epnu}_{1,3+\mfo(\tau)} \Big)\rmd\nu\,.
\end{equs}
To control the second term, observe that $\d_\nu v_\nu(t)=\nu^{-1}z_\nu(t)$ with $z_\nu(t)=-(\Id v')(t/\nu)$, from which we deduce that $z_\nu\in\mcW^1_{N,\mu}$ for all $N\in\N$. Thus, \eqref{eq:1052b} and \eqref{eq:boundXi} imply that 
\begin{equs}
    \nnorm{\d_\nu v_\nu\xi^{\tau}_\epnu}_{1,3+\mfo(\tau)}\lesssim
    \nu^{-1}
     \nnorm{z_\nu\xi^{\tau}_\epnu}_{1,4}\lesssim  \nnorm{\omega\xi^{\tau}_\epnu}_{1,4}\lesssim\nu^{|\tau|-\eta}\,,
\end{equs}
where as before we inserted a smooth compactly supported function $\omega$.

To control the first term, we combine the flow equation \eqref{eq:flowwmuzetamu} and \eqref{eq:propB} with the fact that $\mfo(\tau)-1\geqslant \mfo(\tau_i)$ for $i=1,2$: this gives
\begin{equs}    \nnorm{w_\nu\d_\nu\zeta^\tau_\epnu}_{1,3+\mfo(\tau)} &\lesssim 
\sum_{\substack{(\tau_1,\tau_2)\in\text{\scriptsize{$\mathrm{Ind}(\tau)$}}}}   
    \norm{\mcP_{12,\nu}\dot G_\nu}_{\mcL^{\infty,\infty}}\nnorm{ w_\nu\zeta^{\tau_1}_\epnu}_{1,4+\mfo(\tau_1)}\nnorm{\zeta^{\tau_2}_\epnu}_{4+\mfo(\tau_2)}
    \,.
\end{equs}
Here, we can use the induction hypothesis and \eqref{eq:1052b} to control 
\begin{equs}
    \nnorm{ w_\nu\zeta^{\tau_1}_\epnu}_{1,4+\mfo(\tau_1)}\lesssim\nu
    \nnorm{\zeta^{\tau_1}_\epnu}_{4+\mfo(\tau_1)}\lesssim
   \nu^{1+|b|-\eta}\,,\quad\text{and}\quad      \nnorm{\zeta^{\tau_2}_\epnu}_{4+\mfo(\tau_2)}\lesssim
  \nu^{|c|-\eta}\,,
\end{equs}
which yields
\begin{equs}      \nnorm{w_\nu\d_\nu\zeta^\tau_\epnu}_{1,3+\mfo(\tau)}&\lesssim \sum_{\substack{(\tau_1,\tau_2)\in\text{\scriptsize{$\mathrm{Ind}(\tau)$}}}}  
   \nu^{1+|\tau_1|+|\tau_2|-2\eta}
 \lesssim\nu^{|\tau|-2\eta}
    \,.
\end{equs}
Because $|\tau|\geqslant-1+H$, the singularity at $\mu=0$ is now integrable. We thus end up with
\begin{equs}    \nnorm{w_\mu\zeta^\tau_\epmu}_{4+\mfo(\tau)}\lesssim\mu^{-1}
\int_0^\mu\nu^{|\tau|-\eta}\rmd\nu\lesssim\mu^{|\tau|-\eta}\,,
\end{equs}
which concludes the proof.
\end{proof}

\end{appendix}

\bibliographystyle{Martin}
\bibliography{refs.bib}

\end{document}